\documentclass[final,hidelinks,onefignum,onetabnum]{siamart251216}

\usepackage{lipsum}
\usepackage{amsfonts}
\usepackage{graphicx}
\usepackage{epstopdf}
\usepackage{algorithmic}
\ifpdf
  \DeclareGraphicsExtensions{.eps,.pdf,.png,.jpg}
\else
  \DeclareGraphicsExtensions{.eps}
\fi

\newsiamremark{remark}{Remark}
\newsiamremark{hypothesis}{Hypothesis}
\crefname{hypothesis}{Hypothesis}{Hypotheses}
\newsiamthm{claim}{Claim}
\newsiamremark{fact}{Fact}
\crefname{fact}{Fact}{Facts}

\headers{Kernel Learning of PDE Solution Operators}{Jianyu Hu and Juan-Pablo Ortega}

\title{Kernel Learning of PDE Solution Operators}
\author{
Jianyu Hu\thanks{Division of Mathematical Sciences, School of Physical and Mathematical Sciences, Nanyang Technological University, Singapore (\email{jianyu.hu@ntu.edu.sg}, \email{juan-pablo.ortega@ntu.edu.sg}). This work was funded by the School of Physical and Mathematical Sciences, Nanyang Technological University.}
\and
Juan-Pablo Ortega\footnotemark[1]
}

\usepackage{amsopn}

\usepackage[show]{ed}
\usepackage{enumitem}

\newcommand{\bfi}{\bfseries\itshape}

\usepackage{xr}

\begin{document}

\maketitle

\begin{abstract}
 A kernel-based approach for the learning of the solution operator of general nonhomogeneous partial differential equations (PDEs) is proposed. The method incorporates physical priors, typically encoded through the PDE operator, into a kernel ridge regression framework, and employs a regularization-based formulation to construct an operator learner. 
This yields a closed-form estimator that is independent of the input functions that determine the underlying PDE. From the perspective of regularization theory, the resulting estimator induces a well-defined operator that links input and output spaces, which contain the functions that define a Dirichlet problem and its solution, respectively. Consequently, it effectively shifts from a PDE solver to an operator-based solver. In contrast to standard supervised learning methods, it does not rely on paired input--output training data and enables systematic extrapolation beyond observed regimes. 

A full error analysis is conducted, providing convergence rates for the operator-based solver under suitable choices of regularization parameters. Extensive numerical experiments, including Darcy flow and Helmholtz equations, demonstrate that the proposed method achieves high accuracy and efficiency across a range of problem settings, and compares favorably with operator learning approaches in both approximation quality and computational cost.
\end{abstract}



\section{Introduction}

The mathematical models of many physical phenomena are formulated as initial and boundary value problems for partial differential equations (PDEs), and inverse problems governed by such equations arise naturally across nearly all areas of science and engineering. A central challenge in physics and science in general is therefore to develop methods capable of accurately solving a broad class of PDEs, and, crucially, at higher computational speed, since current approaches still largely rely on sophisticated numerical schemes tailored to specific problems. High computational speed is especially critical in engineering applications that require repeated evaluations of the solution operator, such as in design optimization.

In this paper, we are interested in developing machine learning methods for solving a broad class of linear PDEs. More specifically, consider a system whose state $u$ is the solution of a general inhomogeneous PDE \cite{hormander2007analysis} on a domain $D \subset \mathbb{R}^d$,
\begin{equation}\label{pde}
\begin{cases}
\mathcal{L}_s u = f, & \text{in } D,\\[0.3em]
\mathcal{B} u = g, & \text{on } \partial D.
\end{cases}
\end{equation}
Here, $D$ is not necessarily bounded, $\mathcal{L}_s$ denotes a differential linear operator of order $s$ (see \eqref{diff-operator}) and $\mathcal{B}$ encodes the boundary conditions (Dirichlet or Neumann problem), both of which are assumed to be known. 
\textit{We aim to learn the associated solution operator $\mathcal{T}:(f,g)\to u$, enabling fast evaluation of solutions for varying inputs while providing a theoretically grounded framework for operator approximation.}

\subsection{Literature review}

Classical numerical methods for PDEs~\eqref{pde}, such as the finite 
difference method~\cite{strikwerda2004finite,leveque2007finite}, finite 
element method~\cite{brenner2008mathematical,ciarlet2002finite}, and finite 
volume method~\cite{leveque2002finite}, discretize the domain $D$ onto a 
mesh of $N$ points and represent the forward model as a mapping 
$\mathbb{R}^N \to \mathbb{R}^N$, where the input vector encodes the source 
term $f$ and boundary data $g$ evaluated at those mesh points, and the output vector encodes the approximate solution $u$ at the same points. The differential operator $L_s$ is replaced by a large matrix, reducing the PDE to a finite-dimensional linear system. However, the approximation error depends sensitively on $N$, and achieving high accuracy requires $N$ 
to be large with the number of mesh points scaling as $N \sim h^{-d}$ for 
mesh spacing $h$ in $d$ spatial dimensions. This unfavorable scaling makes repeated evaluations of the forward model computationally prohibitive in engineering applications that require solving the PDE many times, such as in design optimization or uncertainty quantification.

Two major paradigms have emerged for approximating PDEs using scientific machine learning. 
The first one is based on an operator-learning \cite{boulle2024mathematical} perspective that seeks to estimate the mapping between function spaces that produces PDE solutions out of the input functions $f$ and $g$ that determine the PDE problem~\eqref{pde}. This approach uses input--output pairs as data, without explicitly encoding the underlying governing equations. Pioneering efforts along this line include DeepONet \cite{lu2021learning}, the Graph Neural Operator \cite{anandkumar2020neural}, and the Fourier Neural Operator \cite{kovachki2023neural}. Recent extensions further incorporate kernel-based approaches for learning Green functions \cite{stepaniants2023learning},  as well as random feature methods grounded in operator-valued reproducing kernel Hilbert spaces (RKHS)\cite{kadri2016operator,nelsen2021random,nelsen2024operator}. 
A common motivation underlying much of this literature is data-driven model discovery, rather than accelerating PDE solvers in settings where the governing physics is already well understood. In contrast, many practical engineering applications such as fluid dynamics and turbulent flow modeling are characterized by substantial prior physical knowledge. 
From this perspective, purely data-driven operator learning methods do not fully exploit the available structure; instead, they often treat the solution operator, namely the mapping from an input function space to an output function space, as a black box. 
To address this limitation, the physics-informed neural operator introduced in \cite{wang2021learning,li2024physics} incorporates physical constraints into the neural operator framework, thereby improving empirical accuracy.
Nevertheless, these approaches typically require large amounts of training data, and their performance is often highly sensitive to the choice of input sampling strategy. 
As a result, they rarely admit guarantees of stability or consistency, and thus cannot reliably support extrapolation beyond the regimes represented in the training data. 

The second direction adopts a more direct strategy by explicitly incorporating physical priors, such as the differential operator $\mathcal{L}_s$ and the boundary operator $\mathcal{B}$ in equation~\eqref{pde}, into the loss function. In this paradigm, the PDE solution field is first parameterized by a neural network, and the network parameters are then optimized by minimizing a loss functional based on the PDE residual, typically using variants of gradient descent or stochastic gradient descent.
Representative PDE solvers following this approach include 
artificial neural networks (ANNs) \cite{lagaris1998artificial,jianyu2003numerical},
physics-informed neural networks (PINNs) \cite{raissi2019physics,karniadakis2021physics,lu2021deepxde}, the deep Ritz method \cite{weinan2018deep,han2018solving}, the deep Galerkin method \cite{sirignano2018dgm}, and their variants \cite{karniadakis2021physics,cuomo2022scientific}. 
Although mesh-free compared to traditional methods, and unlike operator learning frameworks, these methods do not require paired data between infinite-dimensional input and output spaces, they remain highly problem-specific: any change to the initial-boundary-value problem formulation or PDE parameters necessitates expensive retraining of the neural network used to approximate the solution. This intrinsic limitation restricts the potential of neural networks as accelerators of traditional PDE solvers.

In contrast, kernel-based methods \cite{scholkopf2002learning,steinwart2008support} and Gaussian process regression \cite{williams2006gaussian,owhadi2019operator} offer a more straightforward, data-efficient, and easy to train alternative, and have proven to be both powerful and theoretically robust for modeling nonlinear phenomena. The classical kernel method approximates an unknown function by representing it as a linear combination of kernel sections at sampled data points within a reproducing kernel Hilbert space. 
When differential structures are incorporated, for example, in learning solutions of PDEs, the approach is often referred to as physics-informed kernel learning \cite{doumeche2024physics}. A common strategy is to model the solution in an appropriate Sobolev space and construct a corresponding RKHS equipped with a suitable kernel, thereby reformulating the problem within the classical kernel learning framework \cite{byun1994best,saitoh2004approximate,matsuura2004numerical}. This equivalence allows one to establish theoretical guarantees, such as convergence rates of the estimator toward the true PDE solution \cite{nickl2020convergence,doumeche2024physics}.
For a more comprehensive discussion, we refer to the classical book \cite{saitoh2016theory}.
This idea can also be extended to statistical learning with structured data, where observations are transformed through a general bounded linear operator \cite{schaback2006kernel}. In this framework, strong and weak minimax optimal convergence rates have been established for a broad class of spectral regularization methods over regularity classes defined via appropriate source conditions \cite{de2006discretization,blanchard2018optimal}.
Although this strategy enables rigorous analysis of error bounds, it is often inconvenient for practical computation, as it typically leads to kernels that are complicated or only implicitly defined.
For example, in \cite{doumeche2025physics}, the constructed kernel is characterized as the weak solution of an integro-differential PDE. Moreover, there is generally no guarantee that such structure-induced kernels are universal in the sense of \cite{micchelli2006universal}, and verifying such a property can be highly nontrivial. The universality of kernel functions plays a central role in the approximation abilities of the corresponding RKHS. The absence of universality may negatively affect both approximation accuracy and numerical performance, as discussed in \cite{nelsen2024operator}.

A more direct approach in practice is to model the PDE solution explicitly using a prescribed universal kernel, such as radial basis function \cite{franke1998solving,wendland2004scattered,fornberg2015solving} or Gaussian kernels \cite{graepel2003solving,owhadi2015bayesian,owhadi2017multigrid,swiler2020survey}. 
The advantage of this approach is that the kernel can be chosen conveniently for numerical applications, making the resulting computations straightforward to implement. These ideas have been extended to nonlinear and time-dependent PDEs \cite{raissi2018numerical,chen2021solving}. In particular, \cite{franke1998solving} derived Sobolev error estimates for generalized interpolation methods applied to linear differential operators with constant coefficients, which were later extended to the case of non-constant coefficients in \cite{giesl2007meshless}. Convergence of unsymmetric kernel-based meshless collocation methods for linear operator equations was analyzed in \cite{schaback2007convergence}, and stability conditions guaranteeing convergence of the resulting meshless collocation discretizations were subsequently established in \cite{ling2008stable}. Optimal $H^2(\Omega)$ convergence rates for least-squares formulations of the Kansa kernel collocation method for second-order elliptic PDEs were obtained in \cite{cheung2018h}. A general framework for stable discretizations of well-posed linear operator equations was later developed in \cite{schaback2016all}, showing that convergence rates are determined by the approximation properties of the trial spaces. 
More recently, \cite{batlle2025error} established Sobolev-space error estimates for Gaussian process and kernel-based methods applied to nonlinear and parametric PDEs, highlighting the role of solution regularity in mitigating the curse of dimensionality. For scattered data defined on embedded submanifolds, \cite{fuselier2012scattered} derived Sobolev error estimates for kernel interpolation obtained by restricting positive definite kernels from the ambient space.

In this paper, we adopt the RKHS framework and model the PDE solution in an RKHS. We seek the optimal element in this space by minimizing a physics-informed loss that incorporates the differential operator $\mathcal{L}_s$ and the boundary operator $\mathcal{B}$ from equation~\eqref{pde}, both of which are assumed to be known. We shall show that even in the presence of these physical priors in the loss function, the kernel-based solution for the operator learning problem can be obtained out of a Gramian-type linear regression similar to the one obtained out of the standard representer theorem. More specifically, inspired by recent developments in the learning of interaction potentials \cite{lu2019nonparametric,feng2024learning} and Hamiltonian functions \cite{RCSP2, RCSP3, hu2025kernel}, we provide an operator-theoretic framework for the learning problem that admits a closed-form expression for its unique solution, which coincides with the posterior mean estimator in physics-informed Gaussian process regression \cite{raissi2017machine,kanagawa2018gaussian}.
Importantly, the resulting estimator induces a kernel basis (see Section~\ref{From PDE Solvers to Solution Operators}) once the kernel and the sampling measure on $\overline{D}$ are specified. Consequently, the algorithm does not require retraining when the input functions $f$ and $g$ change, thereby enabling fast computation. This basis can be interpreted as an approximation of the Green operator associated with the PDE~\eqref{pde}. We emphasize that this approximation is obtained solely from samples in the finite-dimensional underlying spaces, and does not require paired samples of infinite-dimensional input--output functions as, for instance, in \cite{stepaniants2023learning}.
This approach is closely related to regularization methods \cite{tikhonov1963solution,tikhonov1977solutions,engl1996regularization,benning2018modern} that provide a regularized and computationally tractable formulation. Finally, we establish comprehensive error bounds for the reconstruction of the solution operator, which distinguishes our analysis from existing kernel-based approaches that primarily focus on approximating a single PDE solution. The code for all experiments is available at \url{https://github.com/jianyuhu/kernel-operator-pde}.

\subsection{Main results}

This paper proposes a kernel-based method for learning the solution operator $\mathcal{T}:(f,g)\mapsto u$ associated with the linear PDE~\eqref{pde}.  A key feature of this approach is that the solution of the learning problem is unique, admits a closed-form, and is obtained by solving a Gram-type linear regression problem. The method is based on data obtained from finite-dimensional function evaluations in the underlying domain using a prescribed universal kernel and directly incorporates the differential operators that define~\eqref{pde}. We emphasize that, unlike existing operator learning approaches no infinite-dimensional paired input--output functional training data is used. 

For notational simplicity, we restrict attention to Dirichlet boundary conditions, so that $\mathcal{B}=I$, and reformulate~\eqref{pde} as:
\begin{align}
\label{reformuted Ps}
P _s u=h, \quad\text{where}\quad P_s := \mathcal{L}_s \mathbf{1}_{ D} + I\mathbf{1}_{\partial  D}
\quad\text{and}\quad
h :=
\begin{cases}
f & \text{in }  D,\\
g & \text{on } \partial  D .
\end{cases}
\end{align}
Let $\mu$ be a sampling measure on $\overline{D}$, and let $\{X^{(i)}\}_{i=1}^N$ be i.i.d.\ samples drawn from $\mu$. 
Since the boundary $\partial D$ is a Borel set of measure zero, in practice (see Section~\ref{Numerical Experiments}) we sample points in the interior $D$ and on the boundary $\partial D$ separately (using two different measures). 
Within this framework, we approximate the solutions $u$ of the PDEs \eqref{pde} by solving the following regularized empirical risk minimization problem
\begin{align}
\widehat{u}_{\lambda,N}
&:= \arg\min_{u\in \mathcal{H}_K} \ \widehat{R}_{\lambda,N}(u), \label{emp-pro0}\\
\widehat{R}_{\lambda,N}(u)
&:= \frac{1}{N} \sum_{i=1}^N \left|P_s u(X^{(i)}) - h(X^{(i)})\right|^2 
 + \lambda \|u\|_{\mathcal{H}_K}^2, \label{emp-fun0}
\end{align}
where $\lambda>0$ is a Tikhonov regularization parameter and $\mathcal{H}_K$ is the reproducing kernel Hilbert space (RKHS) associated with a kernel 
$K:\overline{D}\times\overline{D}\to\mathbb{R}$. We now summarize the paper's outline and main contributions.
\begin{enumerate}[leftmargin=*]
\item  In Section~\ref{Methodology}, we develop a physics-informed kernel framework for learning the solution operator associated with the linear PDE~\eqref{pde}. 
In Subsection~\ref{Problem Formulation}, we reformulate the boundary value problem as a unified operator equation and cast it into a statistical learning framework, emphasizing the objective of learning the operator $\mathcal{T}:h\mapsto u$ rather than individual solutions. 
In Subsection~\ref{A Physics-informed Kernel Approach}, we exploit the differential reproducing property \cite{RCSP2,zhou2008derivative} to show that if the kernel function $K\in C_{b}^{2s+1}(\overline{D}\times\overline{D})$,then the operator $P_s : \mathcal{H}_K \to L^2(\mu)$ and its empirical counterpart $P_{s,N} : \mathcal{H}_K \to \mathbb{R}^N$ are bounded linear operators for any sampling measure $\mu$. This leads to closed-form operator expressions for both empirical and statistical minimizers. Furthermore, we establish a generalized kernel representer theorem that provides an explicit finite-dimensional characterization of the estimator via a generalized Gram regression. 
In Subsection~\ref{From PDE Solvers to Solution Operators}, we use the estimator $\widehat{u}_{\lambda,N}$ to construct a solution operator $\widehat{\mathcal{T}}_{\lambda,N}: L^2(\overline{D};\mu) \to \mathcal{H}_K$ that satisfies $\widehat{u}_{\lambda,N} = \widehat{\mathcal{T}}_{\lambda,N} h$. Importantly, $\widehat{\mathcal{T}}_{\lambda,N}$ admits an efficient practical implementation that enables fast evaluation for arbitrary inputs and reveals a connection to Green functions. 
Finally, we address scalability by introducing an online kernel regression scheme that supports efficient updates in large-scale settings.

\item In Section~\ref{Estimation and approximation error bounds}, we establish rigorous error bounds for the empirical operator estimator $\widehat{\mathcal{T}}_{\lambda,N}$ by decomposing the total reconstruction error into estimation and approximation components. We first derive high-probability bounds (Proposition~\ref{Sam-Err}) for the estimation error, showing that $\widehat{\mathcal{T}}_{\lambda,N}$ converges to the statistical operator $\mathcal{T}_{\lambda}$ (see~\eqref{statistical solution operator estimator}) at a rate governed by the sample size $N$ and the regularization parameter $\lambda$. These results are obtained in a pathwise sense, i.e., for each fixed input function $h$.
In Subsection~\ref{Uniform Convergence of Total Reconstruction Error}, we lift these pathwise guarantees to uniform convergence over function classes $\mathcal{F}_S^\gamma$ (see~\eqref{source space}). To this end, we develop an operator-valued concentration argument (Proposition~\ref{prop:uniform_estimation}) based on Bernstein's inequality for self-adjoint operators \cite{minsker2017some}, which enables uniform control of the error bounds over $\mathcal{F}_S^\gamma$.
By combining the uniform estimation and approximation bounds with a data-dependent choice $\lambda \sim N^{-\alpha}$, we obtain explicit convergence rates for the total reconstruction error. Specifically, for any $\alpha \in (0,\tfrac{1}{2})$ and $0 < \delta < 1$, with probability at least $1-\delta$, it holds uniformly that
\begin{align*}
\|\widehat{\mathcal{T}}_{\lambda,N} - \mathcal{T}\|_{\mathcal{F}_S^\gamma}
:= \sup_{h\in \mathcal{F}_S^\gamma} 
\|\widehat{\mathcal{T}}_{\lambda,N}h - \mathcal{T}h\|_{\mathcal{H}_K}
\;\le\; 
C(d,\gamma,\delta,S) N^{-\min\left\{\alpha\gamma, \tfrac{1}{2}(1-2\alpha)\right\}}.
\end{align*}
This establishes uniform convergence of the learned operator over the source space with high probability, providing a complete statistical characterization of the proposed operator learning framework.
\item   In Section~\ref{Numerical Experiments}, we validate the proposed method through a series of numerical studies. 
In Subsection~\ref{Darcy Flow: Formulation and Experiment}, we apply the method to the Darcy flow problem with various permeability fields, demonstrating consistently low approximation errors and fast computation across both smooth and heterogeneous settings, as well as robustness to discontinuous inputs. 
In Subsection~\ref{A Comparison with Neural Operator Learning}, we compare the proposed approach with the Green operator learning method introduced in \cite{stepaniants2023learning} on a high-frequency Helmholtz equation. The results show that the kernel-based operator significantly outperforms the Green operator approach in terms of accuracy, efficiency, and computational cost, while avoiding the generalization gap typically observed in data-driven operator learning methods. 
Overall, the experiments confirm that the proposed method provides a reliable, efficient, and highly generalizable framework for operator learning in PDEs. Section \ref{sec:conclusions} concludes the paper.
\end{enumerate}

\section{Methodology}\label{Methodology}
In this work, we consider the following linear Dirichlet boundary value problem stated in \eqref{pde} where $D\subseteq \mathbb{R}^d$ is a  domain,  $\mathcal{B}$ encodes the boundary conditions
(Dirichlet or Neumann problem), and $\mathcal{L}_s:H^s(D) \rightarrow L^2(D)$ is a linear differential operator defined as
\begin{align}\label{diff-operator}
\mathcal{L}_s u = \sum_{\alpha \in I_s} \varphi_{\alpha} \partial^{\alpha}u,\quad \alpha\in I_s, 
\end{align}
where $I_s:=\{\alpha\in\mathbb{N}^d:|\alpha|\leq s\}$ with $|\alpha|=\sum_{j=1}^d\alpha_j$ for $\alpha=(\alpha_1,\dots,\alpha_d)\in\mathbb{N}^d$, and $\varphi_\alpha:  D \rightarrow \mathbb{R}$ are functions such that $\max _{\alpha\in I_s}\left\|\varphi_\alpha\right\|_{\infty}<\infty$.  

The problem of interest is the approximation of the solution map $\mathcal{T} : (f,g) \mapsto u$, where $f$ and $g$ denote the input functions of the Dirichlet boundary value problem \eqref{pde}, and $u$ is the corresponding solution. 
A large body of existing work in operator learning aims to construct a surrogate operator $\widehat{\mathcal{T}}$ for the true solution map $\mathcal{T}$. 
These approaches are typically data-driven and non-intrusive, and are trained by minimizing a least-squares loss over paired data samples $\{(f_i, g_i, u_i)\}_{i=1}^n$; representative examples include \cite{kovachki2023neural,lu2021learning,nelsen2024operator}.
While such methods are particularly appealing in settings where an explicit model is unavailable, their computational effectiveness on unseen test inputs depends critically on the sampling strategy used for the input functions. 
Moreover, these approaches typically require extensive training, and their predictive accuracy is generally reliable only in neighborhoods of the sampled data. 

For this reason, we propose a novel machine learning approach for solving general nonhomogeneous PDEs of the type~\eqref{pde} for which we have access to physical prior information, namely, {\it the linear differential operator $\mathcal{L}_s$ and the boundary operator $\mathcal{B}$ are assumed to be known}.  
The proposed method achieves performance comparable to existing operator learning approaches, while requiring no training on paired data samples $\{(f_i, g_i, u_i)\}_{i=1}^n$. 
Moreover, it demonstrates superior accuracy, computational efficiency, and generalization capability compared to standard data-driven operator learning methods.

\subsection{Problem formulation}\label{Problem Formulation}
We restrict attention to Dirichlet boundary conditions in what follows, so that $\mathcal{B}=I$ is the identity operator. 
We first reformulate the boundary value problem \eqref{pde} on the closure $\overline{D}= D\cup\partial D$ as we did in \eqref{reformuted Ps} using the operator $P_s$ and the function $h$. Notice that by \eqref{diff-operator}, the operator $P_s : H^s(\overline{D}) \to L^2(\overline{D})$ is a linear differential operator of the form
\begin{align}\label{pde-operator}
P_s u
= \sum_{\alpha \in I_s} \phi_{\alpha}\, \partial^{\alpha} u, \qquad \alpha \in I_s,
\end{align}
where $\phi_{\alpha} = \mathbf{1}_{\partial  D}$ for $|\alpha| = 0$, and $\phi_{\alpha} = \varphi_{\alpha}\mathbf{1}_{ D}$ for $0 < |\alpha| \le s$.
With this construction, the boundary value problem \eqref{pde} is equivalently reformulated as
$P_s u = h$ in $\overline{D}$.
A standard approach to learning the solution $u $ of $P_s u = h$  is to define a loss function $l:\mathbb{R}\times\mathbb{R}\to [0,\infty)$ and minimize the statistical risk
\begin{align}\label{exp}
\min_{u\in\mathcal{U}} \; \mathbb{E}_{X\sim\mu}\, l\bigl(P_s u(X), h(X)\bigr),
\end{align}
where $\mathcal{U}$ denotes a hypothesis function space and $\mu$ is a sampling probability measure supported on $\overline{D}$. Note that the boundary $\partial D$ is a Borel set of measure zero. In practice (see Section~\ref{Numerical Experiments}), we sample in the interior $D$ and on the boundary $\partial D$ separately.

In practice, only a finite dataset $\{(X^{(i)}, Y^{(i)}=h(X^{(i)}))\}_{i=1}^N$ is available. Problem~\eqref{exp} is therefore approximated by replacing $\mu$ with the empirical measure $\mu^N=\frac{1}{N}\sum_{i=1}^N \delta_{X^{(i)}}$, which yields the empirical risk minimization problem
\begin{align}\label{emp}
\min_{u\in\mathcal{U}} \; \frac{1}{N} \sum_{i=1}^N l\bigl(P_s u(X^{(i)}), Y^{(i)}\bigr).
\end{align}

Solving~\eqref{emp} corresponds to estimating an approximation $\widehat{u}$ of the solution $u $ corresponding to a given input function $h$ of the PDE \eqref{pde}.
However, our ultimate objective is not merely to approximate a single solution, but to learn a solution operator $\widehat{\mathcal{T}}$ that produces $\widehat{u}$ out of $h$ and the finite dataset $\{(X^{(i)}, Y^{(i)}=h(X^{(i)}))\}_{i=1}^N$. Ideally, the estimation procedure should hence be independent of the specific choice of input function $h$. This poses a fundamental challenge. When using neural network-based methods, this goal cannot be achieved, as any change to the input function $h$ necessitates expensive retraining of the neural network used to approximate the solution.

In contrast, kernel-based methods admit closed-form solutions. As shown in Section \ref{A Physics-informed Kernel Approach} and Section \ref{From PDE Solvers to Solution Operators}, the resulting estimator is independent of the input function $h$ and consequently yields a well-defined solution operator, thereby shifting the paradigm from solving individual PDEs to learning an operator-based solver. 
Moreover, the proposed estimator is consistent with the posterior obtained via Gaussian process regression \cite{raissi2017machine,pfortner2022physics} for learning linear PDEs. To the best of our knowledge, this is the first work to explicitly realize such an estimator as a solution operator, thereby providing a theoretically grounded framework for operator approximation.

In what follows, we focus on the squared loss
\[
l(y_1,y_2) = \lvert y_1 - y_2 \rvert^2.
\]
We consider the hypothesis space $\mathcal{U}$ in the minimization problem \eqref{emp} to be a reproducing kernel Hilbert space (RKHS) $\mathcal{H}_K$ associated with a Mercer kernel $K:\overline{D}\times\overline{D}\to\mathbb{R}$. This choice is well justified, as one may select a universal kernel, such as the Gaussian kernel, for which the corresponding RKHS $\mathcal{H}_K$ is dense in the space of continuous functions $C(M)$ for any compact set $M\subset\overline{D}$ with respect to the uniform norm.

\subsection{A physics-informed kernel approach}
\label{A Physics-informed Kernel Approach}

The main idea behind the physics-informed kernel learning approach that we propose is modeling the solution $u$ in the RKHS $\mathcal{H}_K$ associated with a Mercer kernel $K:\overline{D}\times\overline{D}\to \mathbb{R}$ and to incorporate physical priors (typically encoded through PDE operator information) into a kernel ridge regression formulation. 

In order to make the method explicit, we shall be solving the following empirical minimization problem 
\begin{align}
\widehat{u}_{\lambda,N}&:=\mathop{\arg\min}\limits_{u\in \mathcal{H}_{K}} \ \widehat{R}_{\lambda,N}(u), \label{emp-pro}\\
\widehat{R}_{\lambda,N}(u)&:=\frac{1}{N}\sum_{i=1}^{N} \left| P_s u(X^{(i)})-Y^{(i)}\right|^2 + \lambda\|u\|_{\mathcal{H}_K}^2, \label{emp-fun}
\end{align}
where $P_s$ is the linear operator defined in \eqref{pde-operator}, $\{(X^{(i)}, Y^{(i)}=h(X^{(i)}))\}_{i=1}^N$ is the dataset, and $\lambda\geq0$ is the Tikhonov regularization parameter.
The functional $\widehat{R}_{\lambda,N}$ is referred to as the {\bfi regularized empirical risk}. 

The measure-theoretic analogue, referred to as {\bfi regularized statistical risk}, is denoted as $R_{\lambda}$ and is defined by
\begin{equation}
\label{exp-fun}
R_{\lambda}(u):=\|P_su- h\|_{L^2(\mu)}^2+\lambda \|u\|_{\mathcal{H}_K}^2, 
\end{equation}
where $\|\cdot\|_{L^2(\mu)}$ is the $L^2$ norm with respect to the sampling probability measure $\mu$. We denote by $u^{*}_{\lambda}\in\mathcal{H}_K$ the {\bfi best-in-class function} with the minimal associated in-class regularized statistical risk, that is,
\begin{align}
\label{exp-pro}
u^{*}_{\lambda}:= \mathop{\arg\min}\limits_{u\in \mathcal{H}_{K}}  R_{\lambda}(u) .  
\end{align}

\begin{remark}
\normalfont
If $K \in C_b^{2s+1}(\overline{D}\times\overline{D})$, then the differential reproducing property \cite{zhou2008derivative,RCSP2} implies that ${\mathcal H}_K \subset C_b^{s}(\overline{D})$, meaning that every function in ${\mathcal H}_K$ is at least $s$ times continuously differentiable. As a consequence, the regularized empirical risk $\widehat{R}_{\lambda,N}$ defined in~\eqref{emp-fun} is well defined under this condition. 
Furthermore, below in Proposition~\ref{Wel-Ope}, we show that $P_s u$ is $L^2(\mu)$-integrable for all $u \in {\mathcal H}_K$ and hence the regularized statistical risk $R_\lambda$ in \eqref{exp-fun} is well-defined.
\end{remark}

\subsubsection{An operator-theoretic formulation of the learning problem}
In this section, we propose an operator-theoretic framework to characterize the minimizers of the optimization problems~\eqref{emp-pro}–\eqref{emp-fun} and~\eqref{exp-fun}–\eqref{exp-pro}. More precisely, we exploit the differential reproducing property \cite{zhou2008derivative,RCSP2} to show that the linear differential operator $P_s$, defined in~\eqref{pde-operator}, is bounded. Moreover, we derive its adjoint operator with respect to the RKHS inner product, which admits an explicit representation in terms of the kernel function.
In contrast to classical regularization methods \cite{plato2018optimal,benning2018modern}, where the adjoint operator is typically defined in an $L^2$ sense, our perspective yields a fully computable representation of the estimator for the associated minimization problems by leveraging kernel ridge regression. We start by studying the properties of the linear operator $P_s$ in the following proposition. The proof follows the similar strategy as in \cite{RCSP2}, based on the differential reproducing property of the kernel. 

\begin{proposition}\label{Wel-Ope}
Given a pair $(K,\mu)$ with $K\in C_b^{2s+1}(\overline{D}\times\overline{D})$ being a Mercer kernel. Then, the operator $P_s$ defined in \eqref{diff-operator} is a bounded linear operator from $\mathcal{H}_K$ to $L^2(\overline{D};\mu)$ with operator norm satisfying $\|P_s\|\leq C_{d+s}^dC\kappa$, where $C_{d+s}^d=\binom{d+s}{d} = \frac{(d+s)!}{d! \, s!} $ is the binomial coefficient, $\kappa^2=\|K\|_{C_b^{2s}(\overline{D} \times \overline{D})}$, and $C$ is a uniform bound for the function coefficients in the operator $P_s $ (see condition under \eqref{diff-operator}). The adjoint operator 
$P_s^{*}: L^2(\overline{D};\mu) \longrightarrow {\mathcal H} _K$ of $P_s: {\mathcal H}_K \longrightarrow L^2(\overline{D};\mu)$ is given by
\begin{align}\label{adjoint}
P_s^{*}h=\int_{\overline{D}} h(x)P_s^{(1,0)} K(x,\cdot) \, \mathrm{d} \mu(x), \quad \mbox{for all $h\in L^2(\overline{D};\mu)$},
\end{align} 
where the notation $P_s^{(1,0)}K$ stands for the action of the differential operator $P_s$ on the first variable of the kernel function $K$. As a consequence, the bounded linear operator $B_s: \mathcal{H}_K\longrightarrow \mathcal{H}_K$, defined by
\begin{align}\label{positive}
B_sh:=P_s^{*}P_s h=\int_{\overline{D}} P_s h(x) P_s^{(1,0)} K(x,\cdot) \mathrm{~d}\mu(x),
\end{align}
is a positive semidefinite trace class operator that satisfies $\operatorname{Tr}(B_s)\leq (C_{d+s}^dC\kappa)^2$.
\end{proposition}

In practice, we only have access to finite datasets obtained by sampling the measure $\mu$. The following proposition is an empirical version of Proposition \ref{Wel-Ope}.

\begin{proposition}\label{Wel-Emp} 
Let $X^{(1)},\cdots, X^{(N)}$ be i.i.d. random samples drawn from the sampling probability measure $\mu$. Then the operator
$P_{s,N}: \mathcal{H}_K \rightarrow \mathbb{R}^{N}$ defined by
\begin{align*}
P_{s,N}h=\frac{1}{\sqrt{N}}P_sh(X_N):=\frac{1}{\sqrt{N}}\mathrm{Vec}\left(P_sh(X^{(1)})| \cdots | P_sh(X^{(N)})\right),
\end{align*}
is a bounded linear operator.
Its adjoint operator $P_{s,N}^*:\mathbb{R}^{N}  \rightarrow \mathcal{H}_K$ has finite rank and is given by
\begin{align}\label{emp-adjoint}
P^*_{s,N}W=\frac{1}{\sqrt{N}}W^{T}P_s^{(1,0)} K(X_N,\cdot), \quad W\in \mathbb{R}^{N}.   
\end{align}
Moreover, the operator $B_{s,N}$ defined by
\begin{align}\label{emp-ope}
B_{s,N}h:=P^*_{s,N}P_{s,N}h=\frac{1}{N}P_sf(X_N)\cdot P_s^{(1,0)} K(X_N,\cdot),    
\end{align}
is a positive semidefinite compact operator. 
\end{proposition}

\begin{proof}
The formal explicit forms of $P_{s,N}^{\ast}$ and $B_{s,N}$ follow from a direct computation. We have (the constants are the same as in Proposition \ref{Wel-Ope})
{\small
\begin{equation*}
\|P_{s,N}h\|^2=\frac{1}{N}\sum_{n=1}^{N}|P_sh(X^{(n)})|^2\leq \frac{C^2}{N}\sum_{n=1}^{N}\sum_{\alpha \in I_s}\sum_{\beta \in I_s}\|h\|_{C_b^{|\alpha|}}\|h\|_{C_b^{|\beta|}}\leq (C_{d+s}^dC\kappa)^2\|h\|^2_{\mathcal{H}_K},
\end{equation*}}
which implies that $P_{s,N}$ is bounded and that $\|P_{s,N}\|\leq C_{d+s}^dC\kappa$. The compactness of the operator $B_{s,N}$ follows from a proof similar to that of Proposition \ref{Wel-Ope} for $B_s$.
\end{proof}

Having defined the operators $P_s$ and $P_{s,N}$, we immediately obtain the following operator representation of the minimizers that solve the optimization problems \eqref{emp-pro}-\eqref{emp-fun} and \eqref{exp-fun}-\eqref{exp-pro}.

\begin{corollary}
\label{Rep-Ope} 
Let $\widehat{u}_{\lambda,N}$ and $u_{\lambda}^*$ be the minimizers of \eqref{emp-fun} and \eqref{exp-fun} respectively. Then, for all $\lambda>0$, these minimizers are unique and are given by 
\begin{equation}
u_{\lambda}^*:=(B_s+\lambda I)^{-1}P_s^{*}h, \qquad
\widehat{u}_{\lambda,N}:=\frac{1}{\sqrt{N}}(B_{s,N}+\lambda I)^{-1}P_{s,N}^{*}Y_N.\label{em10} 
\end{equation}
\end{corollary}

\subsubsection{Kernel representer}
We now derive a kernel-based representation for the solution of the learning problem~\eqref{emp-pro}–\eqref{emp-fun}. In particular, we obtain a closed-form expression for the estimator in terms of the kernel function by introducing a generalized Gram matrix $P_s^{(1,1)}K(X_N,X_N)$, where the notation $P_s^{(1,1)}K$ denotes the action of the operator $P_s$ on both arguments of the kernel. 

First, we show that the generalized Gram matrix $P_s^{(1,1)}K(X_N,X_N)$ is positive semidefinite. Consequently, for any $\lambda>0$, the matrix $P_s^{(1,1)}K(X_N,X_N) + \lambda N I$ is invertible, where $I$ denotes the identity matrix.

\begin{proposition}\label{Pos-Gra}
Given a Mercer kernel $K$  with $K\in C_b^{2s+1}(\overline{D}\times \overline{D})$, the generalized Gram matrix $P_s^{(1,1)}K(X_N,X_N)$ is positive semidefinite. 
\end{proposition}

\begin{proof}
Since the generalized Gram matrix $P_s^{(1,1)}K(X_N,X_N)$ is real symmetric, then there exists an orthonormal matrix $P\in\mathbb{R}^{N\times N}$ that diagonalizes $P_s^{(1,1)}K(X_N,X_N)$. This means that
\begin{multline*}
P_s^{(1,1)}K(X_N,X_N)=PDP^{\top}\\
=\begin{bmatrix}
|&|&\dots&|\\
f_1&f_2&\dots&f_N\\
|&|&\dots&|
\end{bmatrix}
\begin{bmatrix}
d_1&0&\dots&0 \\
0&d_2&\dots&0 \\
\vdots&\vdots&\ddots&\vdots\\
0&0&\dots&d_N
\end{bmatrix}\begin{bmatrix}
|&|&\dots&|\\
f_1&f_2&\dots&f_N\\
|&|&\dots&|
\end{bmatrix}^{\top},
\end{multline*} 
where  $\left \{d_i \right\}_{i=1}^{N}$ and $\left \{f_i \right\}_{i=1}^{N}$ are the real eigenvalues and the corresponding basis of orthonormal eigenvectors of $P_s^{(1,1)}K(X_N,X_N)$. 
We now define $\widetilde{e}_i=\langle f_i, P_s^{(1,0)}K(X_N,\cdot ) \rangle_{\mathbb{R}^{N}}$. By \cite[Theorem 2.7]{RCSP2}, we obtain that $\widetilde{e}_i\in\mathcal{H}_K$ and that
\begin{equation*}
\begin{aligned}
\|\widetilde{e}_i\|_{\mathcal{H}_K}^2 &= \left\langle\langle f_i, P_s^{(1,0)}K(X_N,\cdot ) \rangle_{\mathbb{R}^{N}},\langle f_i, P_s^{(1,0)}K(X_N,\cdot ) \rangle_{\mathbb{R}^{N}}\right\rangle_{\mathcal{H}_K} \\
&=f_i^{\top}P_s^{(1,1)}K(X_N,X_N)f_i=f_i^{\top}(d_i f_i)
=d_i.
\end{aligned}
\end{equation*}
Hence, $d_i\geq0$ for all $i=1,\cdots,N$ and hence we can conclude that the generalized Gram matrix $P_s^{(1,1)}K(X_N,X_N)$ is positive semidefinite.
\end{proof}

\begin{remark}
\normalfont
(i) Proposition~\ref{Pos-Gra} establishes an interesting result: the regularity of the kernel function $K$ guarantees automatically the positive semidefiniteness of the generalized Gram matrix $P_s^{(1,1)}K(X_N,X_N)$. In other approaches, similar conditions are imposed as additional assumptions, like, for instance, in \cite{franke1998solving,giesl2007meshless}, where the linear independence of a Hermite--Birkhoff interpolation functional is required.

(ii) By the definition of positive semidefinite kernels, the standard Gram matrix $K(X_N, X_N)$ is positive semidefinite, which corresponds to the case that the order index $s=0$ in Proposition~\ref{Pos-Gra}. For $s=1$, this result reduces to the result stated in Proposition~3.4 of \cite{RCSP2}.
\end{remark}

Below, we derive a Differential Representer Theorem that provides a closed-form solution to the minimization problem \eqref{emp-pro}-\eqref{emp-fun}. 

\begin{theorem}[{\bf Differential Representer Theorem}]
\label{Rep-Ker} 
For every $\lambda>0$, the optimization problem \eqref{emp-pro} has a unique solution $\widehat{u}_{\lambda, N}$ that can be represented as
\begin{equation}
\label{rep-ker}
\widehat{u}_{\lambda,N}= \sum_{i=1}^N \widehat {c}_{i}P_s^{(1,0)}K({X}^{(i)},\cdot),
\end{equation}
with $\widehat {c}_{1}, \ldots, \widehat {c}_{N} \in \mathbb{R}$. If we denote by $\widehat{c} \in \mathbb{R}^{N}$ the vectorization of $\left(\widehat {c}_{1}| \cdots | \widehat {c}_{N}\right)$, we have
\begin{align*}
\widehat{c}=(P_s^{(1,1)}K(X_N,X_N)+\lambda NI)^{-1}Y_N.
\end{align*}
\end{theorem}

\begin{proof} The proof is based on the operator representations of the minimizers that we introduced in Proposition \ref{Wel-Emp}, which allows us to use tools from spectral theory. 
Let $\mathcal{H}_K^{N}$ be the space given by
\begin{equation}\label{sub-spa}
\mathcal{H}_K^{N}:=\mathrm{span}\left\{P_s^{(1,0)}K(X^{(i)},\cdot)\mid i=1,\cdots, N\right\}.
\end{equation}
Obviously $\mathcal{H}_K^N$ defined in \eqref{sub-spa} is a subspace of $\mathcal{H}_K$ since $P_s^{(1,0)}K(x,\cdot)\in\mathcal{H}_K$ for all $x\in \mathbb{R}^{d}$. Then by the representation of the operator $B_{s,N}$ in Proposition \ref{Wel-Emp}, we know that $B_{s,N}(\mathcal{H}_{K}^N) \subseteq \mathcal{H}_{K}^N$ (see the expression \eqref{emp-ope}), that is, $\mathcal{H}_{K}^N$ is an invariant space for the operator $B_{s,N}$. This implies that, for any $\lambda>0 $, $(B_{s,N}+ \lambda I)(\mathcal{H}_{K}^N) \subseteq \mathcal{H}_{K}^N$. Now, since by Proposition \ref{Wel-Emp} the operator $B_{s,N} $ is positive semidefinite, we can conclude that the restriction $(B_{s,N}+ \lambda I)|_{\mathcal{H}_{K}^N} $ is invertible and since the space $\mathcal{H}_{K}^N  $ is finite-dimensional then it is also an invariant subspace of $(B_{s,N}+ \lambda I)|_{\mathcal{H}_{K}^N}^{-1} $, that is 
$(B_{s,N}+ \lambda I)|_{\mathcal{H}_{K}^N}^{-1} \left(\mathcal{H}_{K}^N\right) \subset \mathcal{H}_{K}^N$.
Thus, there exist constants $\widehat {c}_{1}, \ldots, \widehat {c}_{N} \in \mathbb{R}$ such that 
\begin{equation}\label{rep-ker1}
\widehat{u}_{\lambda,N}= \sum_{i=1}^N  \widehat {c}_{i} P_s^{(1,0)}K(X^{(i)},\cdot).
\end{equation}
Then, applying $(B_{s,N} + \lambda I)$ on both sides of \eqref{em10}, plugging \eqref{rep-ker1} into the identity, and denoting by $\widehat{c} \in \mathbb{R}^{N}$ the vectorization $\left(\widehat {c}_{1}| \cdots | \widehat {c}_{N}\right)$, we obtain
\begin{align}
\label{euler equ in this case}
\widehat{c}^{\top} \left (\frac{1}{N}P_s^{(1,1)}K(X_N,X_N)+\lambda I\right) P_{s}^{(1,0)} K(X_N,\cdot)=\frac{1}{N}Y_N^{\top}P_{s}^{(1,0)} K(X_N,\cdot).
\end{align}
Since the matrix $P_{s}^{(1,1)}K(X_N,X_N)+\lambda N I$ is invertible due to the positive semidefiniteness of the generalized Gram matrix $P_{s}^{(1,1)}(X_N,X_N)$ that we proved in Proposition \ref{Pos-Gra}, we can write the expression 
\begin{align}
\label{expression for chat}
\widehat{c}= (P_{s}^{(1,1)}(X_N,X_N)+\lambda NI)^{-1} Y_N,
\end{align}
that a straightforward verification shows that plugged into \eqref{euler equ in this case} satisfies \eqref{euler equ in this case}. This shows that the function $\widehat{u}_{\lambda,N} $ in \eqref{rep-ker1} with $\widehat{c} $ determined by \eqref{expression for chat} is a minimizer of the regularized empirical risk functional $ \widehat{R}_{\lambda,N}$ in \eqref{emp-fun}. Since by Proposition \ref{Rep-Ope}, this minimizer is unique, the result follows.
\end{proof}

\begin{remark}[Consistency with Gaussian process regression]\label{Consistence to Gaussian process regression}
\normalfont 
When the kernel $K$ is the Gaussian kernel and $\lambda = \sigma^2/N$, the estimator in~\eqref{rep-ker} coincides with the posterior mean of the physics-informed Gaussian process regression of Raissi--Perdikaris--Karniadakis~\cite{raissi2017machine}, extending the classical Gaussian process--kernel equivalence~\cite{kanagawa2018gaussian,RCSP2} to the operator-learning setting considered here.
This equivalence is useful for two reasons.
First, because the posterior mean is linear in the observations $h(X_N)$, the dependence on the input function factors through the kernel basis (see later on~\eqref{kernel basis}); this is precisely what allows us in Section~\ref{From PDE Solvers to Solution Operators} to pass from a PDE solver to a solution operator $\widehat{\mathcal{T}}_{\lambda,N}$.
Second, our frequentest viewpoint trades the posterior-variance uncertainty quantification of the Bayesian formulation for uniform RKHS-norm convergence rates on the learned operator over source-regularity classes (Theorem~\ref{Tot-Rec-Err}), which to our knowledge have not been established in the physics-informed Gaussian process literature.
\end{remark}

\begin{remark}[Online low-rank kernel regression with kernels]\label{Online Regression with Kernels}
Online and lifelong learning aim at updating a model efficiently when data arrive sequentially. 
For kernel-based regression methods, such updates can be performed using recursive formulas for the inverse Gram matrix, avoiding repeated recomputation from scratch. 
Since the estimator \eqref{rep-ker} has a regression structure, similar ideas can be adapted in our setting by exploiting a block matrix inversion formula for the generalized Gram matrix. 
This yields an efficient online update scheme with reduced computational cost. 
We refer to the supplementary material for details.
\end{remark}

\begin{remark}[Uniqueness of the estimator]\label{Uniqueness}\normalfont
Define the kernel of the operator $P_s$ in the RKHS $\mathcal{H}_K$ as
$\mathcal{H}_{\mathrm{null}}:=\{h\in\mathcal{H}_K\mid P_s h=0\}$.
In general, the space $\mathcal{H}_{\mathrm{null}}$ contains non-zero constant functions. Moreover, it is a closed subspace of $\mathcal{H}_K$. Indeed, for any Cauchy sequence $h_n\in\mathcal{H}_{\mathrm{null}}$ with limit $h\in\mathcal{H}_K$, i.e. $\|h_n-h\|_{\mathcal{H}_K}\to0$ as $n\to\infty$, we have that $h\in\mathcal{H}_{\mathrm{null}}$ since the differential reproducing property \cite[Theorem 2.7]{RCSP2} leads to that
\begin{align*}
\|P_sh_n-P_sh\|_{\infty} \leq \|P_s(h_n-h)\|_{\infty}\leq C_{d+s}^d~C\kappa \|h_n-h\|_{\mathcal{H}_K}\to 0, \quad \text{as } n\to\infty.  
\end{align*}
Hence, we can decompose the RKHS as 
$\mathcal{H}_K = \mathcal{H}_{\mathrm{null}} \oplus \mathcal{H}_{\mathrm{null}}^\bot$,    
where $\mathcal{H}_{\mathrm{null}}^\bot$ stands for the space of all orthonormal complements of $\mathcal{H}_{\mathrm{null}}$ with respect to the RKHS inner product. By the expression \eqref{rep-ker} and \cite[Corollary 2.8]{RCSP2}, it is clear that $\widehat{u}_{\lambda,N} \in \mathcal{H}_{\operatorname{null}}^\bot$.

Although adding elements in $\mathcal{H}_{\mathrm{null}} $ to the estimator $\widehat{u}_{\lambda,N} $ does not change the value of the first part in the empirical risk functional \eqref{emp-fun}, the optimizer $\widehat{u}_{\lambda,N} $ is unique. This is because for all $u\in\mathcal{H}_{\mathrm{null}}$, it can be shown that
\begin{align*}
\widehat{R}_{\lambda,N}&(\widehat{u}_{\lambda,N}+u)=\frac{1}{N}\sum_{n=1}^{N} \|P_s\widehat{u}_{\lambda,N}(X^{(n)})-Y^{(n)}\|^2 + \lambda\left(\|\widehat{u}_{\lambda,N}\|_{\mathcal{H}_K}^2+\|u\|^2_{\mathcal{H}_K} \right).
\end{align*}
\end{remark}

\subsection{From PDE solvers to solution operators}\label{From PDE Solvers to Solution Operators}
Recall that our goal is to learn the solution operator $\mathcal{T}: h \mapsto u$, rather than to approximate a single solution. We rewrite the estimator~\eqref{rep-ker} in the form
\begin{align}\label{rep-ker2}
\widehat{u}_{\lambda,N}
= \big\langle h(X_N), \big(P_s^{(1,1)}K(X_N,X_N) + \lambda N I\big)^{-1}
P_s^{(1,0)}K(X_N,\cdot) \big\rangle,
\end{align}
where $\langle \cdot,\cdot \rangle$ denotes the Euclidean inner product on $\mathbb{R}^N$. The closed-form expression~\eqref{rep-ker2} provides a very simple and efficient procedure for computing approximate solutions of the PDE~\eqref{pde} for any given input function $h$. Specifically, for a kernel $K : \overline{D} \times \overline{D} \to \mathbb{R}$ satisfying $K \in C_b^{2s+1}(\overline{D} \times \overline{D})$, one may precompute, using the sample points $X^{(1)},\dots,X^{(N)}$, the {\bfi kernel basis}
\begin{align}\label{kernel basis}
(\psi_1,\cdots,\psi_N)=\big(P_s^{(1,1)}K(X_N,X_N) + \lambda N I\big)^{-1}
P_s^{(1,0)}K(X_N,\cdot).
\end{align}
Then, for any input function $h \in L^2(\overline{D};\mu)$, the corresponding approximate solution is obtained simply by taking the inner product of $h(X_N)$ with this kernel basis. This expression provides an interpretation of the estimator \eqref{rep-ker2} as a solution operator. 

Using the operator representation \eqref{em10}, we define the map $\widehat{\mathcal{T}}_{\lambda,N} : L^2(\overline{D};\mu) \to \mathcal{H}_K$
by
\begin{align}\label{empirical operator estimator}
\widehat{\mathcal{T}}_{\lambda,N}
:= \frac{1}{\sqrt{N}} \bigl(B_{s,N} + \lambda I\bigr)^{-1} P_{s,N}^{*} \pi_N,
\end{align}
where $\pi_N : L^2(\overline{D};\mu) \to \mathbb{R}^N$ denotes the projection operator defined as
\begin{align}\label{pi}
\pi_N(h) := \bigl(h(X^{(1)}), \dots, h(X^{(N)})\bigr),
\end{align}
with $X^{(1)}, \dots, X^{(N)}$ the i.i.d.\ samples drawn from the sampling probability measure $\mu$. We refer to $\widehat{\mathcal{T}}_{\lambda,N}$ as the {\bfi empirical solution operator estimator}. Similarly, we define the {\bfi statistical solution operator estimator}
$\mathcal{T}_{\lambda} : L^2(\overline{D};\mu) \to \mathcal{H}_K$ by
\begin{align}\label{statistical solution operator estimator}
\mathcal{T}_{\lambda} h
:= (B_s + \lambda I)^{-1} P_s^{*} h.
\end{align}

By Corollary~\ref{Rep-Ope}, it follows that for each
$h \in L^2(\overline{D};\mu)$, the estimators
$\widehat{\mathcal{T}}_{\lambda,N} h$ and $\mathcal{T}_{\lambda} h$
are the respective minimizers of~\eqref{emp-fun} and~\eqref{exp-fun}.
These two solution operator estimators provide regularized
approximations of the solution operator
$\mathcal{T}: h \mapsto u$ associated with the boundary value
problem~\eqref{pde}.

\begin{remark}[Connection to Green's function] \normalfont
In classical PDE theory \cite{evans2022partial}, if the Green’s function $G$ associated with the PDE~\eqref{pde} is known, then for any input functions $f$ and $g$, the corresponding solution can be expressed as
\[
u(x)=(\mathcal{T}h)(x)
= \int_{D} f(y)\, G(x,y)\, \mathrm{d}y
+ \int_{\partial D} g(\xi)\, \partial_{n_\xi} G(x,\xi)\, \mathrm{d}S_\xi.
\]
Suppose that we sample $N_1$ points in the interior of $D$ and $N_2 = N - N_1$ points on the boundary $\partial D$. The solution can then be approximated by
\begin{align}\label{eq1}
u =\mathcal{T}h\approx \sum_{i=1}^{N_1} f(X^{(i)})\, G(\cdot,X^{(i)})
+ \sum_{i=N_1+1}^{N} g(X^{(i)})\, \partial_{n_\xi} G(\cdot,X^{(i)}).
\end{align}
Using the kernel basis~\eqref{kernel basis}, the estimator~\eqref{rep-ker} admits the representation
\begin{align}\label{eq3}
\widehat{u}_{\lambda,N}=\widehat{\mathcal{T}}_{\lambda,N}h 
= \sum_{i=1}^{N_1} f(X^{(i)})\, \psi_i
+ \sum_{i=N_1+1}^{N} g(X^{(i)})\, \psi_i.    
\end{align}
Combining \eqref{eq1} and \eqref{eq3}, the kernel basis $\{\psi_i\}_{i=1}^N$ can be interpreted as an approximation of the Green's function and its normal derivatives associated with the PDE \eqref{pde}. Consequently, the proposed physics-informed kernel method effectively learns an approximation of the Green's function itself. This observation explains why the resulting estimator naturally defines a solution operator, rather than merely approximating a single PDE solution.
\end{remark}

\section{Estimation and approximation error bounds}\label{Estimation and approximation error bounds}
In this section, we propose error bounds for the empirical operator estimator
$
\widehat{\mathcal{T}}_{\lambda,N} : L^{2}( \overline{D};\mu)\to \mathcal{H}_K
$
defined in~\eqref{empirical operator estimator} with respect to the target solution operator
$
\mathcal{T} : h\mapsto u
$ given by the PDE \eqref{pde}.
A standard approach in this setup is to decompose the {\bfi reconstruction error} $\widehat{\mathcal{T}}_{\lambda,N}-\mathcal{T} $ as the sum of what we shall be calling the {\bfi estimation} and {\bfi approximation errors}. 
\begin{align*} 
\widehat{\mathcal{T}}_{\lambda,N}-\mathcal{T}=\underbrace{\widehat{\mathcal{T}}_{\lambda,N}-\mathcal{T}_{\lambda}}_{\text{Estimation error}}\quad+\underbrace{\mathcal{T}_{\lambda}-\mathcal{T}}_{\text{Approximation error}},
\end{align*}
where $\mathcal{T}_{\lambda}:L^{2}( \overline{D};\mu)\to \mathcal{H}_K$ is the statistical solution operator estimator given by \eqref{statistical solution operator estimator}.
Given the kernel framework we work on, we restrict the domain of $\mathcal{T}$ to
\begin{align}\label{input domain}
\mathcal{F}
:= \bigl\{ h \in L^2( \overline{D};\mu)\ \big|\ \mathcal{T}h \in \mathcal{H}_K \bigr\}.
\end{align}
By Remark~\ref{Uniqueness}, we know that for any $h \in L^2(\overline{D};\mu)$, the estimator $\widehat{\mathcal{T}}_{\lambda,N}h$ lies in $\mathcal{H}_{\text{null}}^\bot$. 
This motivates the introduction of the {\bfi effective domain}
\begin{align}\label{effective domain}
\mathcal{F}_{\text{eff}}
:= \bigl\{ h \in L^2(\overline{D};\mu)\ \big|\ \mathcal{T}h \in \mathcal{H}_{\text{null}}^\bot \bigr\}.
\end{align}
The space $\mathcal{F}_{\text{eff}}$ is the largest subset of $L^2(\overline{D};\mu)$ on which the operator $\mathcal{T}$ is accessible by the estimator $\widehat{\mathcal{T}}_{\lambda,N}$. 

In the study of the approximation error, a standard assumption is that the solution satisfies a so-called source condition \cite{lu2019nonparametric,feng2024learning} that restricts the space of potential solutions of the PDE that needs to be solved and allows us to study the convergence properties of the solution estimator $\widehat{\mathcal{T}}_{\lambda, N}$. In order to state the source condition,  let $\gamma \in (0,1)$ and $S > 0$ and define the {\bfi source-induced input spaces} as 
\begin{align}\label{source space}
\mathcal{F}_S^\gamma
:= \bigl\{ h \in L^2(\overline{D};\mu)\ \big|\ \mathcal{T}h \in  \Omega_S^\gamma \bigr\}
\subset \mathcal{F},   
\end{align}
where 
\begin{align}\label{sou-con}
\Omega_S^\gamma
:= \bigl\{ u \in \mathcal{H}_K \,\big|\, u = B_s^\gamma \psi,\ \psi \in \mathcal{H}_K,\ \|\psi\|_{\mathcal{H}_K} < S \bigr\}.
\end{align}
The {\bfi source condition} consists of assuming that the solutions of the PDE that needs to be solved lie in $\Omega_S^\gamma $ or equivalently, that the domain of ${\cal T}$ is $\mathcal{F}_S^\gamma $.

\begin{remark}[Density and interpretation of source-induced spaces]
Since $B_s$ is bounded, self-adjoint, and positive semidefinite by Proposition~\ref{Wel-Ope}, its fractional powers $B_s^\gamma$ are well defined for $\gamma\in(0,1)$ and satisfy $\ker(B_s^\gamma)=\ker(B_s)=\mathcal H_{\mathrm{null}}$.
Hence,
\[
\overline{\bigcup_{S>0}\Omega_S^\gamma}
=\overline{\operatorname{Ran}(B_s^\gamma)}=(\ker B_s^\gamma)^\perp
=\mathcal H_{\mathrm{null}}^\perp.
\]

Through the operator $\mathcal T$, this induces a corresponding family of source-induced input spaces $\mathcal F_S^\gamma=\mathcal T^{-1}(\Omega_S^\gamma)$.
If $\mathcal T:\mathcal F_{\mathrm{eff}}\to \mathcal H_{\mathrm{null}}^\perp$ admits a continuous inverse on its range, then the density property above transfers to the input space, so that
\[
\overline{\bigcup_{S>0}\mathcal F_S^\gamma}
=
\mathcal F_{\mathrm{eff}}.
\]

This suggests that the source condition provides a regularity decomposition of the effective domain.
\end{remark}

\noindent{\bf Approximation error.}\quad 
By Proposition~\ref{Wel-Ope}, the operator $B_s = P_s^* P_s$ is positive and self-adoint on $\mathcal{H}_K$. Let $B_s = \sum_{n=1}^L \lambda_n \langle \cdot, e_n \rangle_{\mathcal{H}_K} e_n
\quad (\text{with } L \in \mathbb{N} \cup \{\infty\})$
be its spectral decomposition, where $0< \lambda_{n+1} \le \lambda_n$ and $\{e_n\}_{n=1}^L$ is an orthonormal basis of $\mathcal{H}_K$ (available for any continuous kernel \cite{steinwart2008support}).
For each $h \in \mathcal{F}_S^\gamma$ we have
\begin{align*}
\|\mathcal{T}_{\lambda} h - \mathcal{T} h\|_{\mathcal{H}_K}^2
&= \|(B_s + \lambda I)^{-1} B_s \mathcal{T}h - \mathcal{T}h\|_{\mathcal{H}_K}^2 = \|\lambda (B_s + \lambda I)^{-1} \mathcal{T}h\|_{\mathcal{H}_K}^2 \\
&= \sum_{n=1}^L \left(\frac{\lambda}{\lambda_n + \lambda}\right)^2\bigl|\langle \mathcal{T}h, e_n \rangle_{\mathcal{H}_K}\bigr|^2 \le \sum_{n=1}^L \left(\frac{\lambda}{\lambda_n}\right)^{2\gamma}
     \bigl|\langle \mathcal{T}h, e_n \rangle_{\mathcal{H}_K}\bigr|^2 \\
&= \lambda^{2\gamma} \sum_{n=1}^L \lambda_n^{-2\gamma}
     \bigl|\langle \mathcal{T}h, e_n \rangle_{\mathcal{H}_K}\bigr|^2
 = \lambda^{2\gamma} \|B_s^{-\gamma} \mathcal{T}h\|_{\mathcal{H}_K}^2,
\end{align*}
where $B_s^{-\gamma} \mathcal{T}h$ denotes the pre-image of $\mathcal{T}h$ under $B_s^\gamma$ and the inequality follows from the concavity of $g(x) = x^\gamma$ on $[0,\infty)$. 
By the source condition~\eqref{sou-con}, for $h \in \mathcal{F}_S^\gamma$ we have $\mathcal{T}h \in  \Omega_S^\gamma$, hence $\|B_s^{-\gamma} \mathcal{T}h\|_{\mathcal{H}_K} \le S$. Consequently,
\begin{align}\label{app-err}
\|\mathcal{T}_{\lambda} - \mathcal{T}\|_{\mathcal{F}_S^\gamma}:=\sup_{h\in \mathcal{F}_S^\gamma} \|\mathcal{T}_{\lambda}h - \mathcal{T}h\|_{\mathcal{H}_K}
\le S \lambda^\gamma.
\end{align}

\noindent{\bf Pathwise convergence of the estimation error.}\quad
We now show that the empirical operator estimator $\widehat{\mathcal{T}}_{\lambda,N}$  converges to the statistical operator estimator $\mathcal{T}_\lambda$ for any function in the input domain \eqref{input domain}. In the next two results $C_{d+s}^d=\binom{d+s}{d} = \frac{(d+s)!}{d! \, s!} $ is the binomial coefficient, $\kappa^2=\|K\|_{C_b^{2s}(\overline{D} \times \overline{D})}$, and $C$ is a uniform bound for the coefficients in the operator $P_s $ (see condition under \eqref{diff-operator}).

By Corollary~\ref{Rep-Ope}, it follows that for each
$h \in L^2(\overline{D};\mu)$, the estimators
$\widehat{\mathcal{T}}_{\lambda,N} h$ and $\mathcal{T}_{\lambda} h$
are the respective minimizers of~\eqref{emp-fun} and~\eqref{exp-fun}. 
This allows us to employ the operator decomposition technique developed in \cite{RCSP2}, while extending the analysis from Hamiltonian operators to general differential operators. 

\begin{proposition}[Pathwise estimation error bounds]\label{Sam-Err} For any function $h \in \mathcal{F}$ and $0< \delta <1$, with probability at least $1-\delta$, it holds that
\footnotesize
\begin{align*}
\left\|\widehat{\mathcal{T}}_{\lambda,N}h-\mathcal{T}_{\lambda}h\right\|_{\mathcal{H}_K}  
\leq\left(
\sqrt{\frac{8\log(2/\delta)}{N}} + 1
\right)
\sqrt{\frac{2\log(2/\delta)}{N\lambda^2}}\,
(C_{d+s}^d\, C\, \kappa)^2\,  \|\mathcal{T}h\|_{\mathcal{H}_K}\left(1+\frac{C_{s+d}^d\kappa C}{\sqrt{\lambda}}\right).
\end{align*} 
\normalfont
\end{proposition}

\subsection{Uniform convergence of the total reconstruction error}\label{Uniform Convergence of Total Reconstruction Error}

Notice that the approximation error~\eqref{app-err} is uniformly controlled over the source-induced input spaces. 
However, such uniform control does not directly extend to the estimation error via Proposition~\ref{Sam-Err}, since the corresponding high-probability events depend on the specific choice of $h \in \mathcal{F}$. 

The key step is to reinterpret the stochastic error in operator form. 
For each $n=1,\dots,N$, define
\[
\xi^{(n)}(h) := h(X^{(n)})\, P_s^{(1,0)} K(X^{(n)},\cdot),
\]
which are independent and bounded for each fixed $h\in\mathcal{F}$.
Let $u:=\mathcal{T}h$, so that $h=P_s u$. Then this representation induces a sequence of random linear operators
\[
\xi^{(n)}:\mathcal{H}_K\to\mathcal{H}_K,\qquad
\xi^{(n)}(u)
= \langle u, P_s^{(1,0)}K(X^{(n)},\cdot)\rangle_{\mathcal{H}_K}\, P_s^{(1,0)} K(X^{(n)},\cdot),
\]
which are independent, self-adjoint, and uniformly bounded in operator norm.

Applying Bernstein's inequality \cite{minsker2017some} for self-adjoint operators to the centered random operators
$\xi^{(n)} - \mathbb{E}\xi^{(n)}$, we obtain the following uniform error bound. 
The proof follows from a standard application of this Bernstein's inequality and is deferred to the supplementary material.

\begin{proposition}[Uniform estimation error bounds]
\label{prop:uniform_estimation}
Let $\gamma\in(0,1)$, $S>0$, and $\lambda>0$. Then, for any $\delta\in(0,1)$, with probability at least $1-\delta$, we have
\[
\|\widehat{\mathcal T}_{\lambda,N}-\mathcal T_\lambda\|_{\mathcal F_S^\gamma}
\le
\frac{2S}{\lambda}\,
(C_{d+s}^dC\kappa)^{2\gamma+2}
\left(
\sqrt{\frac{8L_\delta}{N}}
+
\frac{4L_\delta}{3N}
\right),
\]
where  $L_\delta:=\log\frac{14r_N}{\delta}$ and $r_N:=r\left(\sum_{n=1}^N\mathbb E\bigl[(Z^{(n)})^2\bigr]\right)$. Here $Z^{(n)} := \xi^{(n)} - \mathbb{E}\xi^{(n)}$ are the centered random operators 
defined in the proof, and $r(A) := \mathrm{tr}(A)/\|A\|$ 
denotes the effective rank.
\end{proposition}

The approximation error is uniformly small over $\mathcal{F}_S^\gamma$ whenever the regularization parameter $\lambda$ is small. To link this to the sample size $N$ and obtain an explicit convergence rate, we choose a data-dependent regularization parameter of the form
\begin{align}\label{dyn-sca}
\lambda \propto N^{-\alpha}, \qquad \alpha > 0,
\end{align}
meaning that $\lambda$ is of order $N^{-\alpha}$ as $N \to \infty$.
Combining~\eqref{dyn-sca} with the bound~\eqref{app-err}, we see that the approximation error decays at rate
\begin{align}\label{approximation error}
\|\mathcal{T}_{\lambda} - \mathcal{T}\|_{\mathcal{F}_S^\gamma}
\;\leq \; S N^{-\alpha \gamma}, \quad \text{as}\quad N \to \infty.
\end{align}

Finally, by combining this approximation error bound with the uniform estimation error bounds in Proposition~\ref{prop:uniform_estimation}, we obtain uniform bounds on the reconstruction error for the empirical operator estimator~\eqref{empirical operator estimator}.
\begin{theorem}[{\bf Convergence upper rate of the total reconstruction error}]
\label{Tot-Rec-Err}
Let $\widehat{\mathcal{T}}_{\lambda,N}$ be the operator estimator defined in \eqref{empirical operator estimator}. Assume the regularization parameter satisfies \eqref{dyn-sca}. Then for all $\alpha\in(0,\frac{1}{2})$, and for any $0<\delta<1$, with probability as least $1-\delta$, it holds uniformly that
\begin{align*}
\|\widehat{\mathcal{T}}_{\lambda,N} - \mathcal{T}\|_{\mathcal{F}_S^\gamma}
:= \sup_{h\in \mathcal{F}_S^\gamma} 
\|\widehat{\mathcal{T}}_{\lambda,N}h - \mathcal{T}h\|_{\mathcal{H}_K}
\;\leq\; 
C(d,\gamma,\delta,S)N^{-\min\left\{\alpha\gamma, \tfrac{1}{2}(1-2\alpha)\right\}},
\end{align*}
where $\mathcal{F}_S^\gamma$ denotes the source space defined in~\eqref{source space} for $\gamma \in (0,1)$ and $S>0$. The constant $C$ is given by
\[
C(d,\gamma,\delta,S)
= \max\left\{
S,\;
2\sqrt{8\log(7r_N/\delta)}\,
\bigl(C_{d+s}^d\, C\, \kappa\bigr)^{2\gamma+2}\, S
\right\},
\]
where $r_N$ is specified in Proposition~\ref{prop:uniform_estimation}.
\end{theorem}

\section{Numerical experiments}\label{Numerical Experiments}

We now present several numerical examples to evaluate the performance of the proposed method. 
We begin with a general numerical scheme for solving the linear boundary value problem \eqref{pde}. 
In Section~\ref{Darcy Flow: Formulation and Experiment}, we apply the method to the Darcy flow problem, and in Section~\ref{A Comparison with Neural Operator Learning}, we compare our approach with the operator method proposed in \cite{stepaniants2023learning}. Additional numerical results can be found in the supplementary material.

\noindent{\bf Experimental setup.}\quad 
We uniformly sample $N_1$ points in the domain $D$ and $N_2$ points on the boundary $\partial D$, and denote the total number of samples by $N = N_1 + N_2$. 
Throughout the experiments, we use the Gaussian kernel
\begin{align*}
K_{\eta}(x,y)
=
\exp\!\left(-\frac{\|x-y\|^2}{\eta^2}\right),
\qquad
x,y\in\mathbb{R}^d.
\end{align*}
Given the Gaussian kernel $K_{\eta}$ and the sample points $\{X_i\}_{i=1}^N$, where $X^{(i)}\in D$ for $1\leq i\leq N_1$ and $X^{(i)}\in \partial D$ for $N_1+1\leq i\leq N$, we define the kernel basis as
\begin{align*}
(\psi_1,\cdots,\psi_N)
:=
\big(P_s^{(1,1)}K_\eta(X_N,X_N) + \lambda N I\big)^{-1}
P_s^{(1,0)}K_\eta(X_N,\cdot).
\end{align*}

\noindent{\bf Testing procedure.}\quad 
To assess the accuracy of the method, we generate a family of test functions $\{u_k\}_{k=1}^{M}$, which serve as exact solutions to the underlying linear PDE \eqref{pde} (see Sections~\ref{Darcy Flow: Formulation and Experiment} and~\ref{A Comparison with Neural Operator Learning} for details). 
For each $u_k$, we compute the corresponding source term $f_k$ in $D$ and boundary data $g_k$ on $\partial D$ from the governing equation. 
Given $(f_k,g_k)$, the estimator $\widehat{u}_{k,\lambda,N}$ is obtained using the kernel basis expansion:
\begin{align*}
\widehat{u}_{k,\lambda,N}
=
\sum_{i=1}^{N_1} f_k\big(X^{(i)}\big)\,\psi_i
+
\sum_{i=N_1+1}^{N} g_k\big(X^{(i)}\big)\,\psi_i.
\end{align*}

\noindent{\bf Error metrics.}\quad 
The performance of the estimator is quantified by the relative $L^2$ and $L^\infty$ errors, averaged over the test set:
\begin{align}\label{relative error}
S_{2}=
\frac{1}{M}\sum_{k=1}^{M}
\frac{\|\widehat{u}_{k,\lambda,N}-u_k\|_{L^2}}
{\|u_k\|_{L^2}},\quad
S_{\infty}=
\frac{1}{M}\sum_{k=1}^{M}
\frac{\|\widehat{u}_{k,\lambda,N}-u_k\|_{L^\infty}}
{\|u_k\|_{L^\infty}}.
\end{align}

\subsection{Darcy flow: formulation and experiment} \label{Darcy Flow: Formulation and Experiment}

We consider the two-dimensional Darcy flow problem \cite{gilbarg1998elliptic}, formulated as the following elliptic boundary value problem:
\[
\begin{cases}
-\nabla \cdot (a \nabla u) = f, \quad & \text{in } D = (0,1)^2, \\
u = g, \quad & \text{on } \partial D,
\end{cases}
\]
where $u$ denotes the piezometric head, $a$ is a function that determines the permeability of the porous medium, $f$ represents sources and sinks of the fluid, and $g$ specifies the Dirichlet boundary condition.

We take the setting in which the permeability $a$ is fixed and investigate the associated input--output mapping $(f,g) \mapsto u$. In the numerical experiments, we test three different permeability functions:
\begin{align}\label{permeability function}
a_1 = 1, \quad
a_2(x,y) = \exp(x+y), \quad
a_3(x,y) =
\begin{cases}
1, & \lfloor 4x \rfloor + \lfloor 4y \rfloor \text{ is even}, \\
10, & \text{otherwise}.
\end{cases}
\end{align}

\noindent{\bf Training and Testing.}\quad We uniformly sample $N_1 = 2500$ points in $D$ and $N_2 = 1500$ points on $\partial D$. The Gaussian kernel bandwidth is set to $\eta = 1$, and the regularization parameter is chosen as $\lambda = 5 \times 10^{-5}$.
To evaluate the performance of the proposed method, we randomly generate neural networks $\{u_k\}_{k=1}^{M}$ defined on $\overline{D}$ with $M=50$, which serve as ground-truth solutions to the Darcy flow problem. Then we compute the relative $L^2$ and $L^\infty$ errors defined in \eqref{relative error} for different permeability functions in \eqref{permeability function}. In addition, for the case $a_1 = 1$, we compute the solution for discontinuous input functions: $f(x,y) = \begin{cases}
    6,\quad x\geq0.5\\
    4,\quad x<0.5
\end{cases}$ and $g(x,y) = \begin{cases}
(x-0.5)^2+y^2,\quad x\geq 0.5\\
2(x-0.5)^2+y^2,\quad x<0.5
\end{cases}$. 

\begin{table}[htbp]
\centering
\caption{Numerical performance for different permeability functions}\label{Solution operator performance}
\vspace{-0.5em}
\begin{tabular}{c|c|c|c}
\hline
Case & Cost & Relative $L^2$ error & Relative $L^\infty$ error \\ 
\hline
$a_1$ & 1.547s & 2.531$\times 10^{-3}$ & 5.083$\times 10^{-3}$ \\ 
\hline
$a_2$ & 1.519s & 4.454$\times 10^{-3}$ & 8.143$\times 10^{-3}$ \\ 
\hline
$a_3$ & 1.605s & 6.629$\times 10^{-3}$ & 8.272$\times 10^{-3}$ \\ 
\hline
\end{tabular}
\end{table}

\noindent{\bf Results.}\quad 
The relative $L^2$ and $L^\infty$ errors for different permeability functions are reported in Table~\ref{Solution operator performance}. 
Overall, the proposed method achieves consistently low approximation errors across all test cases, demonstrating robustness with respect to variations in the permeability field. 
As expected, the errors remain on the order of $10^{-3}$, indicating stable and accurate performance. 
Moreover, the computational time is consistently below 2 seconds, highlighting the efficiency of the proposed approach and its favorable scalability with respect to the heterogeneity of the coefficient. 
Figures \ref{case2} illustrate the performance of the method for discontinuous input functions. It shows that the predicted solution closely matches the ground truth, with small and structured errors, demonstrating robustness and accuracy across varying levels of regularity in the data.

\begin{figure}[htp]
\centering

\begin{minipage}{0.32\textwidth}
\centering
\includegraphics[width=\textwidth]{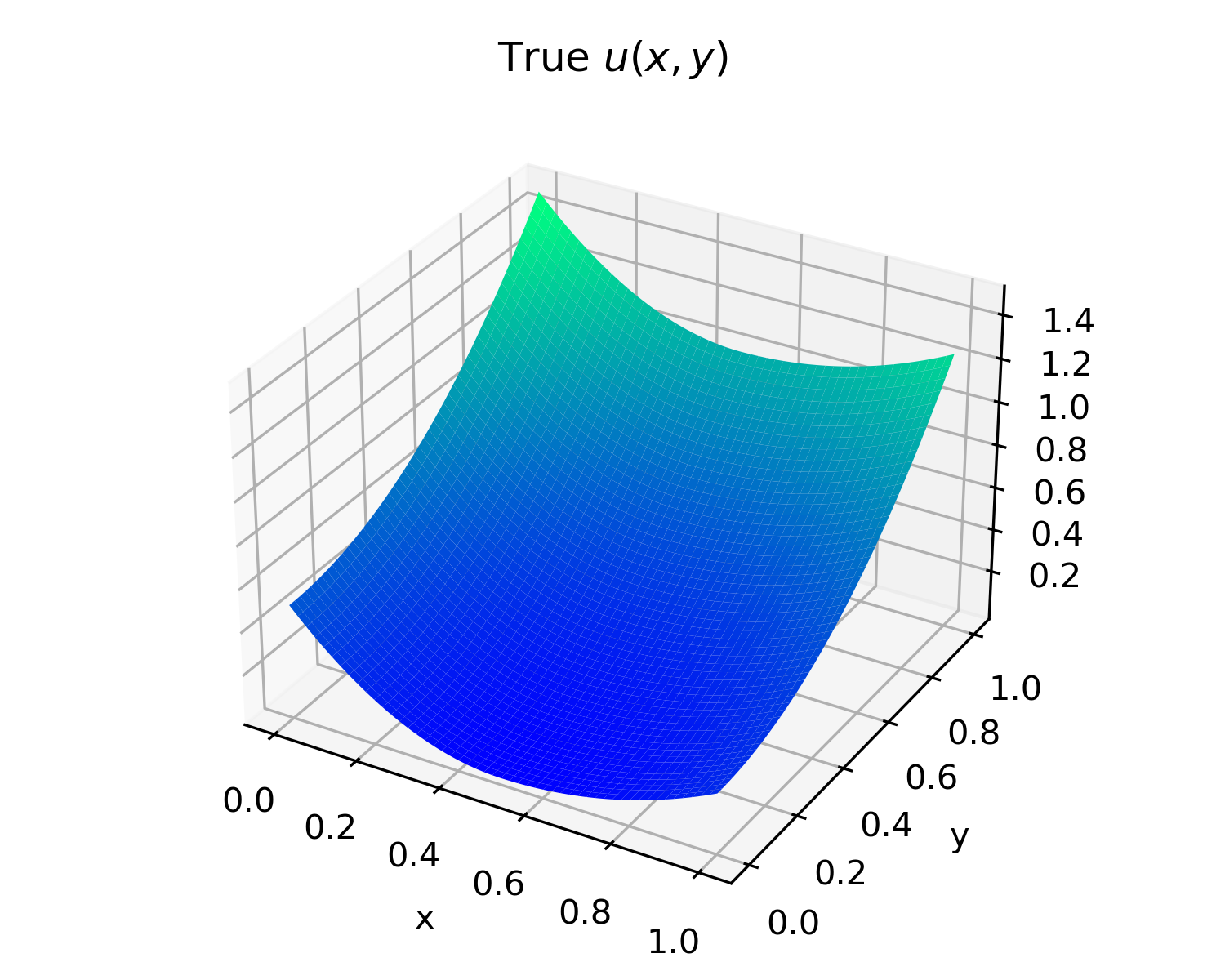}
(a) true solution
\end{minipage}
\hfill
\begin{minipage}{0.32\textwidth}
\centering
\includegraphics[width=\textwidth]{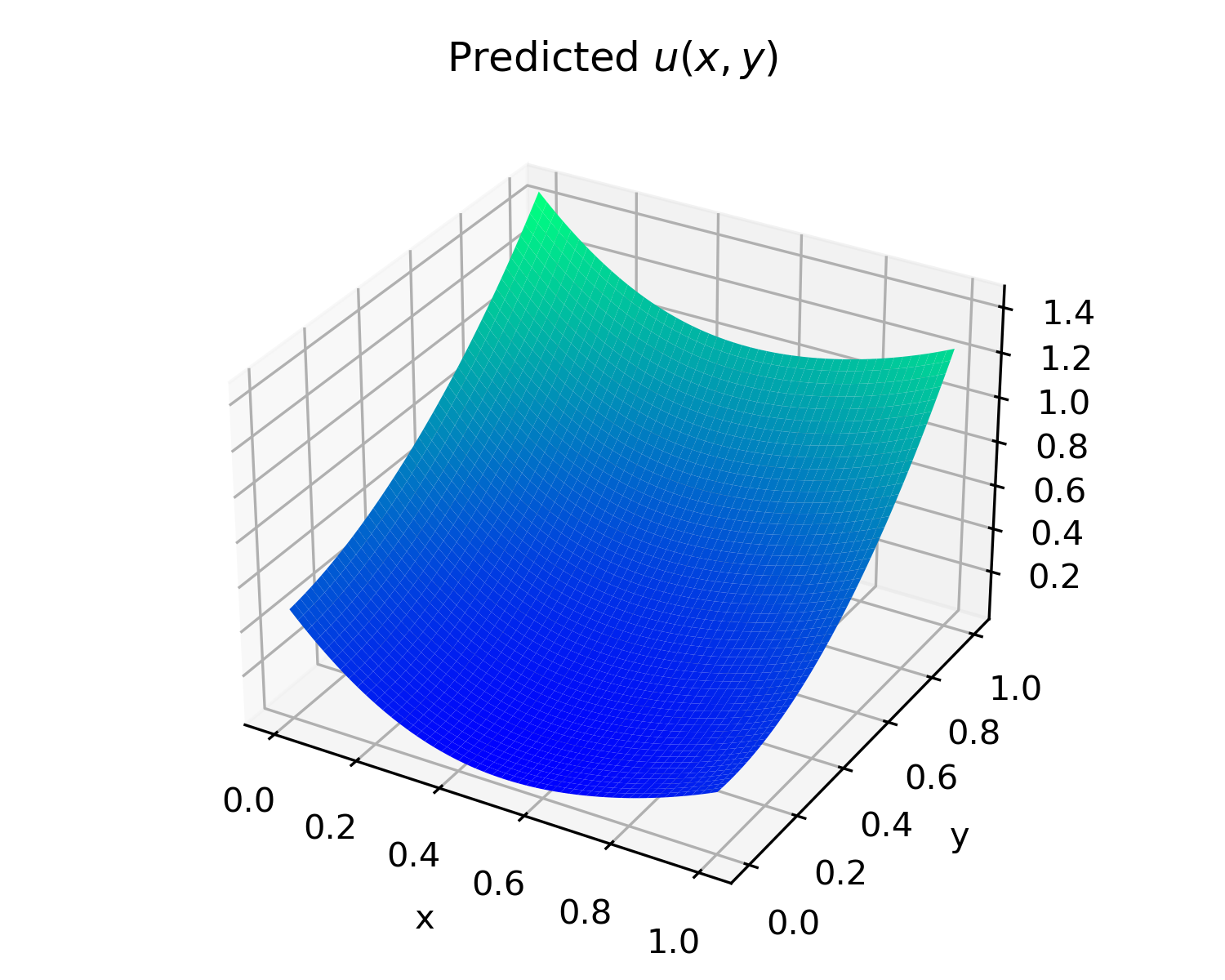}
(b) predicted solution
\end{minipage}
\hfill
\begin{minipage}{0.32\textwidth}
\centering
\includegraphics[width=\textwidth]{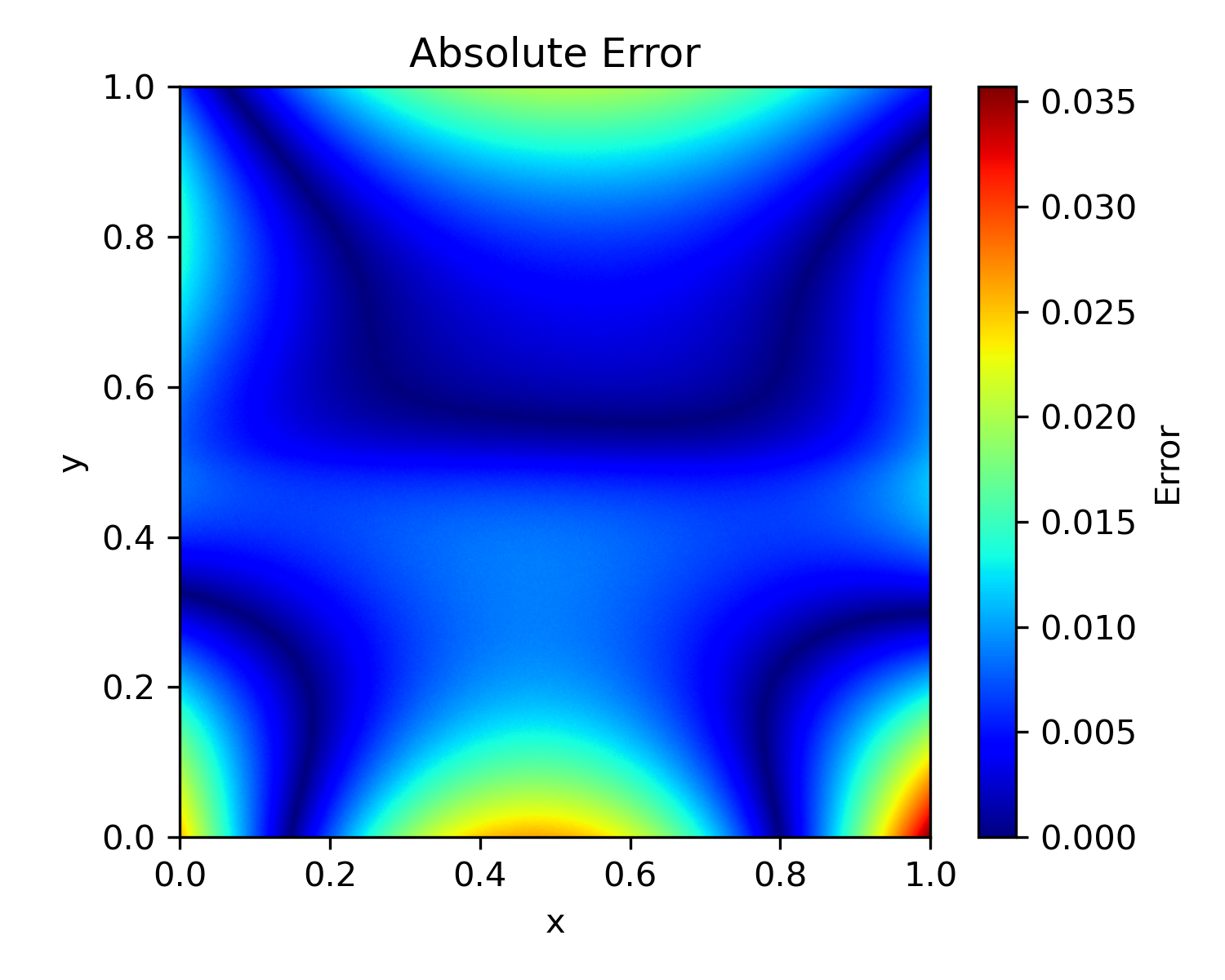}
(c) absolute error
\end{minipage}

\caption{
Input functions:
$f(x,y)=
\begin{cases}
6, & x \geq 0.5,\\
4, & x < 0.5,
\end{cases}
\quad
g(x,y)=
\begin{cases}
(x-0.5)^2 + y^2, & x \geq 0.5,\\
2(x-0.5)^2 + y^2, & x < 0.5.
\end{cases}$.
}
\label{case2}
\end{figure}

\subsection{Comparison with Green Operator Learning: Helmholtz Equation}
\label{A Comparison with Neural Operator Learning}

In this subsection, we compare the proposed approach with the Green operator learning method introduced in \cite{stepaniants2023learning}. 
To illustrate the comparison, we consider the one-dimensional Helmholtz equation
\begin{align}\label{Helmholtz equation}
-\Delta u - \omega^2 u = f, \quad \text{in } D = (0,1),
\end{align}
subject to Dirichlet boundary conditions $u(0) = -0.1$ and $u(1) = 0.1$. 
We investigate both low and high-frequency regimes, with $\omega = 20$ and $\omega = 200$. 
The objective is to learn the solution operator $f \mapsto u$.

The method in \cite{stepaniants2023learning} aims to learn the Green operator from paired input--output data. 
For a fair comparison, we evaluate both approaches in terms of relative $L^2$ and $L^\infty$ errors, as well as computational cost, thereby assessing accuracy, efficiency, and generalization performance.

\noindent{\bf Training and testing.}\quad
For the Green operator learning method with $\omega = 20$, we follow the setup in \cite{stepaniants2023learning}, including data generation and model configuration, and train for 500 epochs. 
For $\omega = 200$, we adjust the hyperparameters to stabilize training: the kernel bandwidth is reduced to $0.005$, the regularization parameter is increased to $10^{-4}$, and the learning rate is set to $10^{-3}$.
For the proposed method, we uniformly sample $N = 1000$ points in $(0,1)$. 
The Gaussian kernel bandwidth is set to $\eta = 0.1$ for $\omega = 20$ and $\eta = 0.01$ for $\omega = 200$, with regularization parameter $\lambda = 10^{-7}$.

For evaluation, we construct the test set (high-frequency functions) of size $M = 100$ as follows: 
\[
u_k(x) = -0.1 + 0.2x + x(1-x)\sin(\omega_k x + \phi),
\]
where $\phi \sim \mathrm{Unif}(0,2\pi)$ and 
$\omega_k = \omega(0.7 + 0.6U)$ with $U \sim \mathrm{Unif}(0,1)$. Here $\omega=20$ or $\omega=200$ is the parameter in \eqref{Helmholtz equation}.

\begin{table}[htbp]
\centering
\caption{Comparison of numerical performance for the Helmholtz equation.}
\label{Comparison1}
\begin{tabular}{c c l c c}
\hline
Case & Method & Cost & Relative $L^2$ error & Relative $L^\infty$ error \\ 
\hline
$\omega=20$ & Ours & 1.533 s & $4.278\times 10^{-3}$ & $8.525\times 10^{-3}$ \\ 
            & Green Operator & 0.721 h & $2.084\times 10^{-2}$ & $4.120\times 10^{-2}$ \\ 
\hline
$\omega=200$ & Ours & 1.670 s & $1.610\times 10^{-3}$ & $1.609\times 10^{-3}$ \\ 
             & Green Operator & 0.745 h & $9.138\times 10^{-1}$ & $9.603\times 10^{-1}$ \\ 
\hline
\end{tabular}
\end{table}

\noindent{\bf Results.}\quad 
Table~\ref{Comparison1} reports averaged relative errors over high-frequency test functions, and Figure~\ref{w200} provides representative examples. In the low-frequency regime ($\omega=20$), both methods qualitatively capture the solution, but KO achieves significantly higher accuracy, with errors on the order of $10^{-3}$ compared to $10^{-2}$ for the Green Operator method. In addition, our method requires only seconds, whereas the Green Operator method takes approximately $0.7$ hours.
But for the high-frequency regime ($\omega=200$), the performance gap becomes pronounced. Our method maintains errors at the $10^{-3}$ level, whereas the Green Operator method deteriorates severely, with errors approaching $1$. Figure~\ref{w200} shows that the Green Operator method fails to resolve high-frequency structures, even when the training error is small, whereas our method remains accurate. Overall, the proposed kernel operator method provides a fast, accurate, and robust framework, with clear advantages in high-frequency regimes.

\begin{figure}[htbp]
\centering
\begin{minipage}[t]{0.48\textwidth}
    \centering
    \includegraphics[width=\linewidth]{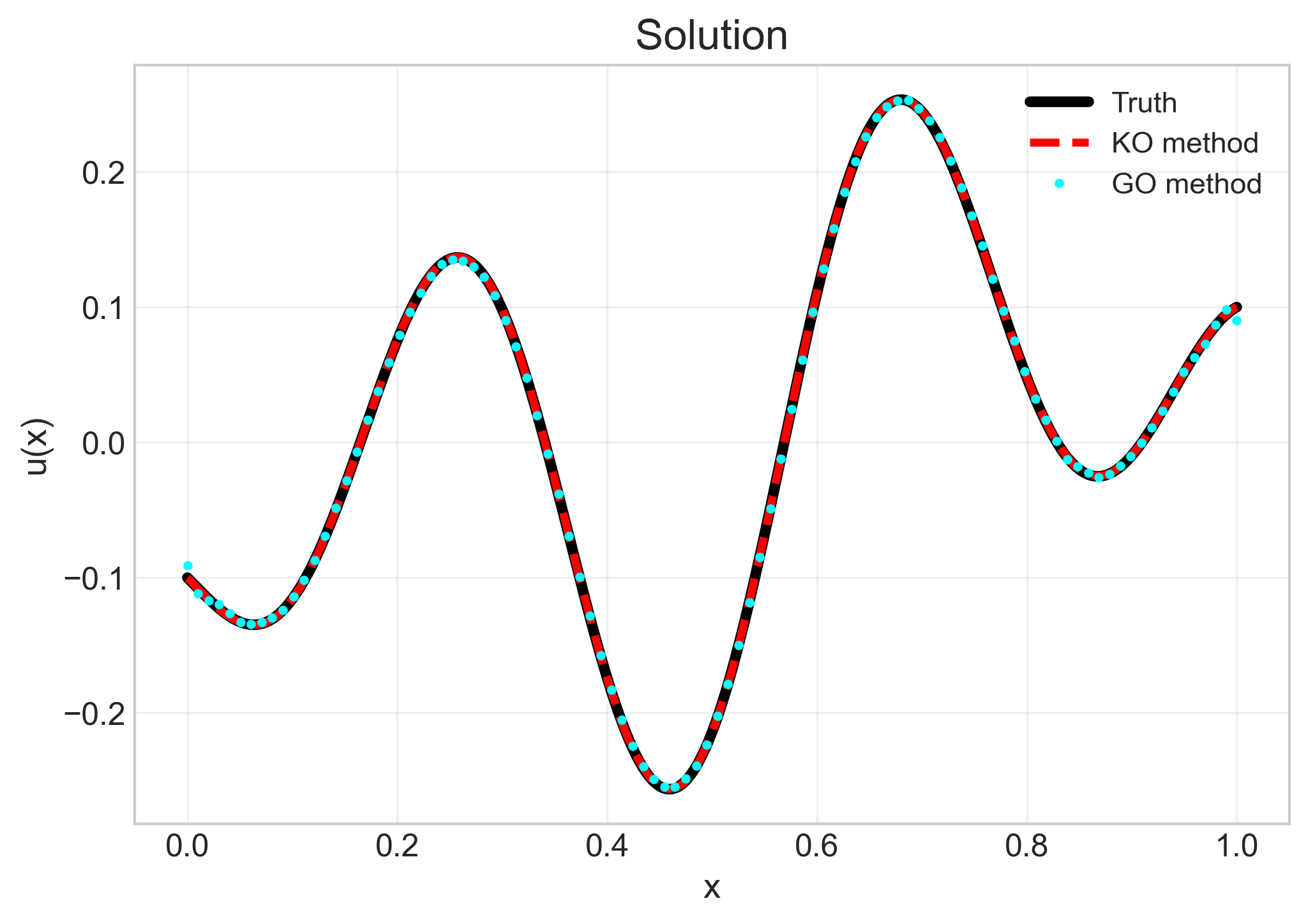}
    \vspace{-2pt}
    {\small (a) $\omega=20 $}
\end{minipage}
\hfill
\begin{minipage}[t]{0.48\textwidth}
    \centering
    \includegraphics[width=\linewidth]{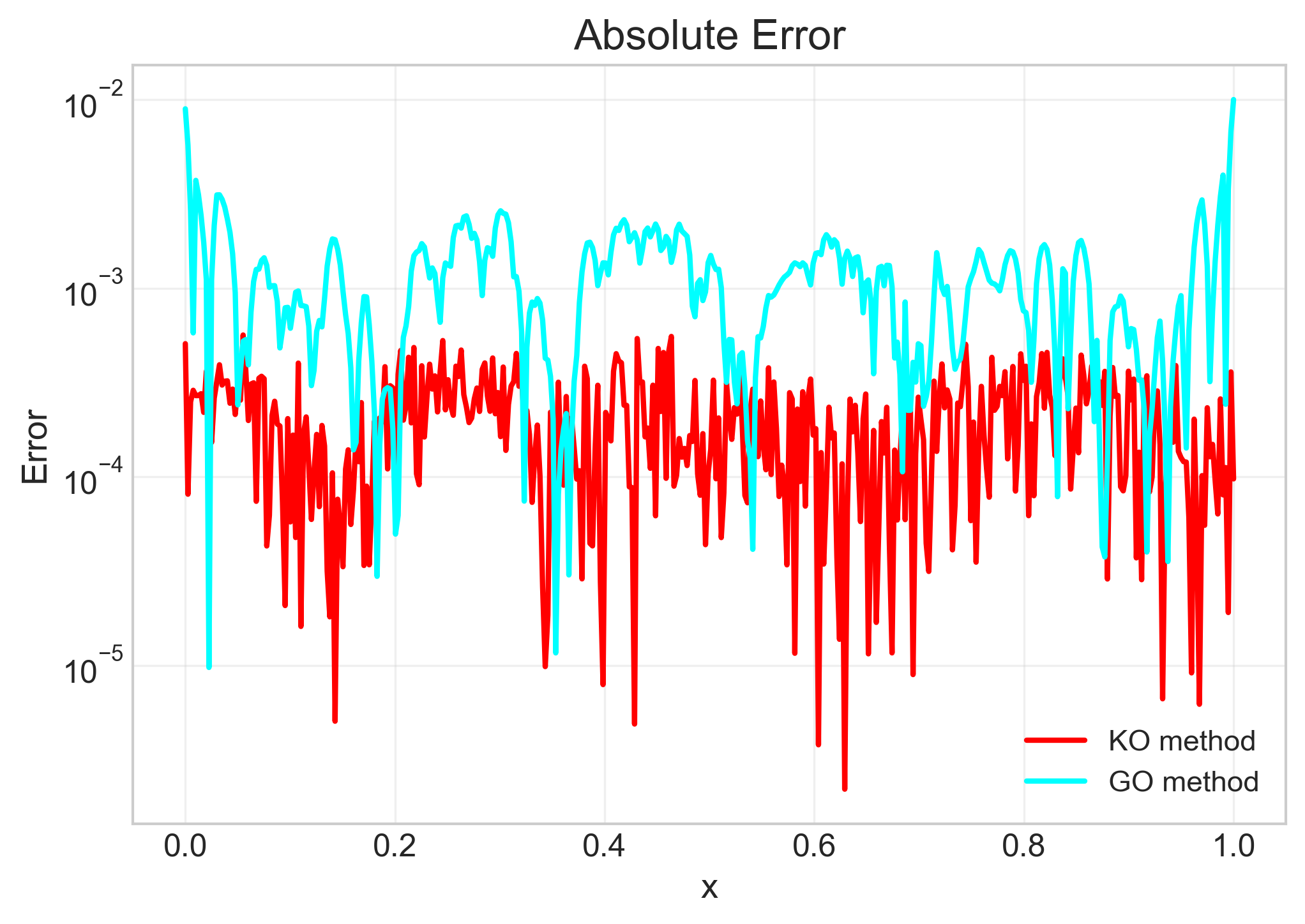}
    \vspace{-2pt}
    {\small (b) $\omega=20 $}
\end{minipage}
\vspace{6pt}
\begin{minipage}[t]{0.48\textwidth}
    \centering
    \includegraphics[width=\linewidth]{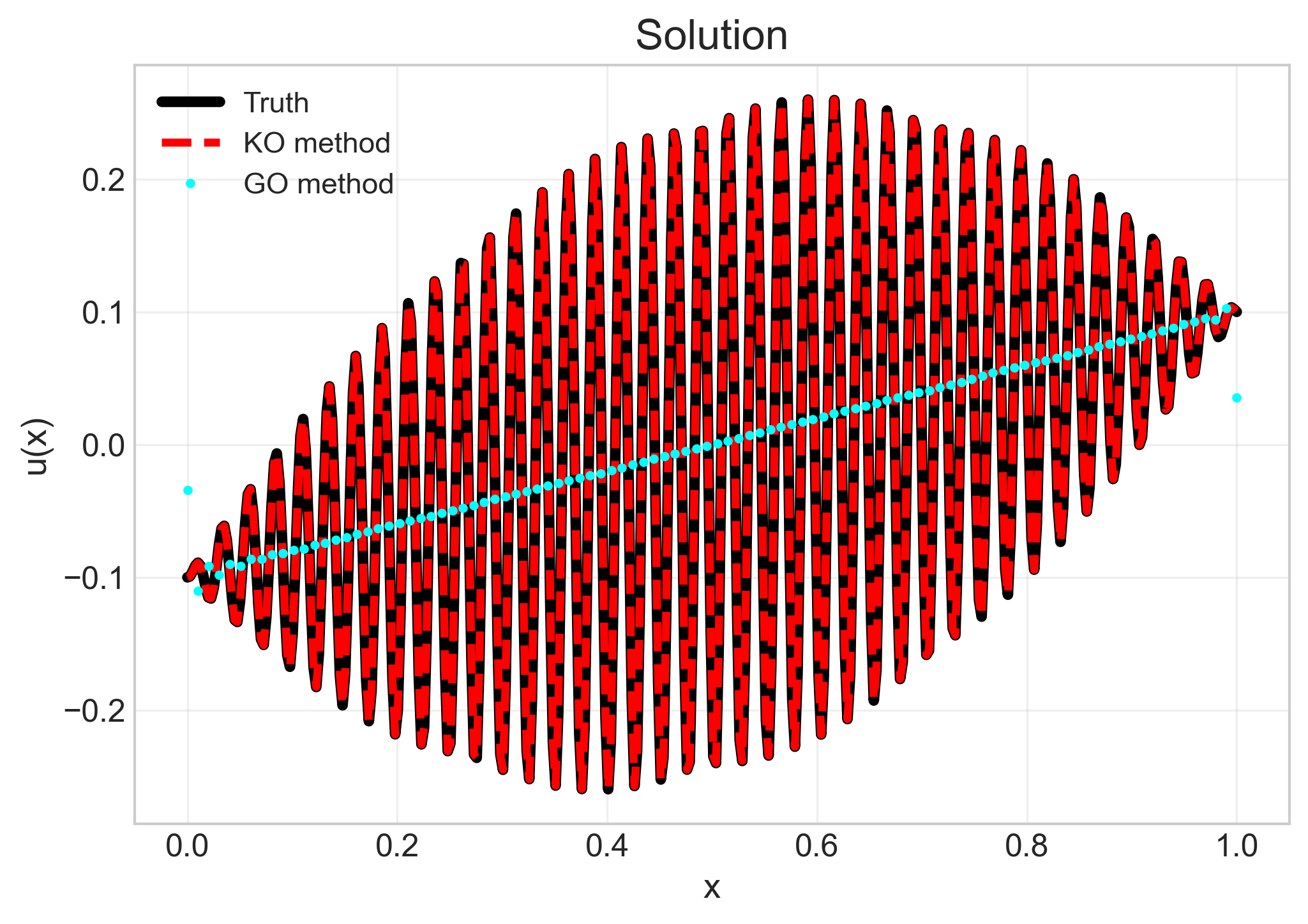}
    \vspace{-2pt}
    {\small (c) $\omega=200 $}
\end{minipage}
\hfill
\begin{minipage}[t]{0.48\textwidth}
    \centering
    \includegraphics[width=\linewidth]{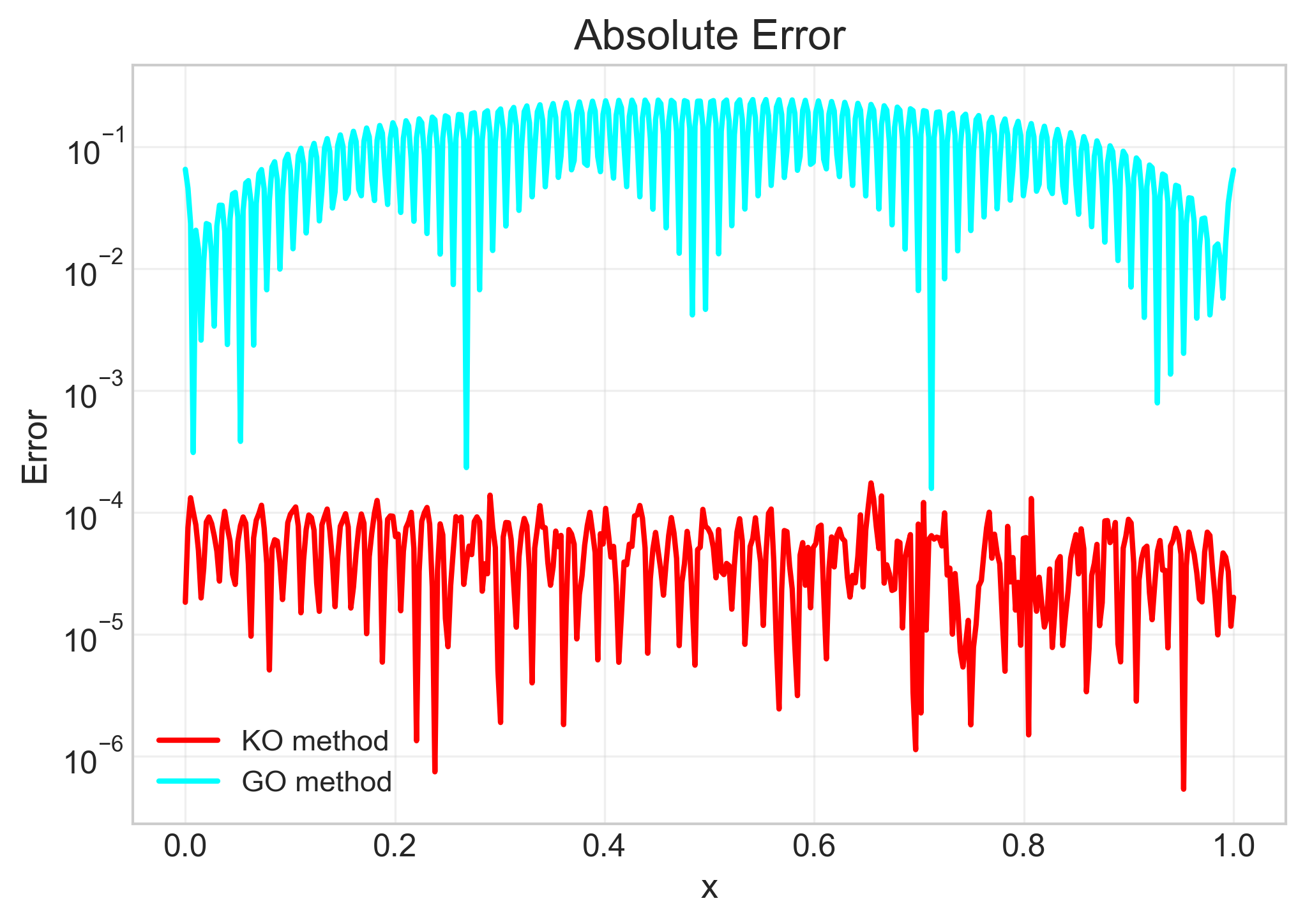}
    \vspace{-2pt}
    {\small (d) $\omega=200 $}
\end{minipage}
\caption{Helmholtz equation.
(a,c) True solutions with Kernel Operator (KO) and Green Operator (GO) approximations for $\omega=20 $ and $\omega=200$, respectively;
(b,d) corresponding absolute errors.}
\label{w200}
\end{figure}

\section{Conclusions}
\label{sec:conclusions}

In this paper, we have introduced a physics-informed kernel framework for learning the solution operator of general linear boundary value problems of the form~\eqref{pde}. By modeling the solution in a reproducing kernel Hilbert space associated with a universal kernel and incorporating the differential operator $\mathcal{L}_s$ and the boundary operator $\mathcal{B}$ directly into a regularized empirical risk, we obtained a learning problem whose unique minimizer admits a closed-form expression through a generalized representer theorem. A distinctive feature of this formulation is that the resulting estimator is independent of the input function $h$, which allows us to lift the classical kernel regression point of view from the approximation of individual PDE solutions to the construction of an operator-based solver $\widehat{\mathcal{T}}_{\lambda,N}\colon L^2(\overline{D};\mu)\to\mathcal{H}_K$. 

We developed a full error analysis by decomposing the total reconstruction error into estimation and approximation components. High-probability pathwise bounds were established for the estimation error, and uniform approximation error bounds of order $\mathcal{O}(\lambda^{\gamma})$ were derived on source spaces $\mathcal{F}^{\gamma}_{S}$ that were shown to form a dense partition of the effective domain $\mathcal{F}_{\mathrm{eff}}$. Combining these results with a data-dependent regularization rule $\lambda\sim N^{-\alpha}$ yielded the uniform convergence rate $N^{-\min\{\alpha\gamma,\,(1-2\alpha)/2\}}$, providing, to the best of our knowledge, the first complete statistical characterization of a kernel-based operator learner that does not rely on paired input--output functional data.

Numerical experiments on the Darcy flow problem and the Helmholtz equation confirm the practical relevance of these theoretical guarantees. In direct comparison with the Green operator learning approach of \cite{stepaniants2023learning} on the Helmholtz benchmark, the kernel-based operator achieved substantially smaller errors at a fraction of the computational cost and, crucially, did not suffer from the generalization gap typically exhibited by purely data-driven operator learning schemes. This behavior is consistent with our analysis: because the estimator inherits the structure of the governing equation through the operator $P_s$, it enables systematic extrapolation beyond any fixed training distribution of input functions.

Several directions remain open for future investigation. The extension of the framework to nonlinear PDEs through structure-preserving kernels in the spirit of \cite{RCSP2,RCSP3,hu2025kernel}, and to equations posed on manifolds along the lines of \cite{fuselier2012scattered} are natural next steps. It would also be of interest to combine the online update scheme of Remark~\ref{Online Regression with Kernels} with randomized low-rank techniques in order to further improve scalability, and to sharpen the convergence rate under stronger source conditions or additional spectral assumptions on $B_s$. We believe that the operator-theoretic perspective developed here offers a principled and computationally attractive bridge between classical kernel methods and modern operator learning, and opens the door to a broader class of physics-informed, structure-preserving surrogate solvers for parametric PDE problems arising in science and engineering.

\section*{Acknowledgments}
The authors thank Lyudmila Grigoryeva and Daiying Yin for helpful discussions and remarks and acknowledge financial support from the School of Physical and Mathematical Sciences of the Nanyang Technological University. 

\bibliographystyle{siamplain}
\bibliography{references}
\end{document}


\maketitle

\section{Detailed proofs of selected propositions}\label{Detailed proofs of propositions and lemmas}
In this section, we provide detailed proofs of various propositions stated in the main text that mimic existing results in the literature and that we gather in this supplement for the sake of completeness. In the following, $C_{d+s}^d=\binom{d+s}{d} = \frac{(d+s)!}{d! \, s!} $ is the binomial coefficient, $\kappa^2=\|K\|_{C_b^{2s}(\overline{D} \times \overline{D})}$, and $C$ is a uniform bound for the coefficients in the linear differential operator $P_s$.

\begin{proof}[Proof of Proposition \ref{main-Wel-Ope}]
By the differential reproducing property \cite[Theorem 2.7]{RCSP2}, the regularity condition $K\in C_b^{2s+1}(\overline{D}\times\overline{D})$ implies that the corresponding RKHS $\mathcal{H}_K\subseteq C_b^s(\mathbb{R}^{d})$ and that the operator $P_s$ is well-defined as a map $P_s: {\mathcal H}_K \longrightarrow L^2(\overline{D};\mu)$. 
Indeed, for any  $h\in \mathcal{H}_K$, the differential reproducing property \cite[Theorem 2.7]{RCSP2} implies that
{\small
\begin{equation*}
\begin{aligned}
\|P_s h\|^2_{L^2(\mu)} &=\int_{\mathbb{R}^{d}}\bigg|\sum_{\alpha \in I_s} \phi_{\alpha} D^{\alpha} h(x)\bigg|^2\mathrm{d}\mu(x) =  \sum_{\alpha \in I_s}\sum_{\beta \in I_s}\int_{\mathbb{R}^{d}}\phi_\alpha(x)\phi_\beta(x)  D^{\alpha} h(x)D^{\beta} h(x)\mathrm{d}\mu(x)\\
&\leq C^2 \sum_{\alpha \in I_s}\sum_{\beta \in I_s}\|h\|_{C_b^{|\alpha|}}\;\|h\|_{C_b^{|\beta|}}  \leq (C_{s+d}^dC\kappa)^2\|h\|_{\mathcal{H}_K}^2,
\end{aligned}    
\end{equation*}}
which shows that $P_s: {\mathcal H}_K \longrightarrow L^2(\overline{D};\mu)$ is a bounded linear operator and that $\|P_s\|\leq C_{s+d}^d ~ C\kappa$.

Next, we prove \eqref{main-adjoint}. For any
$h \in \mathcal{H}_K$ and any $g\in L^2(\overline{D};\mu)$,
{\small
\begin{multline*}
\langle P_s h,g\rangle_{L^2(\mu)}  = \sum_{\alpha \in I_s}\int_{\mathbb{R}^{d}}  \phi_\alpha(x)D^\alpha h(x)g(x)  \mathrm{d} \mu(x)\\
=\sum_{\alpha \in I_s}\int_{\mathbb{R}^{d}}  \phi_\alpha(x)\left\langle h, (D^\alpha K)_x\right\rangle_{\mathcal{H}_K}g(x)  \mathrm{d} \mu(x)
=\bigg\langle h,\sum_{\alpha \in I_s}\int_{\mathbb{R}^{d}}  \phi_\alpha(x) (D^\alpha K)_xg(x)  \mathrm{d} \mu(x)\bigg\rangle_{\mathcal{H}_K}\\
=\left\langle h, \int_{\mathbb{R}^{d}} g(x) P_s^{(1,0)}K(x,\cdot) \, \mathrm{d} \mu(x) \right\rangle_{\mathcal{H}_K},
\end{multline*} }
where the second equality is due to the differential reproducing property \cite[Theorem 2.7]{RCSP2}. Since $g\in L^2(\overline{D};\mu)$ in the previous equality is arbitrary, we have hence shown that \eqref{main-adjoint} holds. 
Since $B_s=P_s^{\ast}P_s$, $B_s$ is clearly a bounded linear operator. Equation \eqref{main-positive} follows from \eqref{main-adjoint} by direct calculation and the fact that the integral commutes with the scalar product. 

We now prove that $B_s$ is a trace class operator, that is, we show that $\operatorname{Tr}(|B_s|)<\infty$, where $|B_s|=\sqrt{B_s^* B_s}$. Since $B$ is positive semidefinite, we have that $|B_s|=B_s$. Therefore, it is equivalent to show that $\operatorname{Tr}(B_s)<\infty$. In order to do that, we choose a spanning orthonormal set $\left\{e _n\right\} _{n \in \mathbb{N}}$ for ${\mathcal H} _K$ whose existence is guaranteed by \cite[Lemma A.3]{RCSP2} and \cite[Theorem 2.4]{owhadi2017separability}. Then,
$$
\begin{aligned}
\operatorname{Tr}(B_s)&=\operatorname{Tr}\left(P_s^* P_s\right) =\sum_n\left\langle P_s^* P_s e_n, e_n\right\rangle_{\mathcal{H}_K}=\sum_n\left\langle P_s e_n, P_s e_n\right\rangle_{L^2\left(\mu\right)} \\
& =\sum_n\int_{\mathbb{R}^{d}}\bigg(\sum_{\alpha \in I_s} \phi_{\alpha}(x) D^{\alpha} e_n(x)\bigg)\bigg(\sum_{|\beta|\leq s} \phi_{\beta}(x) D^{\beta} e_n(x)\bigg)\mathrm{d}\mu(x) \\
& =\sum_{\alpha \in I_s} \sum_{|\beta|\leq s}\int_{\mathbb{R}^{d}}\phi_{\alpha}(x) \phi_{\beta}(x) \sum_n\Big(D^{\alpha} e_n(x)  D^{\beta} e_n(x)\Big)\mathrm{d}\mu(x)\\
& = \sum_{\alpha \in I_s} \sum_{|\beta|\leq s}\int_{\mathbb{R}^{d}}\phi_{\alpha}(x) \phi_{\beta}(x) D^{(\alpha,\beta)}K(x,x)\mathrm{d}\mu(x) \leq (C_{d+s}^d)^2C^2\kappa^2,
\end{aligned}
$$
where the sixth equality is due to 
\begin{align*}
\sum_n &D^\alpha e_n(x)D^\beta e_n(x)=
\sum_n D^\alpha e_n(x)\langle(D^\beta K)_x,e_n\rangle_{\mathcal{H}_K}=\sum_n \langle (D^\beta K)_x,D^\alpha e_n(x)e_n\rangle_{\mathcal{H}_K}\\
&= \Big\langle (D^\beta K)_x,\sum_n D^\alpha e_n(x)e_n\Big\rangle_{\mathcal{H}_K} = \Big\langle (D^\beta K)_x,\sum_n \langle(D^\alpha K)_x,e_n\rangle_{\mathcal{H}_K}e_n\Big\rangle_{\mathcal{H}_K}\\
&=\langle (D^\beta K)_x, (D^\alpha K)_x\rangle_{\mathcal{H}_K}=D^{(\alpha,\beta)}K(x,x).    
\end{align*}
Finally, the form of the operator $B_s=P_s^{*}P_s$ automatically guarantees that it is positive semidefinite. The result follows.
\end{proof}

\begin{lemma}\label{Dec-Omp}
Let $K \in C_b^{2s+1}( \overline{D} \times  \overline{D})$ be a Mercer kernel. Then, for any $h \in \mathcal{F}$ and any $0 < \delta < 1$, with probability at least $1-\delta$,
\small
\begin{align}\label{bound BsN}
\left\|
\frac{1}{\sqrt{N}}P_{s,N}^* \pi_N h - P_s^* h
\right\|_{\mathcal{H}_K}
\;\le\;
\left(
\sqrt{\frac{8\log(2/\delta)}{N}} + 1
\right)
\sqrt{\frac{2\log(2/\delta)}{N}}\,
(C_{d+s}^d\, C\, \kappa)^2\,  \|\mathcal{T}h\|_{\mathcal{H}_K}.
\end{align}
\normalfont
\end{lemma}
\begin{proof}
Firstly, note that for any $h\in \mathcal{F}$, Proposition \ref{main-Wel-Ope} yields that
\begin{align*}
\|h\|_{\infty} = \|P_s\mathcal{T}h\|_{\infty}  \leq C_{d+s}^d C\kappa \|\mathcal{T}h\|_{\mathcal{H}_K}.
\end{align*}
Since $K \in C_b^{2s+1}( \overline{D} \times  \overline{D})$, the differential reproducing property \cite{RCSP2} yields
\small
\begin{multline*}
\|P_{s,N}^* \pi_N h\|_{\mathcal{H}_K}^2
= \frac{1}{N} \big\| h^\top(X_N)\, P_s^{(1,0)} K(X_N,\cdot) \big\|_{\mathcal{H}_K}^2
  = \frac{1}{N} h^\top(X_N)\, P_s^{(1,1)} K(X_N,X_N)\, h(X_N) \\
= \frac{1}{N} 
   \sum_{i=1}^N \sum_{j=1}^N \sum_{\alpha \in I_s} \sum_{\beta \in I_s}
   h(X^{(i)}) \phi_\alpha(X^{(i)})\,
   D^{(\alpha,\beta)} K(X^{(i)},X^{(j)})\,
   \phi_\beta(X^{(j)}) h(X^{(j)}) \\
\le N \bigl(C_{d+s}^d\bigr)^2 C^2 \kappa^2 \|h\|_{\infty}^2\le N \bigl(C_{d+s}^d\bigr)^4 C^4 \kappa^4 \|\mathcal{T}h\|_{\mathcal{H}_K}^2,
\end{multline*}
\normalsize
which shows that $\frac{1}{\sqrt{N}} P_{s,N}^* \pi_N h$ is a bounded $\mathcal{H}_K$-valued random variable. Moreover,
\[
\mathbb{E}\Bigl[\bigl\|\tfrac{1}{\sqrt{N}} P_{s,N}^* \pi_N h\bigr\|_{\mathcal{H}_K}^2\Bigr]
\;\le\; \bigl(C_{d+s}^d\bigr)^4 C^4 \kappa^4 \|\mathcal{T}h\|_{\mathcal{H}_K}^2.
\]
Define now the $\mathcal{H}_K$-valued random variables
\begin{align}\label{xi_h}
\xi^{(n)}(h) :=  h(X^{(n)})\, P_s^{(1,0)} K(X^{(n)},\cdot),
\quad n = 1,\dots,N.    
\end{align}
Then $\{\xi^{(n)}(h)\}_{n=1}^N$ are i.i.d., and
$\frac{1}{\sqrt{N}} P_{s,N}^* \pi_N h - P_s^* h
= \frac{1}{N} \sum_{n=1}^N \bigl(\xi^{(n)}(h) - \mathbb{E}[\xi^{(n)}(h)]\bigr)$.
The desired concentration bound \eqref{bound BsN} follows by applying, for instance, \cite[Lemma~8]{de2005learning} or the results in~\cite{yurinsky1995sums} to the sequence $\{\xi^{(n)}(h)\}_{n=1}^N$.
\end{proof}

\begin{proof}[Proof of Proposition \ref{main-Sam-Err}] For any $h\in \mathcal{F}$, we decompose 
\small
\begin{multline*}
\frac{1}{\sqrt{N}} (B_{s,N} + \lambda I)^{-1} P_{s,N}^{*} \pi_Nh-(B_{s}+\lambda I)^{-1}P_s^*h \\
=\frac{1}{\sqrt{N}} (B_{s,N} + \lambda I)^{-1} P_{s,N}^{*} \pi_Nh-(B_{s,N}+\lambda I)^{-1}P_s^*h + (B_{s,N}+\lambda I)^{-1}P_s^*h- (B_s+\lambda I)^{-1}P_s^*h.
\end{multline*}
\normalsize
Since the operator norm satisfies $\|(B_{s,N}+\lambda I)^{-1}\|\leq \frac{1}{\lambda}$, we have that 
\small
\begin{equation*}
\left\|\frac{1}{\sqrt{N}} (B_{s,N} + \lambda I)^{-1} P_{s,N}^{*} \pi_Nh-(B_{s,N}+\lambda I)^{-1}P_s^*h\right\|_{\mathcal{H}_K}\leq \frac{1}{\lambda} \left\|\frac{1}{\sqrt{N}}  P_{s,N}^{*} \pi_Nh-P_s^*h\right\|_{\mathcal{H}_K}.
\end{equation*}
\normalsize
Applying Lemma \ref{Dec-Omp} to $\frac{1}{\sqrt{N}}  P_{s,N}^{*} \pi_Nh-P_s^*h$, we obtain that, with probability at least $1-\delta/2$,
\begin{multline}\label{bounds1}
\left\|\frac{1}{\sqrt{N}} (B_{s,N} + \lambda I)^{-1} P_{s,N}^{*} \pi_Nh-(B_{s,N}+\lambda I)^{-1}P_s^*h\right\|_{\mathcal{H}_K}\\ \leq \left(
\sqrt{\frac{8\log(2/\delta)}{N}} + 1
\right)
\sqrt{\frac{2\log(2/\delta)}{N\lambda^2}}\,
(C_{d+s}^d\, C\, \kappa)^2\,  \|\mathcal{T}h\|_{\mathcal{H}_K}.
\end{multline}
On the other hand, we have that
\small
\begin{multline*}
\|(B_{s,N}+\lambda I)^{-1}P_s^*h - (B_s+\lambda I)^{-1}P_s^*h\|_{\mathcal{H}_K} = \|(B_{s,N}+\lambda I)^{-1}(B_s-B_{s,N})(B_s+\lambda I)^{-1}P_s^*h\|_{\mathcal{H}_K} \\
\leq \frac{1}{\lambda}\|(B_s-B_{s,N})(B_s+\lambda I)^{-1}P_s^*h\|_{\mathcal{H}_K}.
\end{multline*}
\normalsize
Since $u_{\lambda}^*=(B_s+\lambda I)^{-1}P_s^*h=(B_s+\lambda I)^{-1}B_s\mathcal{T}h$ is the unique minimizer of the regularized statistical risk $R _\lambda(h) =\|P_sh-P_s( \mathcal{T}h)\|^2_{L^2(\mu)}+\lambda\|h\|_{\mathcal{H}_K}^2,$ plugging $h=0$, we obtain that 
\begin{align*}
\|P_s{u_{\lambda}^*}-P_s(\mathcal{T}h)\|^2_{L^2(\mu)}+\lambda\|u_{\lambda}^*\|_{\mathcal{H}_K}^2 <\| P_s(\mathcal{T}h)\|^2_{L^2(\mu)}.
\end{align*}
Then by Proposition \ref{main-Wel-Ope}, we have
\begin{equation}\label{eq2}
\begin{aligned}
\|u_{\lambda}^*\|_{\mathcal{H}_K} <\frac{1}{\sqrt{\lambda}}\| P_s(\mathcal{T}h)\|_{L^2(\mu)}\leq \frac{C_{s+d}^d\kappa C}{\sqrt{\lambda}}\| \mathcal{T}h\|_{\mathcal{H}_K}.
\end{aligned}
\end{equation}
Applying Lemma \ref{Dec-Omp} to $u_{\lambda}^*=(B_s+\lambda I)^{-1}P_s^*h$ and combining it with equation \eqref{eq2}, we obtain that with probability at least $1-\delta/2$,
\begin{equation}\label{bounds2}
\begin{aligned}
\frac{1}{\lambda}\|(B_s-B_{s,N})u_{\lambda}^*\|_{\mathcal{H}_K}
& \leq \left(
\sqrt{\frac{8\log(2/\delta)}{N}} + 1
\right)
\sqrt{\frac{2\log(2/\delta)}{N\lambda^2}}\,
(C_{d+s}^d\, C\, \kappa)^2\,  \|u_{\lambda}^*\|_{\mathcal{H}_K}\\
&\leq \left(
\sqrt{\frac{8\log(2/\delta)}{N}} + 1
\right)
\sqrt{\frac{2\log(2/\delta)}{N\lambda^3}}\,
(C_{d+s}^d\, C\, \kappa)^3\,  \|\mathcal{T}h\|_{\mathcal{H}_K}.
\end{aligned}
\end{equation}
Finally, by combining the bounds \eqref{bounds1} and \eqref{bounds2}, we obtain that with a probability at least $1-\delta$,
\footnotesize
\begin{align*}
\left\|\widehat{\mathcal{T}}_{\lambda,N}h-\mathcal{T}_{\lambda}h\right\|_{\mathcal{H}_K}  
\leq\left(
\sqrt{\frac{8\log(2/\delta)}{N}} + 1
\right)
\sqrt{\frac{2\log(2/\delta)}{N\lambda^2}}\,
(C_{d+s}^d\, C\, \kappa)^2\,  \|\mathcal{T}h\|_{\mathcal{H}_K}\left(1+\frac{C_{s+d}^d\kappa C}{\sqrt{\lambda}}\right).
\end{align*} 
\normalsize
\end{proof}

The following lemma extends Lemma~\ref{Dec-Omp} to establish uniform control. 
\begin{lemma}\label{lemmasm}
Let $K \in C_b^{2s+1}( \overline{D} \times  \overline{D})$ be a Mercer kernel. Then, for any $0 < \delta < 1$, with probability at least $1-\delta$,
\[
\|B_{s,N}-B_s\|
\le
\bigl(C_{d+s}^d\bigr)^2C^2\kappa^2
\left(
\sqrt{\frac{8L_\delta}{N}}
+
\frac{4L_\delta}{3N}
\right),
\]
where $L_\delta:=\log\frac{14r_N}{\delta}$ and $r_N:=r\left(\sum_{n=1}^N\mathbb E\bigl[(Z^{(n)})^2\bigr]\right)$stands  for the effective rank.
\end{lemma}

\begin{proof} From~\eqref{xi_h}, this representation induces a sequence of  random linear operators
\[
\xi^{(n)}:\mathcal{H}_K\to\mathcal{H}_K,\qquad
\xi^{(n)}(u)
= \langle u, P_s^{(1,0)}K(X^{(n)},\cdot)\rangle_{\mathcal{H}_K}\, P_s^{(1,0)} K(X^{(n)},\cdot),
\]
for $u = \mathcal{T}h$. Moreover, these operators are uniformly bounded with operator norm $\|\xi^{(n)}\| \le \bigl(C_{d+s}^d\bigr)^2 C^2 \kappa^2$. Hence, $Z^{(n)}:=\xi^{(n)}-\mathbb{E}\xi^{(n)}$ is a sequence of independent self-adjoint random operators
such that $\mathbb{E}Z^{(n)}=0$, $\|Z^{(n)}\| \le 2\bigl(C_{d+s}^d\bigr)^2 C^2 \kappa^2$ and 
\begin{align*}
\left\|\sum_{n}^N\mathbb{E}(Z^{(n)})^2\right\| \le 4N\bigl(C_{d+s}^d\bigr)^4 C^4 \kappa^4.     
\end{align*}
Notice that $B_{s,N}-B_s=\frac{1}{N}\sum_{n=1}^NZ^{(n)}$. Then by \cite[Theorem 3.1]{minsker2017some}, we obtain that for any $t\geq \frac{1}{6N}(U^2+\sqrt{U+36\sigma^2})$,
\begin{align*}
\mathbb{P}\left(\left\|B_{s,N}-B_s\right\|>t\right)  < 14 r\left(\sum_{n=1}^N\mathbb{E}(Z^{(n)})^2\right)\exp\left\{-\frac{N^2t^2/2}{\sigma^2+NUt/3}\right\} 
\end{align*}
where $U=2\bigl(C_{d+s}^d\bigr)^2 C^2 \kappa^2$, $\sigma^2=4N\bigl(C_{d+s}^d\bigr)^4 C^4 \kappa^4$ and $r(A):=\frac{\operatorname{tr}(A)}{\|A\|}$ stands for the effective rank. Let
\[
r_N:=r\left(\sum_{n=1}^N\mathbb E\bigl[(Z^{(n)})^2\bigr]\right),
\qquad
L_\delta:=\log\frac{14r_N}{\delta}.
\]
Then, with probability at least $1-\delta$, we have
\[
\|B_{s,N}-B_s\|
\le
\bigl(C_{d+s}^d\bigr)^2C^2\kappa^2
\left(
\sqrt{\frac{8L_\delta}{N}}
+
\frac{4L_\delta}{3N}
\right).
\]
\end{proof}

\begin{proof}[Proof of Proposition \ref{main-prop:uniform_estimation}]
Let \(h\in\mathcal F_S^\gamma\). Then there exists
\(u\in\mathcal H_K\) such that $h=P_su$, $u=B_s^\gamma\psi$ and $\|\psi\|_{\mathcal H_K}\le S$.
Following the proof in Proposition \ref{main-Sam-Err}, we decompose
\[
\begin{aligned}
\widehat{\mathcal T}_{\lambda,N}h-\mathcal T_\lambda h
&=
(B_{s,N}+\lambda I)^{-1}
\left(B_{s,N}u -B_su
\right) \\
&\quad+
(B_{s,N}+\lambda I)^{-1}
(B_s-B_{s,N})
(B_s+\lambda I)^{-1}
B_su .
\end{aligned}
\]

By the functional calculus for bounded positive self-adjoint operators,
\[
\left\|
(B_s+\lambda I)^{-1}B_s^{\gamma+1}
\right\|
=
\sup_{\mu\in\sigma(B_s)}
\frac{\mu^{\gamma+1}}{\mu+\lambda}
\le
\sup_{\mu\in\sigma(B_s)}\mu^\gamma
\le
\|B_s\|^\gamma .
\]
Hence, for \(u=B_s^\gamma\psi\) with \(\|\psi\|_{\mathcal{H}_K}\le S\),
\[
\|(B_s+\lambda I)^{-1}B_su\|_{\mathcal{H}_K}
\le
S\|B_s\|^\gamma .
\]

Since $\|u\|_{\mathcal{H}_K}=\|B_s^\gamma\psi\|_{\mathcal{H}_K}\le S\|B_s\|^\gamma$, we obtain that
\begin{align*}
\|\widehat{\mathcal T}_{\lambda,N}h-\mathcal T_\lambda h\|_{\mathcal{H}_K}&\leq \frac{1}{\lambda}  \|B_{s,N} -B_s\| ~\|u\|_{\mathcal{H}_K} + \frac{1}{\lambda}  \|B_{s,N} -B_s\| ~\|(B_s+\lambda I)^{-1}B_su\|_{\mathcal{H}_K}\\
&\leq \frac{2S}{\lambda} \|B_s\|^\gamma ~\|B_{s,N} -B_s\|.
\end{align*}
Applying Lemma \ref{lemmasm} for \(B_{s,N}-B_s\), with probability at least \(1-\delta\),
\[
\|B_{s,N}-B_s\|
\le
\bigl(C_{d+s}^d\bigr)^2C^2\kappa^2\left(
\sqrt{\frac{8L_\delta}{N}}
+
\frac{4L_\delta}{3N}
\right).
\]
The result follows from the fact that $\|B_s\|\le \bigl(C_{d+s}^d\bigr)^2C^2\kappa^2$.
\end{proof}

\section{Online low-rank kernel regression}
In the proposed kernel algorithm, each entry of the generalized Gram matrix $P_s^{(1,1)}K(X_N,X_N)$ requires evaluation at all pairs of data points, making both the computation and storage of all $N^2$ entries potentially burdensome. Performing data analysis tasks with the kernel matrix typically involves solving linear systems, least-squares problems, or computing eigenvalue decompositions.
With direct algorithms, these linear algebra operations incur a computational cost of $\mathcal{O}(N^3)$, which renders kernel methods impractical for large-scale datasets. 

To address these computational challenges, we adopt a low-rank approximation \cite{meanti2020kernel,chen2025randomly} of the kernel operator together with an online regression strategy \cite{RCSP2}. The key idea is to restrict the solution space to a finite-dimensional subspace of the RKHS, thereby reducing the computational complexity, while processing the data in a streaming fashion to avoid storing the full design matrix. This leads to a scalable method that avoids forming the full design matrix while preserving the operator structure.

\paragraph{Low-rank approximation}
Let $Z_L=\{z_1,\dots,z_L\}\subset \overline{D}$ be a set of kernel centers with $L\ll N$. We restrict the representation \eqref{main-rep-ker} to the subspace
\[
\mathcal{H}_L
=
\operatorname{span}\{P_s^{(1,0)}K(z_j,\cdot)\}_{j=1}^L,
\]
and approximate the solution as
\[
u_{N,L}(x)
=
\sum_{j=1}^L c_j\, P_s^{(1,0)}K(z_j,x).
\]

Given data $\{(X^{(i)},Y^{(i)})\}_{i=1}^N$, we define the design matrix
\[
A_N \in \mathbb{R}^{N\times L}, 
\qquad
(A_N)_{ij}
=
P_s^{(1,0)}K(z_j, X^{(i)}),
\]
and the vector $Y_N=(Y^{(1)},\dots,Y^{(N)})^\top$. 
The coefficient $c=(c_1,\cdots,c_L)^\top$ is obtained by solving
\[
(A_N^\top A_N + \lambda I)c = A_N^\top Y_N.
\]

\paragraph{Online update}
Forming the full matrix $A_N$ explicitly requires $\mathcal{O}(NL)$ memory, which becomes prohibitive when $N$ is large. 
To address this issue, we compute the normal matrix and the right-hand side incrementally.
Let the data be processed in batches $\{(A_k,Y_k)\}_{k=1}^K$, where 
$A_k \in \mathbb{R}^{q\times L}$ denotes the local design matrix and 
$Y_k\in\mathbb{R}^q$ the corresponding observations. 

We initialize $G^{(0)}=\lambda I$ and $B^{(0)}=0$. Define
\begin{align*}
G^{(k)} &= G^{(k-1)} + A_k^\top A_k,\\
B^{(k)} &= B^{(k-1)} + A_k^\top,
\end{align*}
for $k=1,\dots,K$. After processing all batches, we obtain
\[
G_N = G^{(K)}, \qquad B_N = B^{(K)} = A_N^\top.
\]

This yields the linear operator
\[
W_N = (G_N)^{-1} B_N,
\]
which defines a mapping from the sampled data to the coefficient vector. 
For any input $f$, the solution is computed as
\[
c_N = W_N f(X_N), 
\qquad
u_{N,L}(x)
=
\sum_{j=1}^L c_j\, P_s^{(1,0)}K(z_j,x).
\]

\begin{algorithm}[H]\label{alg:online-lowrank}
\caption{Online low-rank kernel operator learning}
\begin{algorithmic}[1]
\STATE Initialize $G \leftarrow \lambda I$, $B \leftarrow 0$
\FOR{each batch $A_k$}
    \STATE $G \leftarrow G + A_k^\top A_k$
    \STATE $B \leftarrow B + A_k^\top$
\ENDFOR
\STATE Compute $W \leftarrow G^{-1} B$
\RETURN $W$
\end{algorithmic}
\end{algorithm}

\paragraph{Computational complexity} The online procedure is summarized in Algorithm~\ref{alg:online-lowrank}.
The proposed method only requires storing the matrix $G_N \in \mathbb{R}^{L\times L}$, leading to a memory cost of $\mathcal{O}(L^2)$ instead of $\mathcal{O}(N^2)$. Each batch update costs $\mathcal{O}(q L^2)$, and the overall complexity is $\mathcal{O}(N L^2)$, which is significantly more efficient than the $\mathcal{O}(N^3)$ complexity of the full kernel operator method.

Once the operator $W_N$ is constructed, evaluating the solution for a new input requires only matrix-vector multiplications. This separates the training and evaluation phases and enables efficient reuse of the learned operator.

\section{Additional numerical results}
In this section, we report additional numerical experiments to further evaluate the proposed method. These results complement those presented in the main text and demonstrate the method's stability and generalization properties across different test settings.

In Sections~\ref{Schrodinger Equation} and \ref{Heat equation}, 
we demonstrate the proposed method on time-independent Schr\"odinger equation and heat equation. In Section~\ref{Addition: A Comparison with Neural Operator Learning}, 
we further compare our approach with the operator learning method introduced in \cite{stepaniants2023learning}.

\subsection{Schr\"odinger Equation} \label{Schrodinger Equation}

We study the time-independent Schr\"odinger equation in two spatial dimensions:
\begin{align}\label{Schrodinger Eqn}
-\Delta u + V u = f \quad \text{in } D=(0,1)^2, 
\qquad u = g \quad \text{on } \partial D,
\end{align}
where $V$ is a given potential function. We consider the setting in which the potential $V$ and the boundary data $g$ are fixed, and investigate the associated input--output mapping $f \mapsto u$.

In the numerical experiments, we take $g=0$ and consider the potential 
\begin{align*} 
V(x) = V_0 \mathbf{1}_{A}(x),\quad A = \left\{
\max\!\left(
|x_1-0.5|,\;
\frac{1}{2}|x_1-0.5| + \frac{\sqrt{3}}{2}|x_2-0.5|
\right) \le 0.2
\right\},
\end{align*}
which represents a hexagonal potential well of height $V_0 = 1000$ centered at $(0.5,0.5)$ (see Figure~\ref{fig:schrodinger}(a)).

\begin{figure}[htbp]
\centering

\begin{minipage}[t]{0.48\textwidth}
    \centering
    \includegraphics[width=\linewidth]{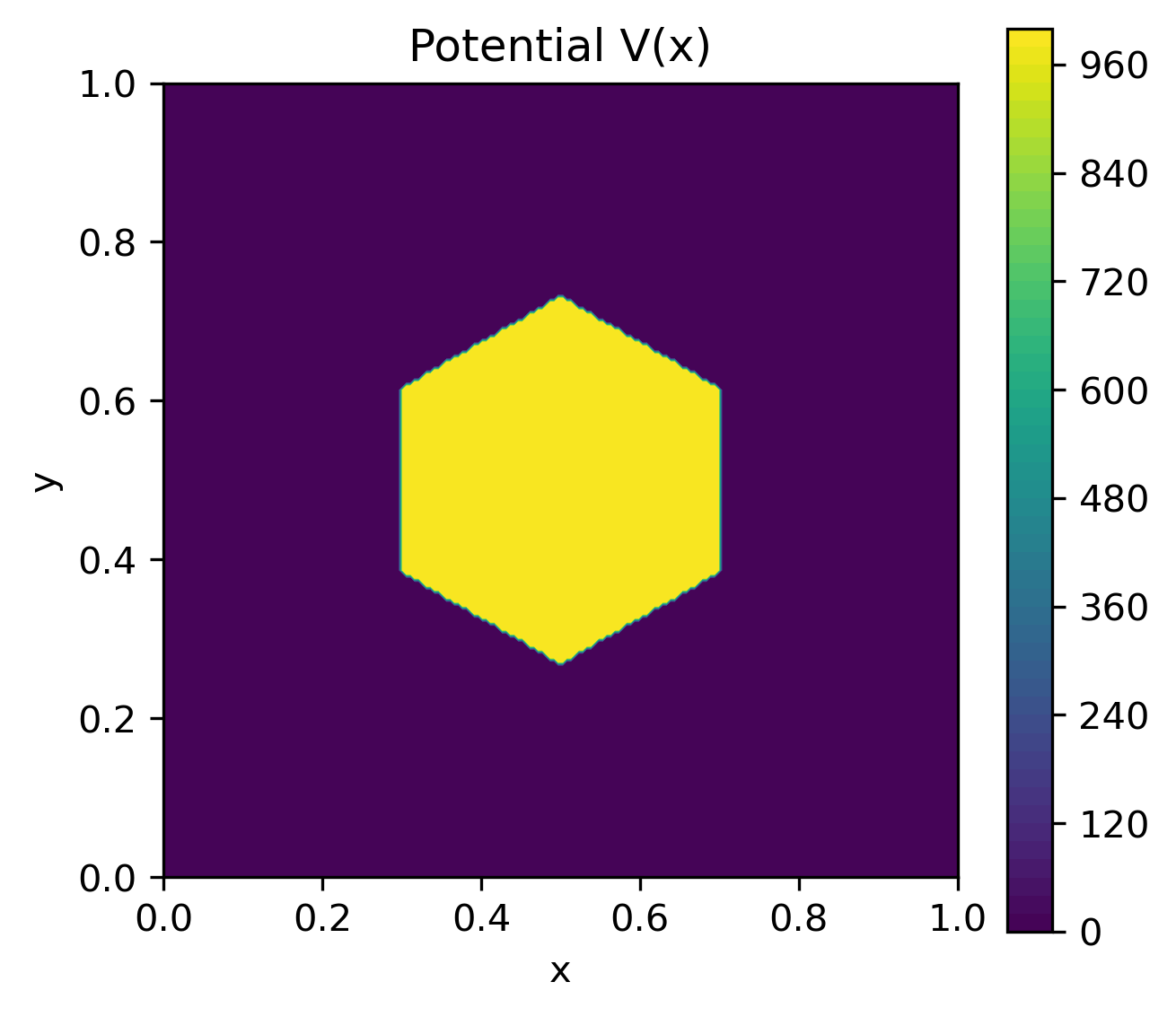}
    \vspace{-2pt}
    {\small (a)}
\end{minipage}
\hfill
\begin{minipage}[t]{0.48\textwidth}
    \centering
    \includegraphics[width=\linewidth]{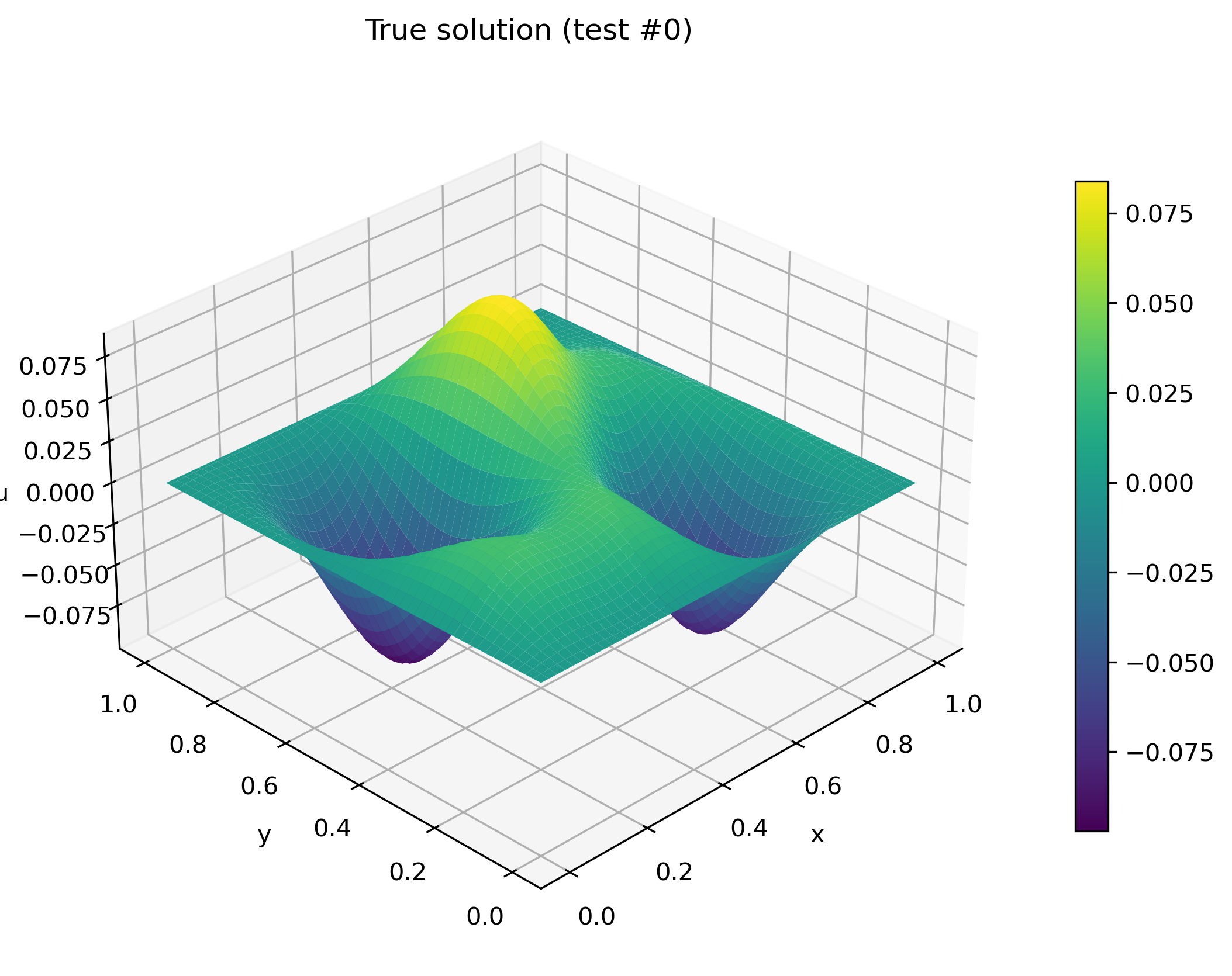}
    \vspace{-2pt}
    {\small (b)}
\end{minipage}

\vspace{6pt}

\begin{minipage}[t]{0.48\textwidth}
    \centering
    \includegraphics[width=\linewidth]{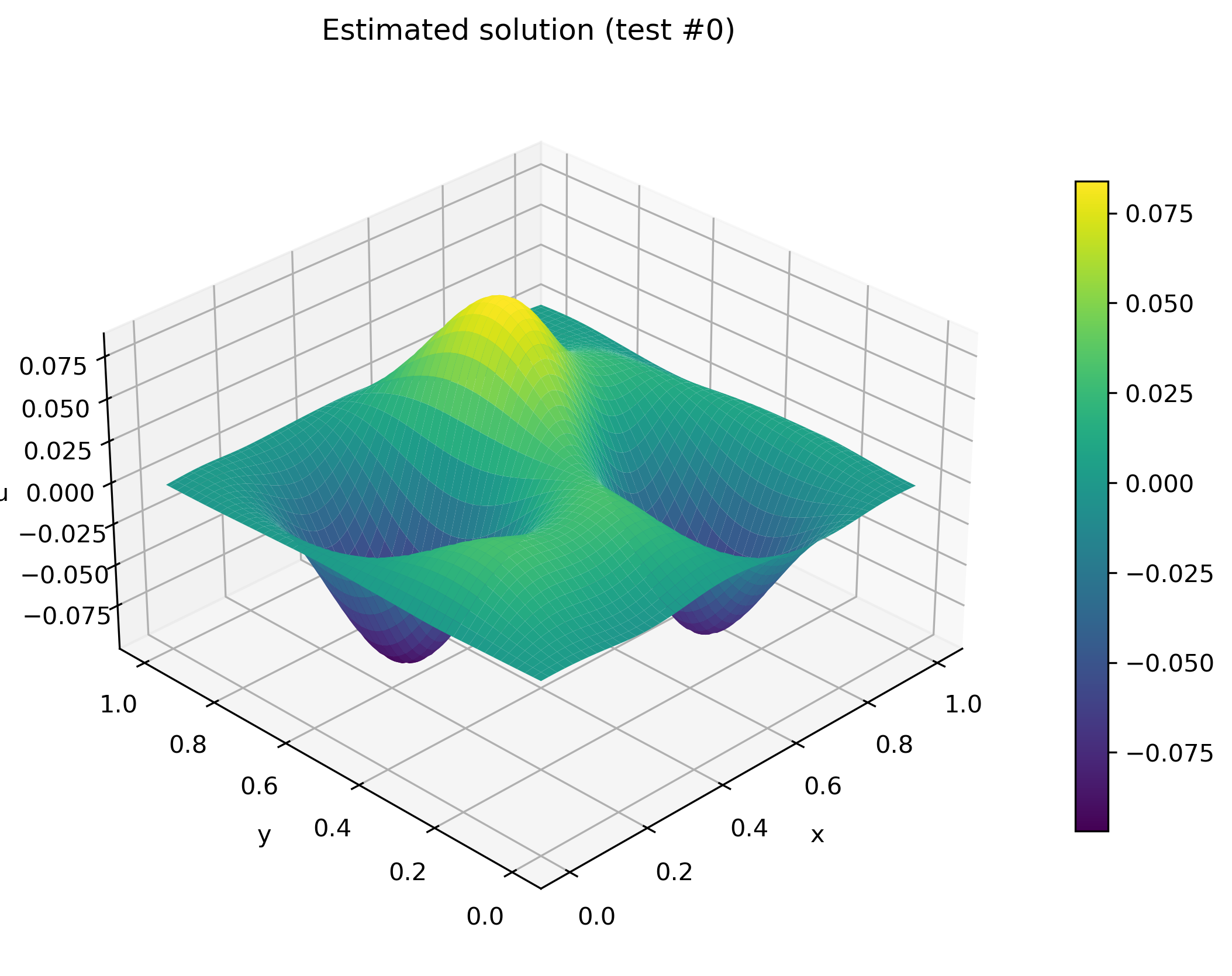}
    \vspace{-2pt}
    {\small (c)}
\end{minipage}
\hfill
\begin{minipage}[t]{0.48\textwidth}
    \centering
    \includegraphics[width=\linewidth]{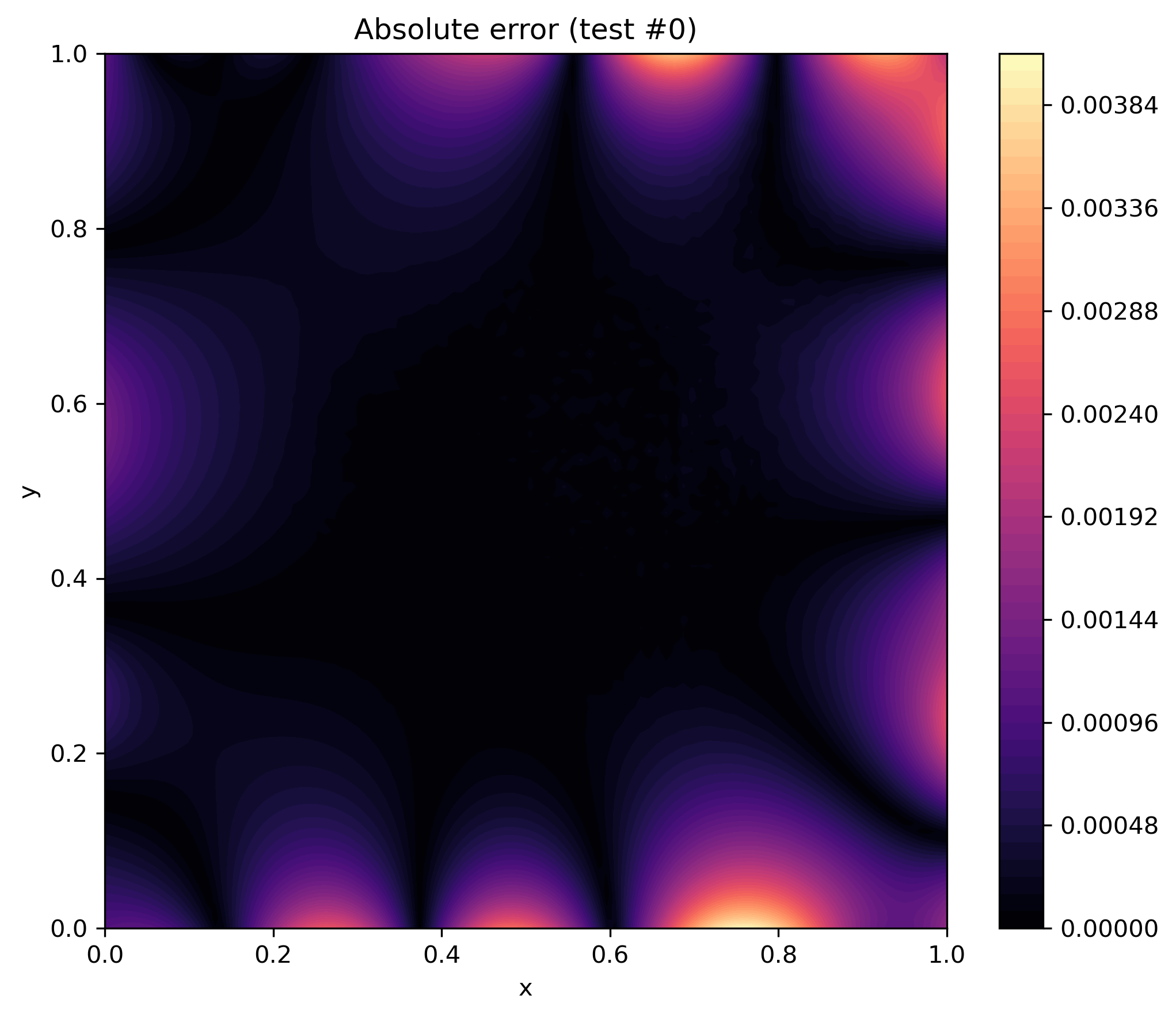}
    \vspace{-2pt}
    {\small (d)}
\end{minipage}

\caption{Schr\"odinger equation. 
(a) Potential $V(x)$; 
(b) true solution; 
(c) estimated solution; 
(d) absolute error.}
\label{fig:schrodinger}
\end{figure}

\noindent{\bf Training and Testing.}\quad We uniformly sample $N_1 = 1200$ points in $D$ and $N_2 = 800$ points on $\partial D$. The Gaussian kernel bandwidth is set to $\eta = 0.08$, and the regularization parameter is chosen as $\lambda = 3 \times 10^{-3}$.
To evaluate the performance of the proposed method, we construct a test family of oscillatory modes with varying frequencies and phases as
\begin{align*}
u_k(x,y)
&=
1.5\, x(1-x)\, y(1-y) \notag \\
&\quad \times \Big[
\sin\!\big(2\pi m x + \phi_1\big)\,
\sin\!\big(2\pi n y + \phi_2\big) \notag \\
&\qquad
+ 0.5\, \cos\!\big(2\pi (m+n)x + \phi_1\big) \notag \\
&\qquad
+ 0.35\, \sin\!\big(2\pi (m x + n y) + \phi_2\big)
\Big],
\end{align*}
where $(i,j)\in\{0,\dots,9\}^2$, $m=(i\bmod 4)+1$, $n=(j\bmod 4)+1$, and $\phi_1=2\pi i/10$, $\phi_2=2\pi j/10$. It can be checked that this family satisfies the boundary condition of the Schr\"odinger equation \eqref{Schrodinger Eqn}. The chosen family spans a set of smooth functions with varying frequency content and non-separable structure, serving as a challenging benchmark 
for operator learning methods.

\begin{table}[htbp]
\centering
\caption{Numerical performance for Sch\"odinger equation}\label{Numerical performance for Schodinger equation}
\vspace{-0.5em}
\begin{tabular}{c|c|c}
\hline
 Cost & Relative $L^2$ error & Relative $L^\infty$ error \\ 
\hline
0.210 s & 4.549$\times 10^{-2}$ & 7.480$\times 10^{-2}$ \\ 
\hline
\end{tabular}
\end{table}

\noindent{\bf Results.}\quad 
The relative $L^2$ and $L^\infty$ errors are reported in Table~\ref{Numerical performance for Schodinger equation}. 
Overall, the proposed method achieves consistently low approximation errors across 
all test cases, demonstrating robustness with respect to variations in the potential 
function $V$.
As expected, the errors remain on the order of $10^{-2}$, indicating stable and accurate performance. 
Moreover, the computational time is consistently below 1 second, highlighting the efficiency of the proposed approach and its favorable scalability with respect to the heterogeneity of the coefficient. 
Figure~\ref{fig:schrodinger} illustrates the performance of the method for one input functions. It shows that the predicted solutions closely match the ground truth, with small and structured errors, demonstrating robustness and accuracy across varying levels of regularity in the data.

\subsection{Heat equation}\label{Heat equation}
We now consider the heat equation on the space-time domain $D=[0,1] \times[0, T]$ with Dirichlet boundary conditions is
\begin{equation}\label{heat equation}
\begin{cases}
\frac{\partial u}{\partial t}-\alpha \frac{\partial^2 u}{\partial x^2}=f(x, t), \quad(x, t) \in[0,1] \times[0, T],\\
u(0, t)=u(1, t)=0,\quad u(x, 0)=0,
\end{cases}    
\end{equation}
where $\alpha=0.1$ is given. We set $T=1$ and investigate the associated input--output mapping $f \mapsto u$. 

\begin{figure}[htbp]
\centering

\begin{minipage}[t]{0.32\textwidth}
    \centering
    \includegraphics[width=\linewidth]{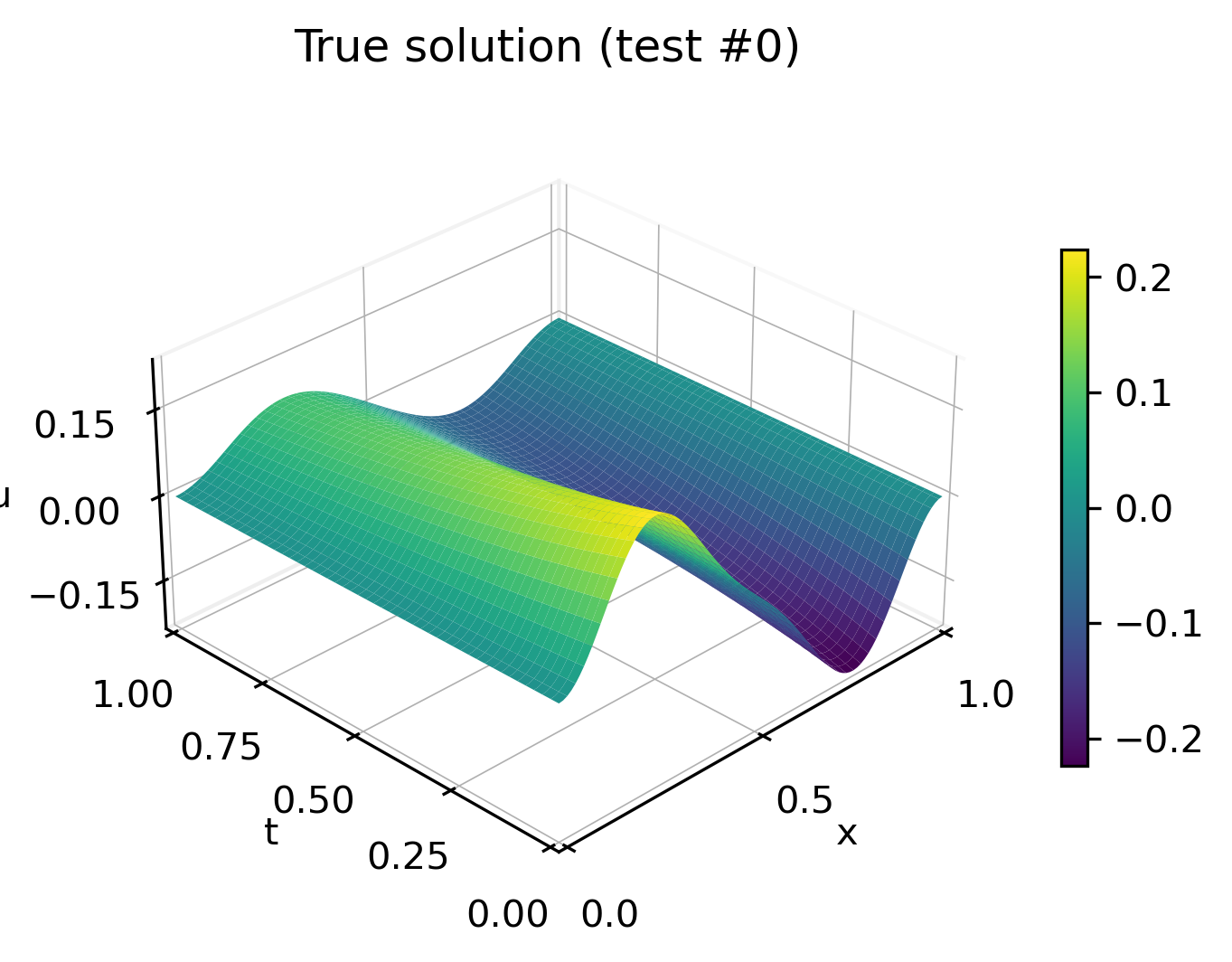}
    \vspace{-2pt}
    {\small (a)}
\end{minipage}
\hfill
\begin{minipage}[t]{0.32\textwidth}
    \centering
    \includegraphics[width=\linewidth]{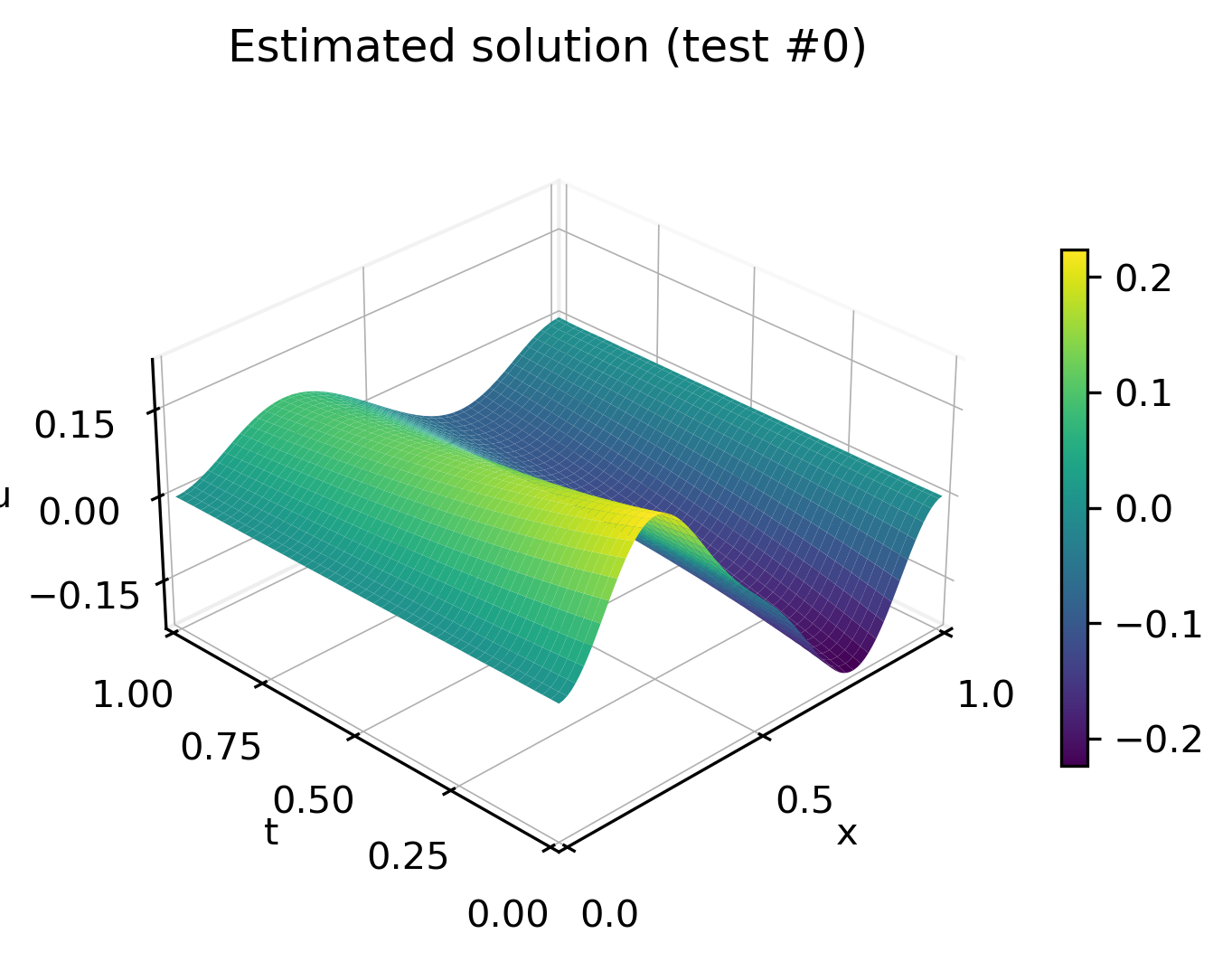}
    \vspace{-2pt}
    {\small (b)}
\end{minipage}
\hfill
\begin{minipage}[t]{0.32\textwidth}
    \centering
    \includegraphics[width=\linewidth]{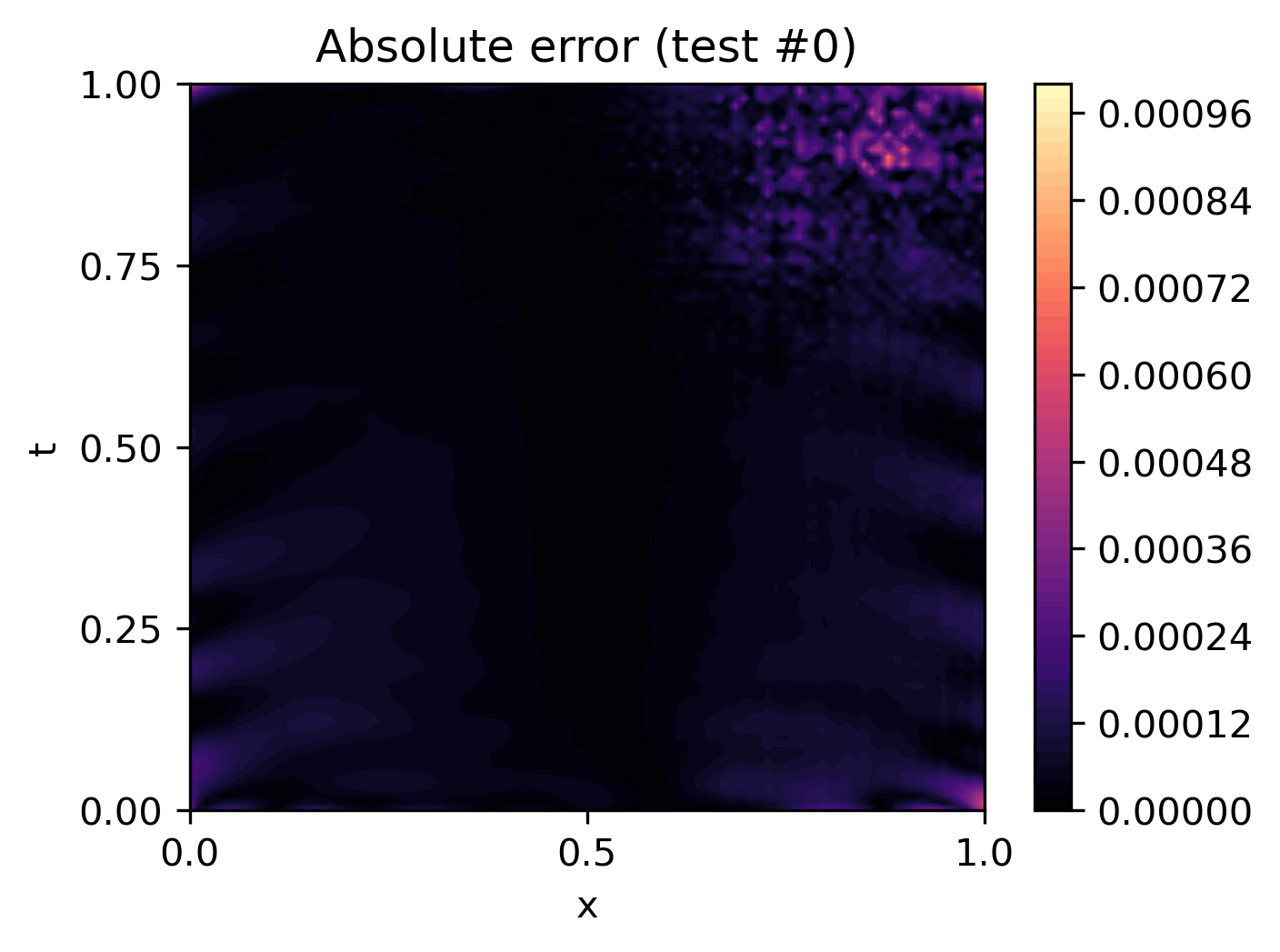}
    \vspace{-2pt}
    {\small (c)}
\end{minipage}

\caption{Heat equation. 
(a) true solution; 
(b) estimated solution; 
(c) absolute error.}
\label{fig:heat}
\end{figure}

\noindent{\bf Training and Testing.}\quad We uniformly sample $N_1 = 1800$ points in $D$ and $N_2 = 1800$ points on $\partial D$ and $N_3=1400$ in $[0,1]$ for initial samples. The Gaussian kernel bandwidth is set to $\eta = 0.08$, and the regularization parameter is chosen as $\lambda = 10^{-5}$.
To evaluate the performance of the proposed method, we construct a test family of oscillatory modes with varying frequencies and phases as
We define a family of test functions
\begin{equation*}
u_k(x,t)
=
1.2\, x(1-x)
\left(
\sin(2\pi m x + \phi_1)\, e^{-\lambda t}
+
0.3\, \sin(2\pi (m+1)x)\, e^{-(\lambda+1)t}
\right),
\end{equation*}
with $m=(i \bmod 4)+1$, $\lambda=(j \bmod 4)+1$, and $\phi_1=2\pi i/10$.
The chosen test functions are designed to (i) satisfy the boundary conditions of the heat equation \eqref{heat equation}, 
(ii) exhibit multi-scale spatial oscillations, and (iii) reflect the intrinsic temporal decay structure of the heat equation. 

\begin{table}[htbp]
\centering
\caption{Numerical performance for heat equation}\label{Numerical performance for heat equation}
\vspace{-0.5em}
\begin{tabular}{c|c|c}
\hline
 Cost & Relative $L^2$ error & Relative $L^\infty$ error \\ 
\hline
3.610 s & 1.874$\times 10^{-3}$ & 4.227$\times 10^{-3}$ \\ 
\hline
\end{tabular}
\end{table}

\noindent{\bf Results.}\quad 
The relative $L^2$ and $L^\infty$ errors are reported in  Table~\ref{Numerical performance for heat equation}. 
Overall, the proposed method achieves consistently low approximation errors across 
all test cases, demonstrating robustness with respect to variations in the source 
term $f$.
As expected, the errors remain on the order of $10^{-3}$, indicating stable and accurate performance. 
Moreover, the computational time is consistently about 3 seconds, highlighting the efficiency of the proposed approach and its favorable scalability with respect to coefficient heterogeneity. 
Figure~\ref{fig:heat} illustrates the performance of the method for one input function. It shows that the predicted solutions closely match the ground truth, with small and structured errors, demonstrating robustness and accuracy across varying levels of regularity in the data.

\subsection{Comparison with Green Operator Learning: Helmholtz Equation}
\label{Addition: A Comparison with Neural Operator Learning}

In addition to the experiments presented in Section~\ref{main-A Comparison with Neural Operator Learning}, we include an additional test case with non-oscillatory functions to further assess the performance of the proposed method. This complements the previous results and provides additional evidence of the robustness and accuracy of the kernel operator approach.

For evaluation, we further construct the test set (non-oscillatory functions) of size $M = 100$:
\[
u_k(x) = -0.1 + 0.2x + x(1-x)\,NN_k(x),
\]
where $NN_k$ are randomly generated neural networks.

\begin{table}[htbp]
\centering
\caption{Comparison of numerical performance for the Helmholtz equation.}
\label{Comparison2}
\begin{tabular}{c c l c c}
\hline
Case & Method & Cost & Relative $L^2$ error & Relative $L^\infty$ error \\ 
\hline
$\omega=20$ & Ours & 1.533 {\rm s} & $1.068\times 10^{-3}$ & $2.354\times 10^{-3}$ \\ 
       & Green operator  & 0.721 {\rm h} & $2.547 \times 10^{-2}$ & $9.509\times 10^{-2}$ \\ 
\hline
$\omega=200$ & Ours & 0.380 {\rm s} & $9.929\times 10^{-4}$ & $1.585\times 10^{-3}$ \\ 
       & Green operator  & 0.745 {\rm h} & $1.609\times 10^{-1}$ & $7.364\times 10^{-1}$ \\ 
\hline
\end{tabular}
\end{table}

\noindent{\bf Results.}\quad 
Table~\ref{Comparison2} reports averaged relative $L^2$ and $L^\infty$ errors on the additional test set of non-oscillatory functions, and Figure~\ref{w20} shows representative examples.
The results are consistent with those in Section \ref{main-A Comparison with Neural Operator Learning}. The proposed method maintains stable accuracy across both low- and high-frequency regimes, whereas the Green operator (GO) approach deteriorates as the frequency increases. 
As shown in Figure~\ref{w20}, the dominant error arises from the approximation near the boundary. In the Green operator method, the boundary condition is not imposed as a hard constraint, whereas in the proposed kernel operator method it is incorporated into the loss function. This leads to improved accuracy of the proposed method near the boundary.
These results show that the performance gap is not limited to oscillatory inputs, but also persists for smooth test functions, further supporting the robustness of the proposed method.

\begin{figure}[htbp]
\centering

\begin{minipage}[t]{0.48\textwidth}
    \centering
    \includegraphics[width=\linewidth]{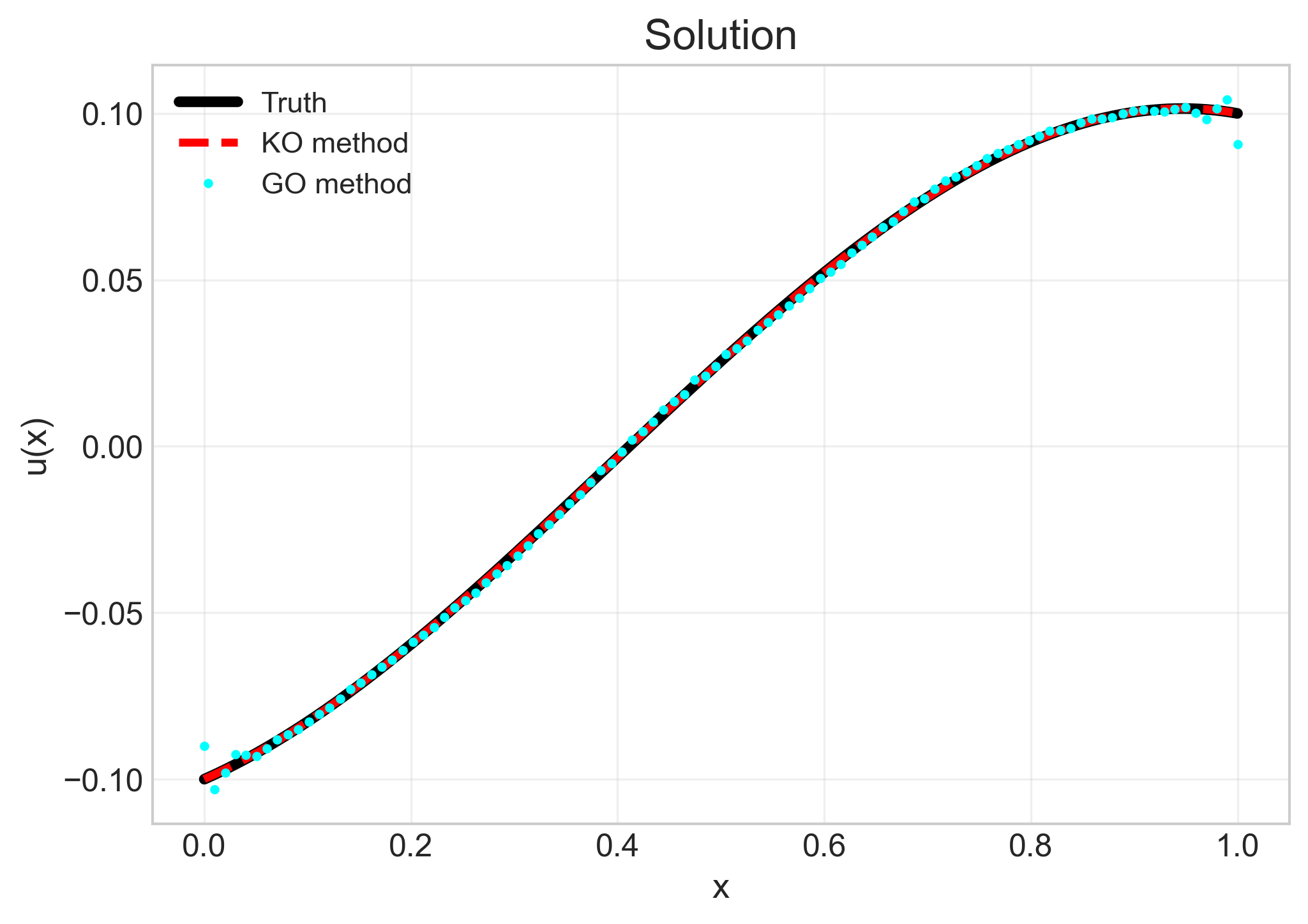}
    \vspace{-2pt}
    {\small (a) $\omega=20$}
\end{minipage}
\hfill
\begin{minipage}[t]{0.48\textwidth}
    \centering
    \includegraphics[width=\linewidth]{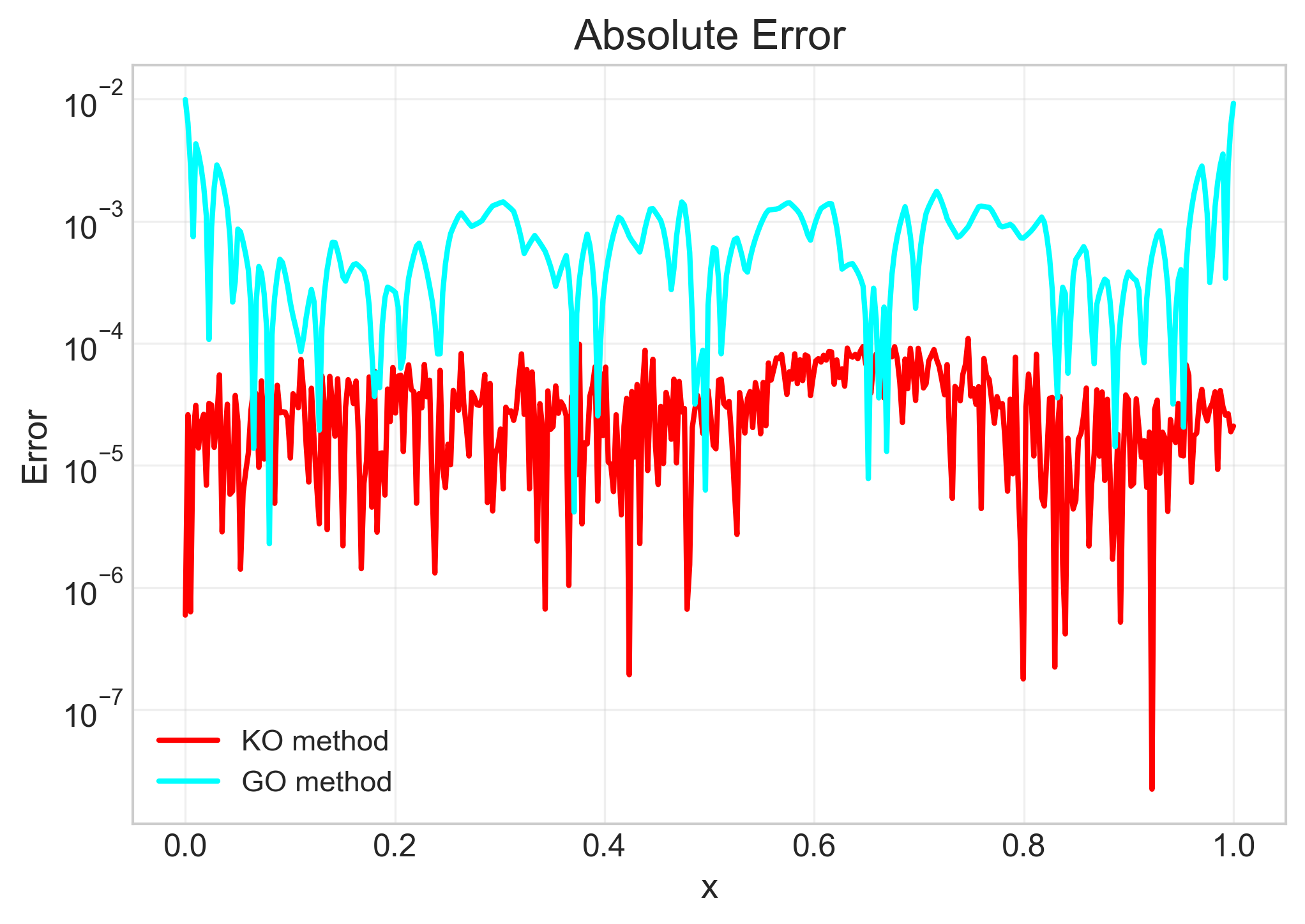}
    \vspace{-2pt}
    {\small (b) $\omega=20$}
\end{minipage}

\vspace{6pt}

\begin{minipage}[t]{0.48\textwidth}
    \centering
    \includegraphics[width=\linewidth]{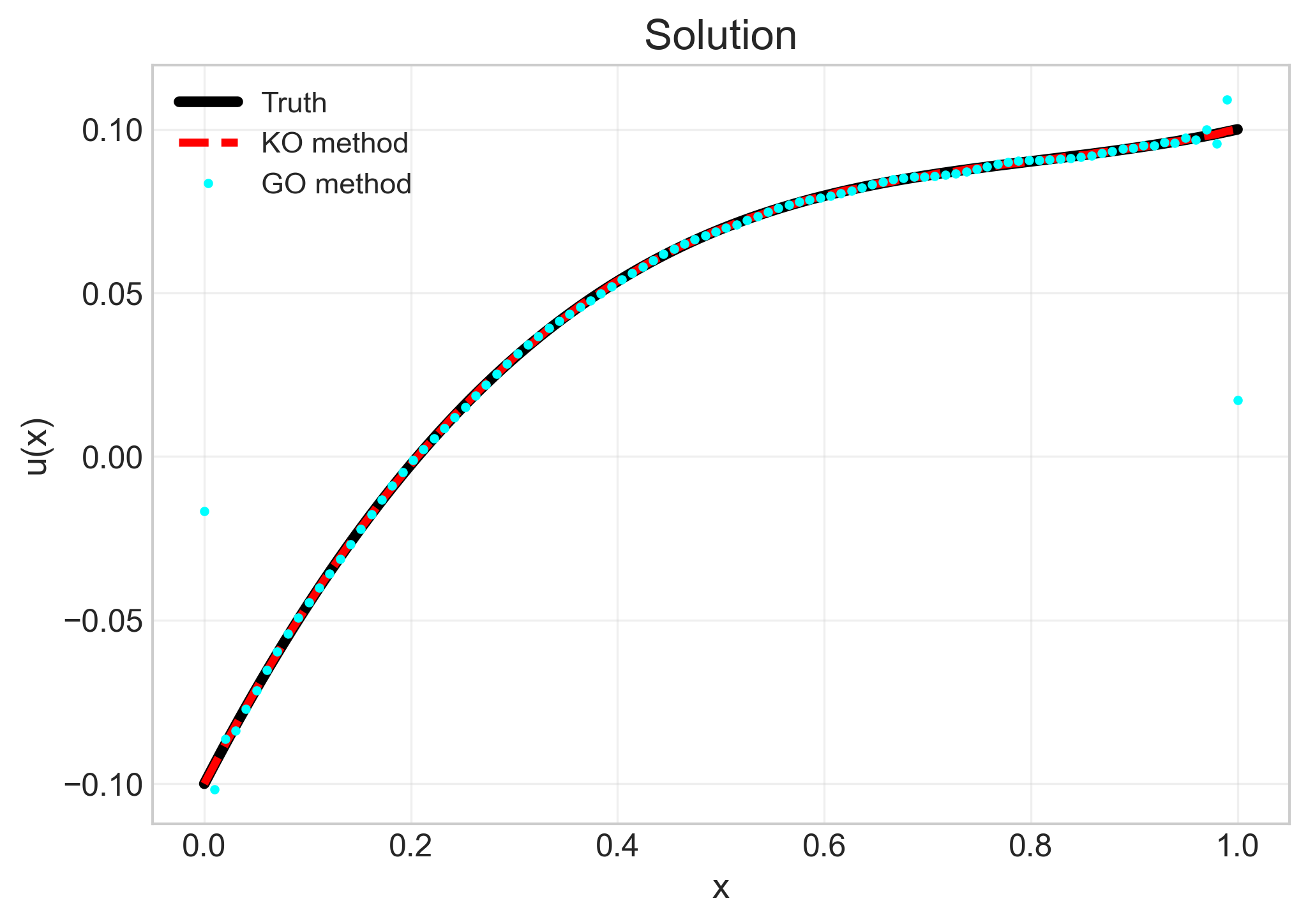}
    \vspace{-2pt}
    {\small (c) $\omega=200$}
\end{minipage}
\hfill
\begin{minipage}[t]{0.48\textwidth}
    \centering
    \includegraphics[width=\linewidth]{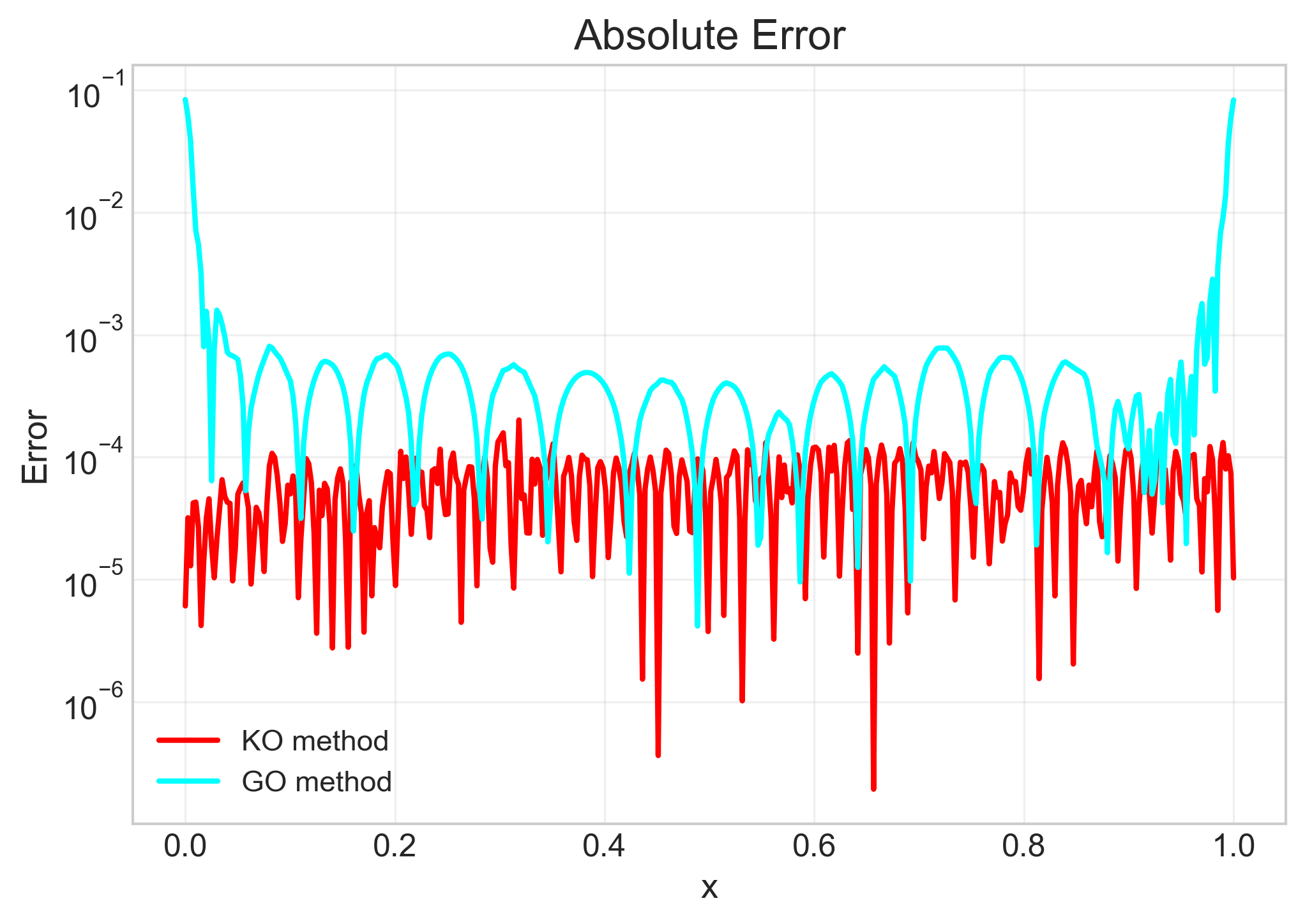}
    \vspace{-2pt}
    {\small (d) $\omega=200$}
\end{minipage}

\caption{Helmholtz equation.
(a,c) True solutions with Kernel Operator (KO) and Green Operator (GO) approximations for $\omega=20$ and $\omega=200$, respectively;
(b,d) corresponding absolute errors.}
\label{w20}
\end{figure}

\subsection{Error analysis}

We investigate the empirical convergence behavior of the proposed low-rank online kernel method for the Poisson problem
\begin{equation}\label{eq:poisson-exp}
\begin{cases}
-\Delta u(x)=f(x), & x\in D,\\
u(x)=0, & x\in \partial D,
\end{cases}
\end{equation}
posed on the domain \(D=(0,1)^3\). As a test family, we consider the Laplacian eigenfunctions
\[
u_k(x)
=
\prod_{\ell=1}^3 \sin(\pi k_\ell x_\ell),
\qquad 
k\in\{1,\dots,k_{\max}\}^3,
\]
with corresponding right-hand side
\[
f_k(x)=\pi^2 |k|^2 u_k(x).
\]
We take \(k_{\max}=4\) and use \(M=50\) test samples.

We emphasize that the primary objective of this experiment is to investigate the asymptotic convergence behavior. 
Therefore, relatively large sample sizes are considered in order to access the regime in which the statistical error dominates 
and the empirical convergence rate becomes observable. To study the dependence on the sample size, we consider
\[
N=
10^4\times
\{1,2,3,5,8,12,16,20,25,30\},
\]
and independently generate \(20\) training datasets for each value of \(N\).

To enforce the homogeneous Dirichlet boundary condition, we adopt a boundary-adapted representation
\[
u_{N,L}(x)
=
\rho(x)\sum_{j=1}^L c_j K_{\eta}(z_j,x),
\qquad
\rho(x)=\prod_{\ell=1}^3 x_\ell(1-x_\ell),
\]
which guarantees \(u_{N,L}|_{\partial D}=0\) by construction. 

In the experiments, we set the number of kernel centers to \(L=5000\), the batch size to \(q=2000\), the kernel bandwidth to \(\eta=0.2\). We consider both a fixed regularization parameter $\lambda = 10^{-8}$ and a data-adaptive choice $\lambda = 10^{-7} *N^{-0.4}$.
We evaluate both the relative \(L^2\) and \(L^\infty\) errors defined in \eqref{main-relative error}. 
The reported results are averaged over the sampled modes and multiple independent realizations of the training data.

\textbf{Results.}\quad Figures~\ref{fig: convergence rate}--\ref{fig: convergence rate adaptive} illustrate the convergence behavior of the proposed online low-rank kernel solver for the Poisson equation using fixed and data-adaptive regularization parameters, respectively. 
In both figures, panels~(a)--(b) display the relative \(L^2\) and \(L^\infty\) errors on a linear scale, while panels~(c)--(d) show the corresponding log-scale plots together with least-squares fitted reference slopes. 

In both cases, the relative \(L^2\) and \(L^\infty\) errors decrease steadily as the sample size \(N\) increases, indicating stable convergence of the estimator. 
The data-adaptive regularization parameter exhibits slightly improved numerical performance compared with the fixed choice. 
Moreover, the standard deviations remain small across all sample sizes, demonstrating robustness with respect to random sampling.

The fitted slopes indicate an approximately algebraic convergence rate of the form
\[
\mathrm{Err}(N)\sim N^{-\beta}.
\]
For the fixed regularization parameter $\lambda=10^{-8}$, the empirical exponents are approximately \(\beta\approx0.38\) and \(\beta\approx0.40\) for the relative \(L^2\) and \(L^\infty\) errors, respectively. 
For the data-adaptive regularization parameter $\lambda = 10^{-7} *N^{-0.4}$, the observed exponents improve to approximately \(\beta\approx0.52\) and \(\beta\approx0.56\).
These observations are consistent with the theoretical convergence behavior predicted by the statistical error analysis.

\begin{figure}[htbp]\label{fig: convergence rate}
\centering
\begin{minipage}[t]{0.48\textwidth}
    \centering   \includegraphics[width=\linewidth]{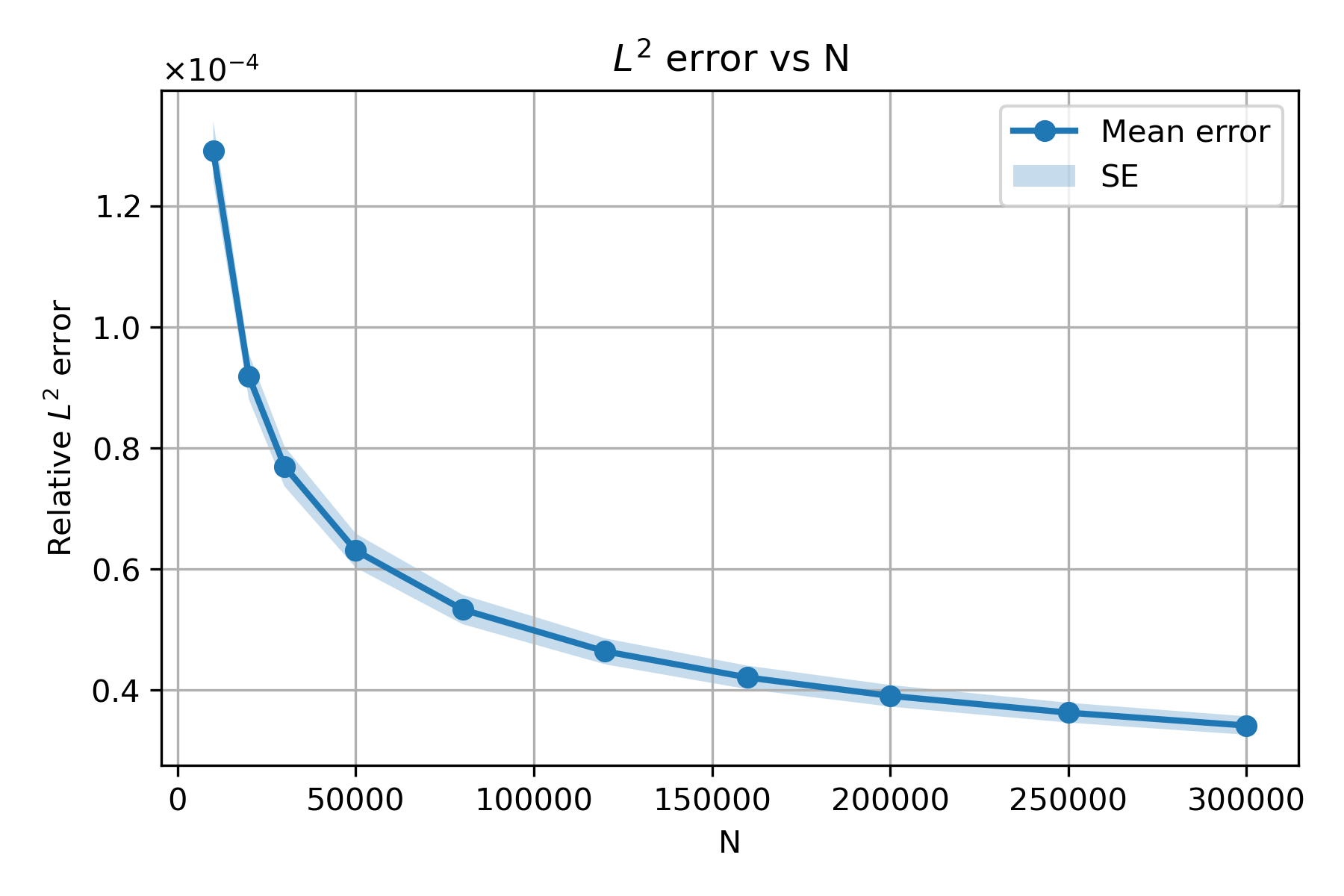}
    \vspace{-2pt}
    {\small (a) }
\end{minipage}
\hfill
\begin{minipage}[t]{0.48\textwidth}
    \centering
    \includegraphics[width=\linewidth]{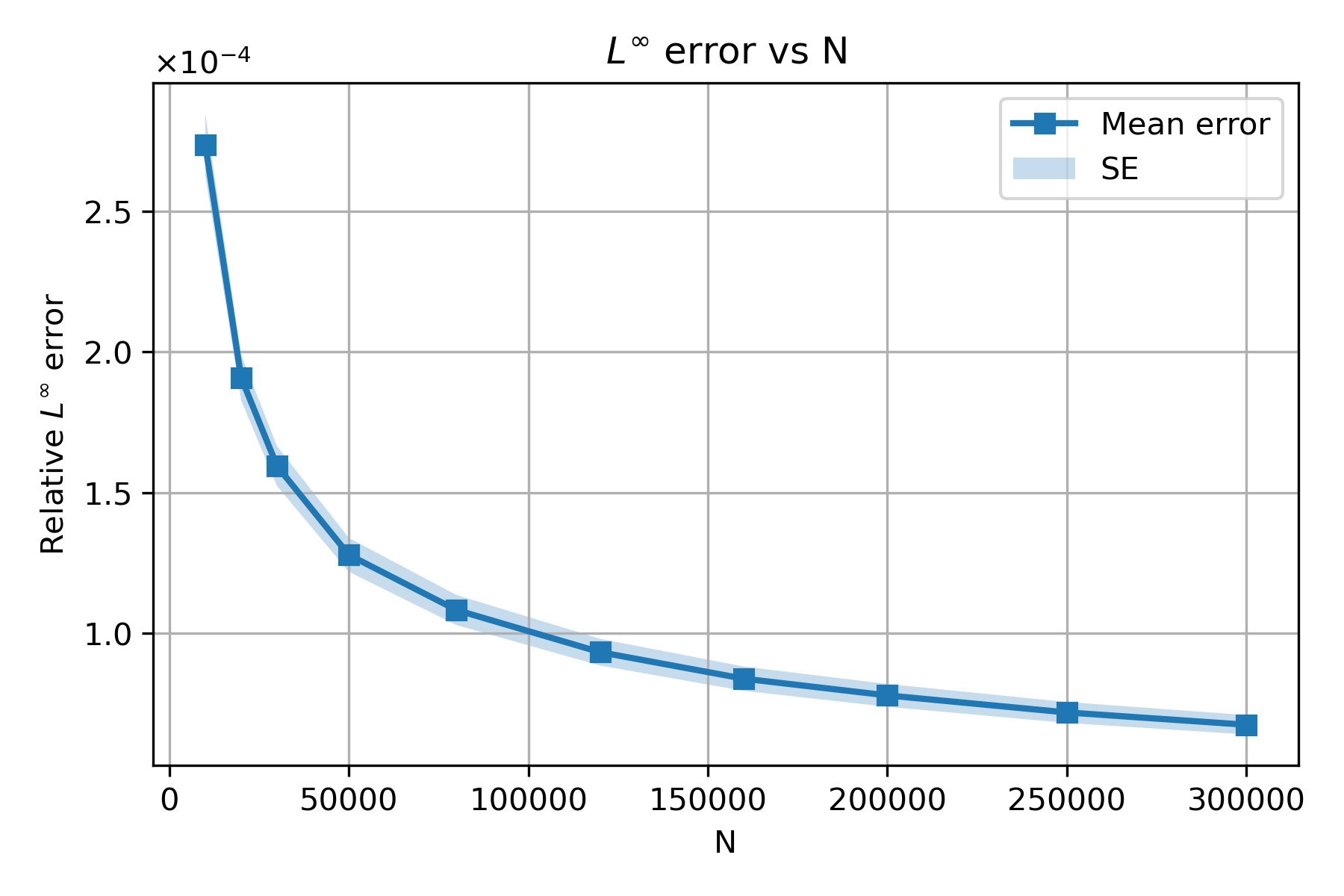}
    \vspace{-2pt}
    {\small (b) }
\end{minipage}
\vspace{6pt}

\begin{minipage}[t]{0.48\textwidth}
    \centering
    \includegraphics[width=\linewidth]{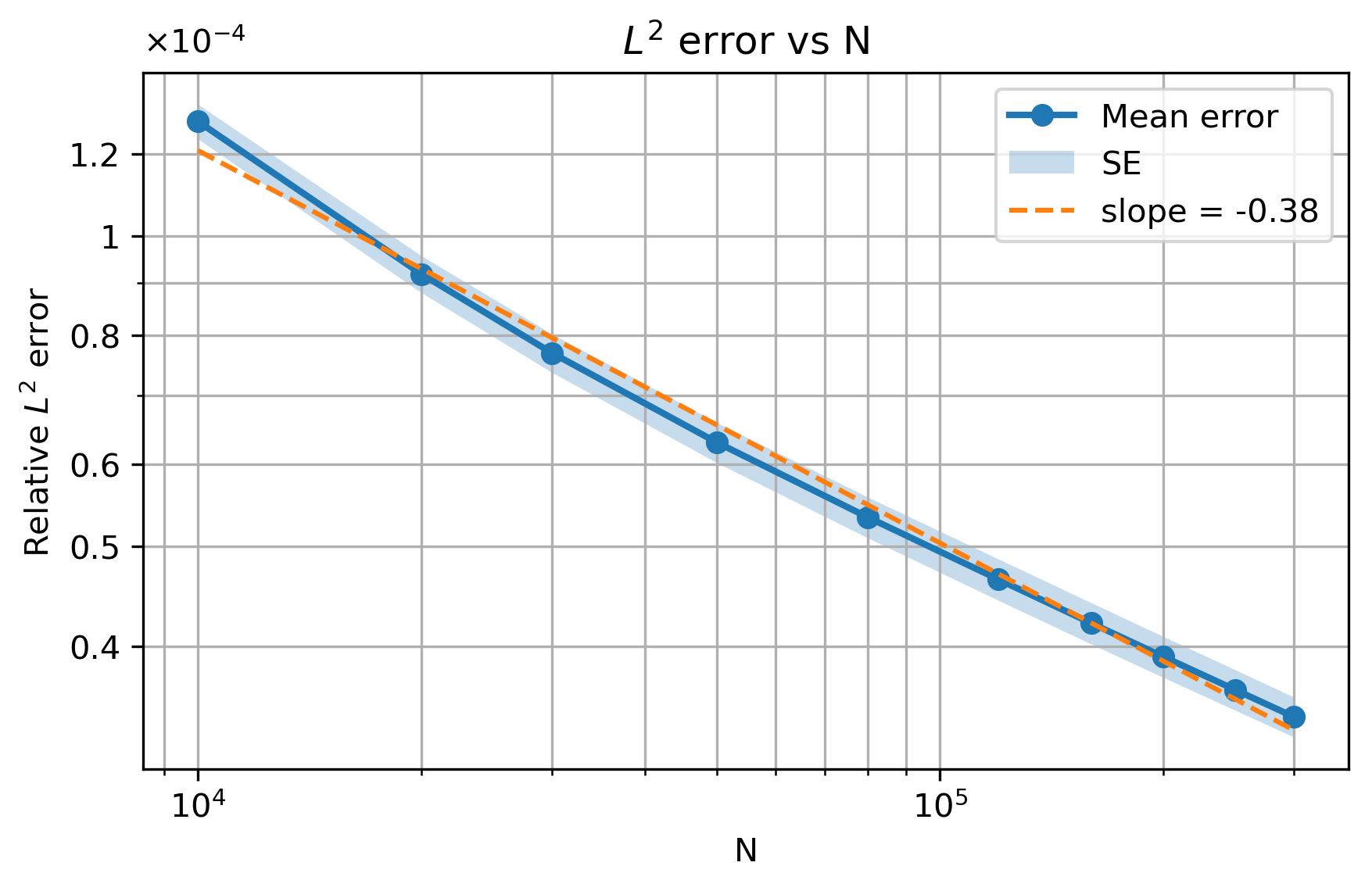}
    \vspace{-2pt}
    {\small (c) }
\end{minipage}
\hfill
\begin{minipage}[t]{0.48\textwidth}
    \centering
    \includegraphics[width=\linewidth]{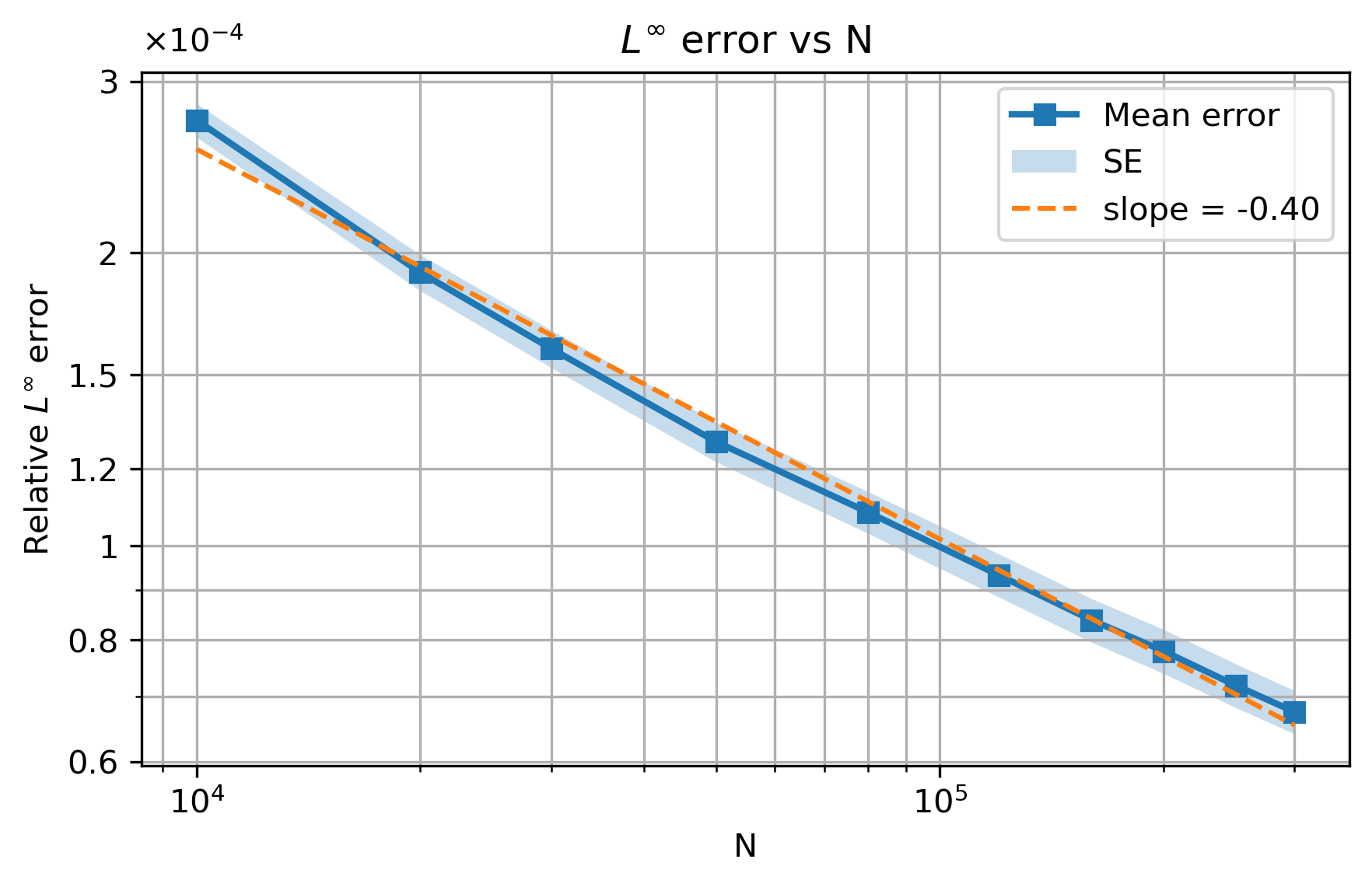}
    \vspace{-2pt}
    {\small (d) }
\end{minipage}

\caption{Fixed $\lambda=10^{-8}$.
(a,b) Mean-SE convergence of the relative $L^2$ and $L^\infty$ errors versus sample size $N$ (b) Mean-SE convergence of the relative $L^2$ and $L^\infty$ errors in the log-scale.}
\end{figure}

\begin{figure}[htbp]\label{fig: convergence rate adaptive}
\centering
\begin{minipage}[t]{0.48\textwidth}
    \centering   \includegraphics[width=\linewidth]{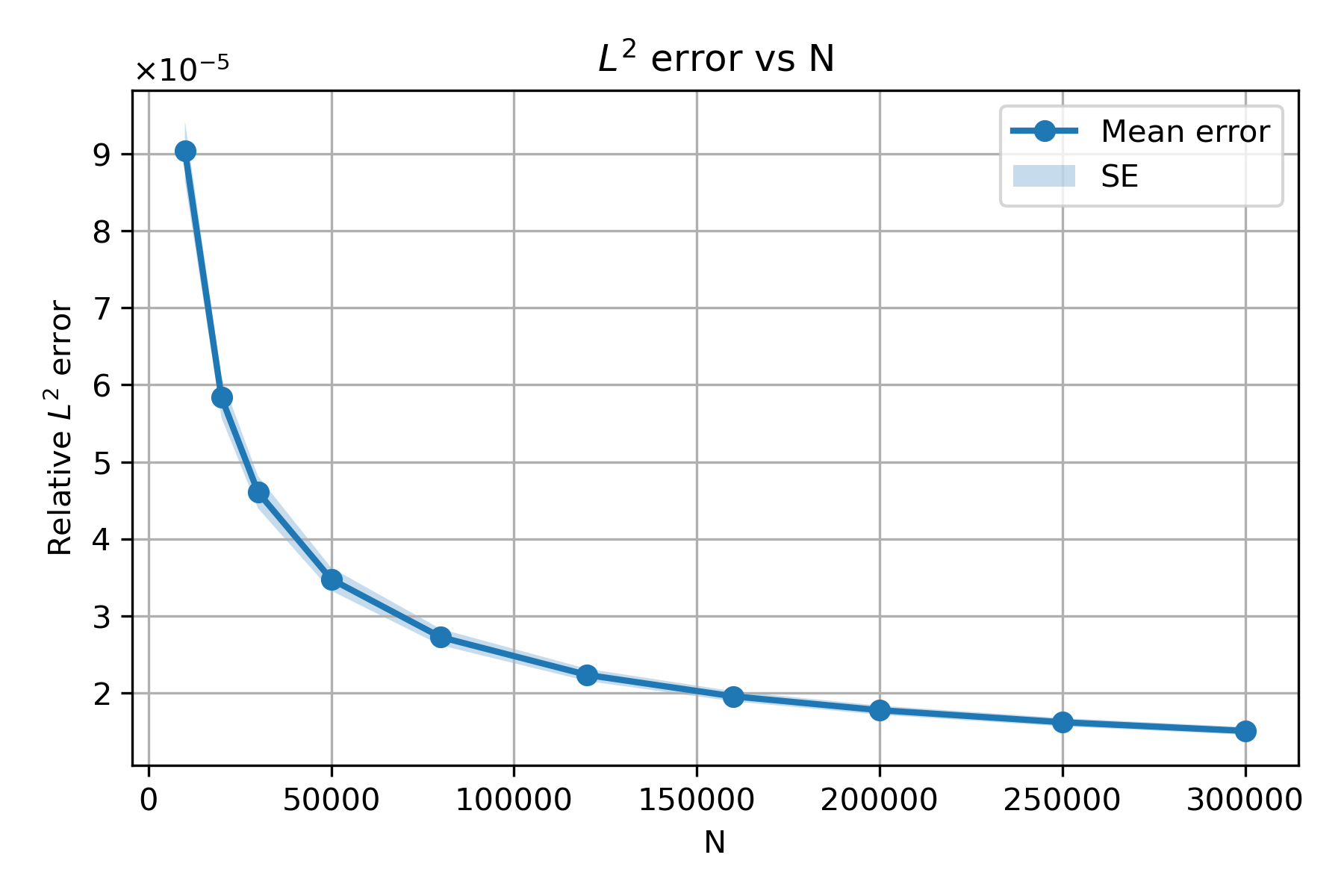}
    \vspace{-2pt}
    {\small (a) }
\end{minipage}
\hfill
\begin{minipage}[t]{0.48\textwidth}
    \centering
    \includegraphics[width=\linewidth]{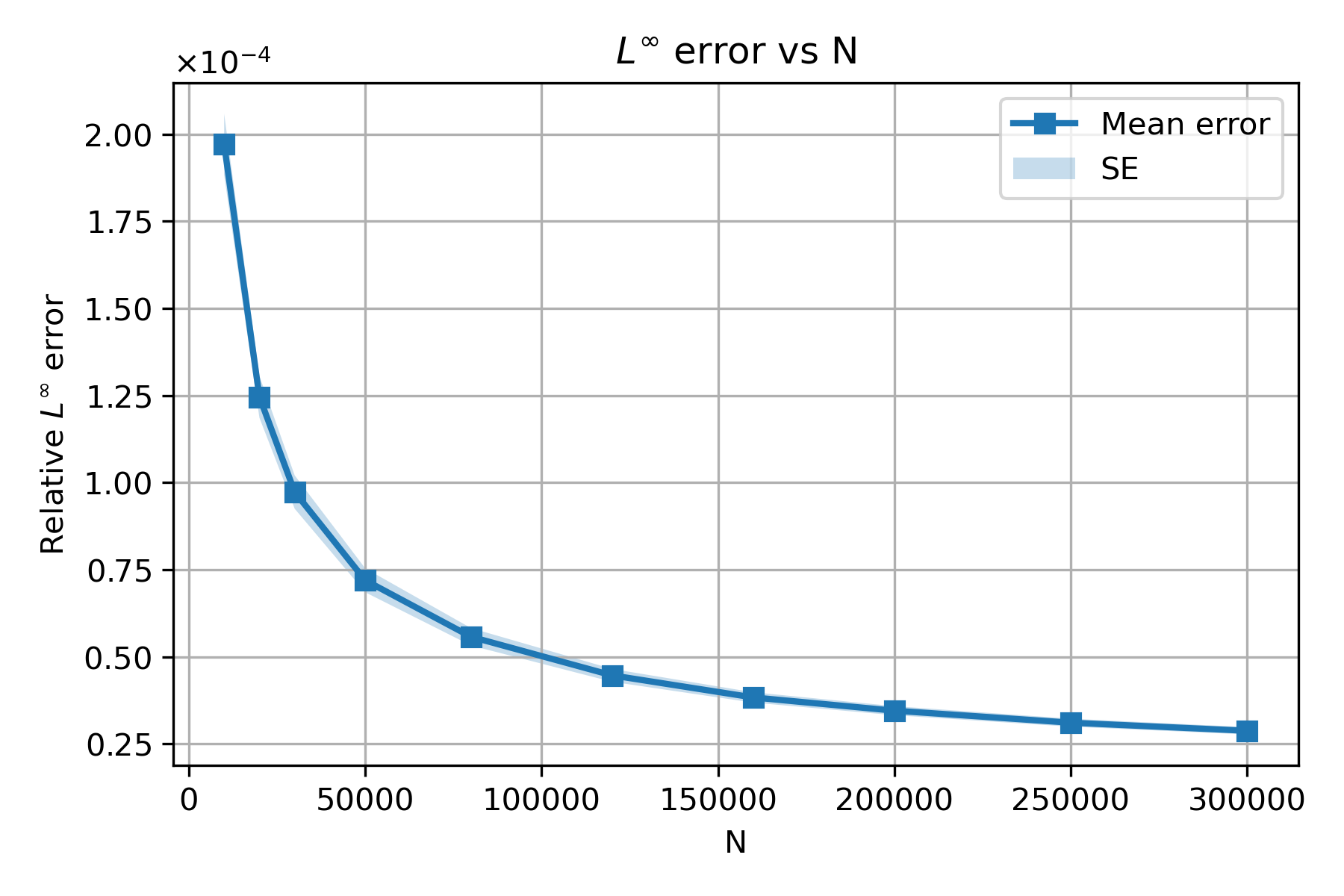}
    \vspace{-2pt}
    {\small (b) }
\end{minipage}
\vspace{6pt}

\begin{minipage}[t]{0.48\textwidth}
    \centering
    \includegraphics[width=\linewidth]{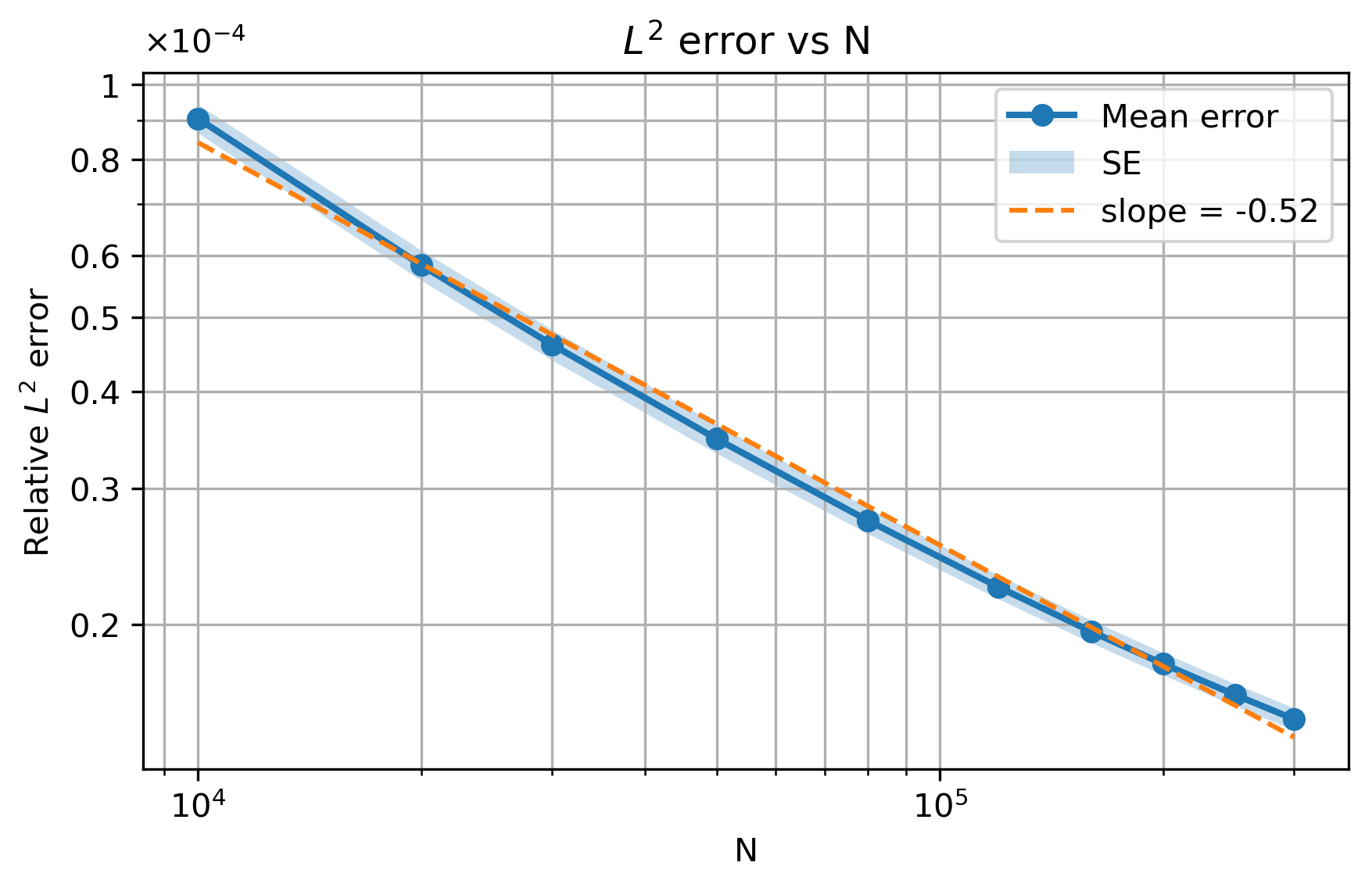}
    \vspace{-2pt}
    {\small (c) }
\end{minipage}
\hfill
\begin{minipage}[t]{0.48\textwidth}
    \centering
    \includegraphics[width=\linewidth]{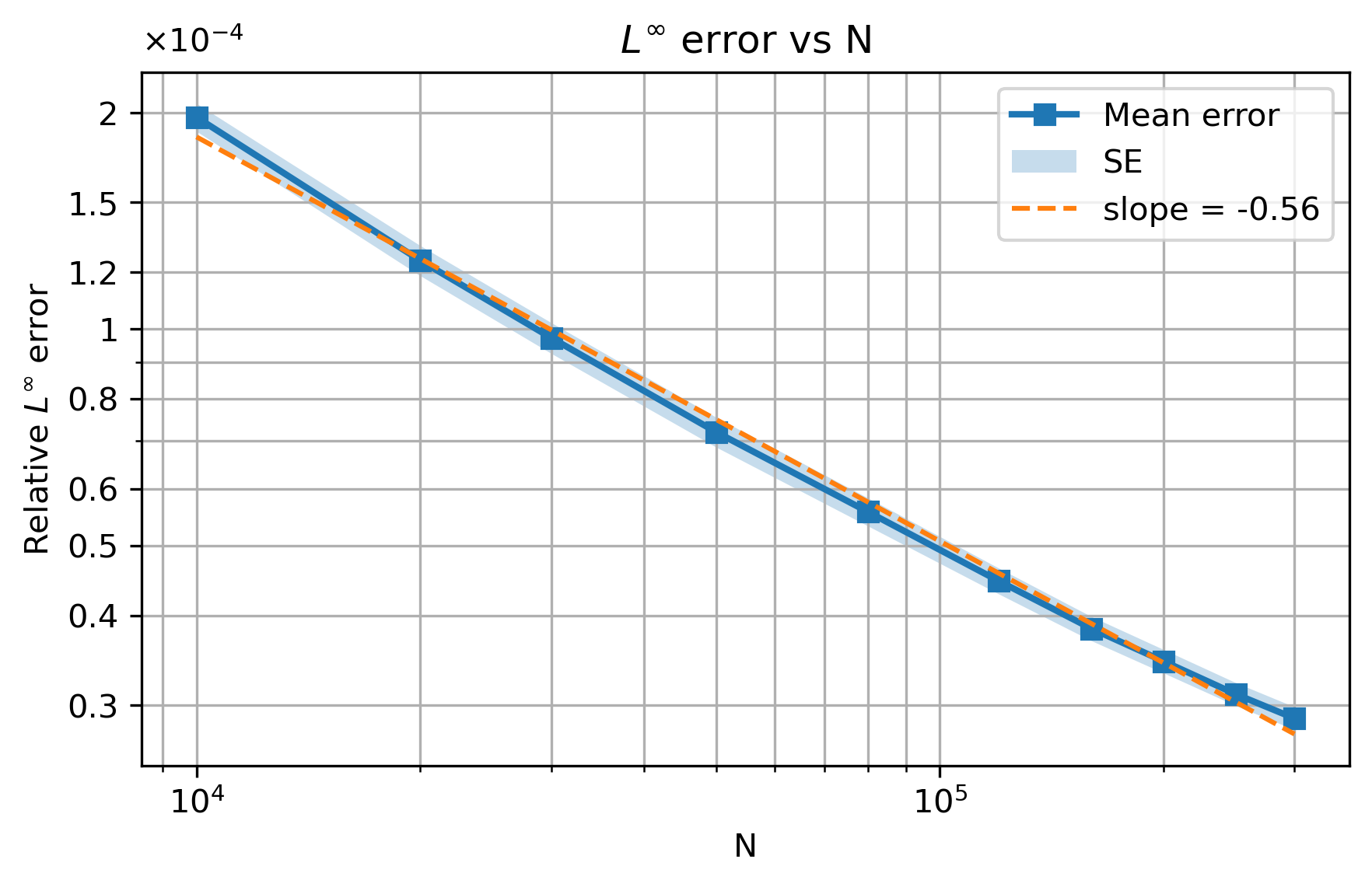}
    \vspace{-2pt}
    {\small (d) }
\end{minipage}

\caption{Data-adaptive $\lambda=10^{-7}*N^{-0.4}$.
(a,b) Mean-SE convergence of the relative $L^2$ and $L^\infty$ errors versus sample size $N$ (b) Mean-SE convergence of the relative $L^2$ and $L^\infty$ errors in the log-scale.}
\end{figure}

\bibliographystyle{siamplain}
\bibliography{references}